\documentclass[11pt,reqno]{amsart}

\usepackage[a4paper]{geometry}
\usepackage{amsmath,amssymb,amsfonts,mathrsfs,mathtools}
\usepackage{enumitem}
\usepackage{booktabs}
\usepackage{array}
\usepackage{tabularx}
\usepackage{float}
\usepackage{microtype}
\usepackage{xcolor}
\usepackage{tikz}
\usepackage[colorlinks=true,allcolors=blue]{hyperref}

\numberwithin{equation}{section}

\allowdisplaybreaks
\hypersetup{
  hypertexnames=false,
  pdftitle={Exact High-Order GPT Cancellation in Noncircular Multilayers},
  pdfauthor={Youjun Deng, Hongyu Liu, Longyue Tao},
  pdfsubject={Generalized polarization tensor cancellation in two-dimensional conductivity multilayers},
  pdfkeywords={generalized polarization tensors, GPT-vanishing structures, noncircular multilayers, near-cloaking}
}

\definecolor{outerblue}{RGB}{63,118,163}
\definecolor{middleblue}{RGB}{137,185,213}
\definecolor{innergold}{RGB}{232,188,92}
\definecolor{coregray}{RGB}{95,98,104}
\definecolor{softgray}{RGB}{245,246,247}

\newtheorem{theorem}{Theorem}[section]
\newtheorem{proposition}[theorem]{Proposition}
\newtheorem{corollary}[theorem]{Corollary}
\newtheorem{lemma}{Lemma}[section]
\theoremstyle{definition}
\newtheorem{definition}{Definition}[section]

\theoremstyle{remark}
\newtheorem{remark}{Remark}[section]

\title[Exact high-order GPT cancellation]
{Exact High-Order GPT Cancellation in Noncircular Multilayers}

\author{Youjun Deng}
\address[Youjun Deng]{School of Mathematics and Statistics, Central South University, Changsha, Hunan 410083, China}
\email{youjundeng@csu.edu.cn}

\author{Hongyu Liu}
\address[Hongyu Liu]{Department of Mathematics, City University of Hong Kong, Hong Kong SAR, China}
\email{hongyu.liuip@gmail.com, hongyliu@cityu.edu.hk}

\author{Longyue Tao}
\address[Longyue Tao]{Department of Mathematics, City University of Hong Kong, Hong Kong SAR, China}
\email{sdyctly@163.com, longyue.tao@my.cityu.edu.hk}

\date{\today}

\begin{document}
\raggedbottom

\begin{abstract}
This paper is concerned with exact high-order generalized polarization tensor (GPT) cancellation in two-dimensional noncircular multilayers with finite positive isotropic coating conductivities.
We first prove that, for every admissible cyclically symmetric core--shell geometry, there exists a unique positive coating conductivity for which the structure is weakly neutral to every uniform incident field.
For higher orders, we establish locally unique coating conductivities that cancel the complete \(K\)-contracted GPTs (K-CGPTs) block on a broad class of fixed noncircular multilayers with cyclic symmetry.
This includes homothetic multilayers generated by strictly star-shaped symmetric domains, without assuming that the geometry is close to concentric disks.
We also replace the near-background condition by a small-core condition covering finite, insulating, and perfectly conducting cores.
By varying selected interface modes together with the coating conductivities near a radial GPT-vanishing structure, we further construct exact families with neither rotational nor reflection symmetry.
For three coatings, the relevant nondegeneracy conditions and an explicit chiral example are rigorously certified.
Finally, we show that the resulting cancellation yields enhanced near-cloaking for every entire harmonic incident field.
A set of certified numerical examples fully corroborates the theoretical predictions established in this work.
\end{abstract}

\maketitle
\enlargethispage{3pt}

\section{Introduction}

Let \(D\Subset\mathbb{R}^2\) be the bounded region occupied by a multilayer, and let \(\sigma\in L^\infty(\mathbb{R}^2)\) be a scalar piecewise constant conductivity satisfying \(0<c_*\leq\sigma(x)\leq c^*<\infty\) almost everywhere in \(\mathbb{R}^2\) and \(\sigma=1\) almost everywhere in \(\mathbb{R}^2\setminus\overline D\).
For an entire harmonic incident field \(H\), the associated total potential \(u\in H^1_{\rm loc}(\mathbb{R}^2)\) solves
\begin{equation}\label{eq:conductivity-transmission-problem}
 \left\{
 \begin{array}{@{}l@{\qquad}l@{}}
  \nabla\!\cdot(\sigma\nabla u)=0
  &\text{in }\mathbb{R}^2,\\
  (u-H)(x)=O(|x|^{-1})
  &\text{as }|x|\to\infty.
 \end{array}
 \right.
\end{equation}
For the homogeneous coefficient \(\sigma\equiv1\), uniqueness in \eqref{eq:conductivity-transmission-problem} yields \(u=H\).
For \(\sigma\not\equiv1\), define the scattered potential \(w:=u-H\); the associated exterior response is \(w|_{\mathbb{R}^2\setminus\overline D}\).
Multilayered piecewise constant conductivities arise naturally in geophysical imaging, composite materials, and anatomical models, and have been studied in electrical impedance tomography (EIT) in connection with reconstruction, stability, and anomaly identification \cite{borcea1999multiscattering,deng2024identifying,garde2020reconstruction,kong2024inverse,li2022refined}.
Related algebraic response-matrix frameworks have been developed for multilayer electrostatic and electro-osmotic systems, with applications to plasmonic resonance and hydrodynamic cloaking \cite{FangdengMMA23,FDLMMA15,kong2024enlargement,kong2024electro}.
Perfect cloaking requires \(w=0\) in the exterior for every admissible \(H\), while near-cloaking requires the exterior response to be uniformly small; see \cite{greenleaf2003nonuniqueness,kohn2008cloaking} for  EIT and \cite{bao2014nearly,kohn2010cloaking,li2015regularized} for related electromagnetic and Helmholtz models.
This common objective can be pursued through several distinct mechanisms.

Transformation cloaking uses coordinate invariance and, in its ideal form, singular anisotropic media \cite{greenleaf2003nonuniqueness,pendry2006controlling}.
Resonance-based cloaking exploits anomalous localized resonance, typically with sign-changing parameters and limiting absorption \cite{milton2006cloaking}.
Passive scattering cancellation tunes coatings to cancel selected scattering or multipole coefficients \cite{alu2005achieving,liu2024simultaneously}.
Active exterior cloaking uses controlled sources to create a quiet region with a small far field \cite{vasquez2009active}.
We follow the third route with finite positive scalar isotropic coating conductivities and perfectly bonded interfaces, without singular transformations, sign-changing media, or active sources.

The scattered potential \(w\) has a far-field multipole expansion whose leading block is the classical polarization tensor \cite{schiffer1949virtual} and whose full coefficient family consists of the generalized polarization tensors (GPTs) \cite{ammari2003high,ammari2003properties}.
In two dimensions, the ordinary real contracted GPTs (CGPTs) express this expansion in a real harmonic basis and quantify the coupling between incident and outgoing harmonic modes \cite{ammari2014reconstruction,ammari2007polarization}.
Their first-order part is the polarization tensor, whose vanishing removes the leading response to uniform incident fields and is called weak neutrality.
Complete \(K\)-CGPT cancellation means that every such coupling through harmonic order \(K\) vanishes, thereby eliminating the corresponding finite set of scattering channels and providing the mechanism underlying GPT-vanishing structures \cite{ammari2013enhancement}.

Exact neutrality means that a chosen incident field is completely unchanged outside the inclusion.
For a single applied field, this can occur even in noncircular structures, including constructions based on conformal maps \cite{jarczyk2012neutral,milton2001neutral}.
The situation becomes rigid when exact neutrality is required for two independent uniform fields: the two boundaries must be confocal ellipses, reducing to concentric circles when the background is isotropic \cite{kang2014coated}.
For \(K>1\), complete \(K\)-CGPT cancellation strengthens weak neutrality by removing every coupling in the square block \(1\leq m,n\leq K\), but it remains a finite-order condition because responses outside this block may survive.
We therefore ask whether these finitely many scattering coefficients can be canceled exactly by a noncircular multilayer with perfectly bonded interfaces and finite positive scalar isotropic coating conductivities.

In a regularized near-cloak, the accuracy is measured by the decay of the exterior perturbation as the virtual inclusion size \(\rho\) tends to zero \cite{kohn2008cloaking}.
Using a concentric GPT-vanishing structure improves this decay: rotational symmetry leaves one scalar condition at each order, and canceling the first \(K\) radial GPTs yields an error of \(O(\rho^{2K+2})\) in two dimensions \cite{ammari2013enhancement}.
This raises the material question of whether positive coating conductivities achieving these cancellations exist for arbitrary prescribed radii and a fixed core conductivity.

In both two and three dimensions, suitable positive coating conductivities exist for arbitrary prescribed radii and any fixed nonnegative core conductivity, allowing the first \(K\) radial GPTs to be canceled \cite{sun2026existence}.
Thus the radial material problem is settled; the main difficulty begins beyond circular geometry, where angular modes no longer decouple.
Existing approaches address this difficulty by varying the interface shapes, adding controls on imperfect interfaces, or adapting the multipole basis to the geometry.

The first treats the interfaces as design variables: shape differentiation and sensitivity analysis describe how GPTs respond to boundary perturbations \cite{ammari2022localized,ammari2010conductivity} and support numerical optimization of GPT-vanishing multilayers \cite{feng2017construction}.
Rigorous existence is also known for small perturbations of balls and disks, but only for weak neutrality rather than complete higher-order cancellation \cite{kang2021polarization,kang2022existence}.

The second direction introduces additional controls through imperfect interfaces.
Position-dependent interface parameters yield weakly neutral inclusions of general shape \cite{kang2019construction}, while a Faber polynomial framework yields constructions with vanishing leading GPTs \cite{choi2024construction}.
These results demonstrate the flexibility of interface controls, but their transmission laws differ from those of the perfectly bonded interfaces considered here.

The third direction adapts the multipole basis to the geometry.
Exterior conformal maps, Faber polynomials, and Grunsky coefficients provide explicit representations of layer-potential operators and geometric multipole coefficients \cite{choi2021analytical,duren1983univalent,faber1903uber,grunsky1939koeffizientenbedingungen,jung2021series,kang2015construction}.
This framework also yields a geometric multipole expansion and semi-neutral inclusions of general shape \cite{choi2023geometric}.
This representation makes noncircular mode coupling explicit.

Across these three directions, complete higher-order cancellation for multilayers with perfectly bonded interfaces and finite positive isotropic coating conductivities remains unresolved.
The underlying obstacle is a mismatch between constraints and controls.
For a radial structure with \(K\) coatings, rotational symmetry leaves \(K\) scalar cancellation conditions that can be matched by the \(K\) coating conductivities.
In a generic noncircular multilayer, mixed modes no longer vanish automatically, so the same material variables must satisfy many more conditions and the system becomes overdetermined.
A radial GPT-vanishing structure also cannot simply be carried to a noncircular geometry by an exterior conformal map.
Such a map generally does not extend through the whole multilayer as a physical coordinate transformation, while a nonconformal extension would produce anisotropic conductivities outside the class considered here.
Conformal analysis can therefore describe noncircular mode coupling, but it does not supply the missing controls needed for complete higher-order cancellation.

We overcome this mismatch in two ways.
First, a finite cyclic symmetry forces most coupled responses to vanish exactly, leaving one material equation for each coating conductivity on a fixed noncircular geometry.
Second, selected Fourier modes of the interfaces are varied together with the coating conductivities, providing enough variables for complete cancellation without imposing symmetry on the final structure.

The main contributions, in the order developed below, are as follows.

\begin{enumerate}[label=\textnormal{(\roman*)},leftmargin=*,itemsep=0.6em]
\item
\emph{Global one-coating weak neutrality.}
For \(K=1\), every admissible fixed \(C_Q\)-symmetric core--shell geometry with \(Q\geq3\), and every positive pair of core and background conductivities, admits a unique positive coating conductivity for which the polarization tensor vanishes.
This result is global on the positive coating axis and independent of the perturbative higher-order arguments.

\item
\emph{High-order cancellation on cyclic noncircular geometries.}
For \(K\geq2\) and \(Q>2K\), a nonsingular moment Jacobian yields locally unique coating conductivities that cancel the complete \(K\)-CGPT block.
This covers arbitrary fixed \(C_Q\)-symmetric targets for near-background cores and, alternatively, fixed finite, insulating, or  perfectly conducting core types when the core is sufficiently small; homothetic strictly star-shaped targets satisfy the required nondegeneracy automatically.
Global existence can also be established when the cyclic perturbation is small enough.

\item
\emph{Symmetry-free high-order cancellation.}
Starting from a radial \(K\)-GPT-vanishing root, explicit material and interface nondegeneracy allows selected interface Fourier modes and coating conductivities to cancel every ordinary CGPT through total degree \(2K+1\) without imposing symmetry on the final geometry.
For \(K=3\), interval certification of the radial root and all derivative blocks yields exact families with neither rotational nor reflection symmetry.

\item
\emph{Scattering consequences and certified examples.}
Total-degree cancellation gives an \(O(\varepsilon^{2K+2})\) exterior estimate for every entire harmonic incident field after scaling.
We also certify a fixed chiral \(C_8\) three-coating structure whose complete \(3\)-CGPT block vanishes, while a symmetry-permitted response at the first total degree beyond the certified range, namely total degree eight, is rigorously nonzero.
\end{enumerate}

The symmetry-free construction is the principal higher-order result of the paper.

Except for the scalar \(K=1\) theorem, the higher-order existence results are local in the parameter regimes specified above.
For the fixed chiral \(C_8\) geometry, uniqueness is established only in the validated box.
Global continuation on general noncircular geometries and global uniqueness for the higher-order material problem remain open.

\subsection*{Organization of the paper}

The remainder of the paper is organized as follows.
Section~\ref{sec:foundations} formulates the multilayer conductivity problem and develops the conformal representation and cyclic selection rules.
Section~\ref{sec:main-results} states the global ($K$=1) theorem and the local higher-order results under cyclic symmetry, together with their scattering consequences.
Section~\ref{sec:global-one-shell} proves the scalar result, while Section~\ref{sec:local-constructions} proves, in order, the fixed-geometry theorem and homothetic corollary, the small-core theorem, and the continuation results for common conformal level curves.
Section~\ref{sec:symmetry-free-completion} then states and proves the main result without geometric symmetry by adding interface controls to the conductivity tuning.
Section~\ref{sec:harmonic-backgrounds} converts the cancellation results into scattering estimates for arbitrary entire harmonic incident fields, and Section~\ref{sec:examples} gives the computer-assisted and numerical examples.

\section{Problem formulation, conformal geometry, and cyclic selection}
\label{sec:foundations}

This section introduces the multilayer conductivity problem and the CGPTs, then derives the selection rules associated with cyclic symmetry.

\subsection{Conductivity problem and the complete CGPT block}
\label{sec:problem}

Let
\[
 D_0\Subset D_1\Subset\cdots\Subset D_K\subset\mathbb{R}^2
\]
be bounded simply connected domains with sufficiently smooth interfaces.
The core is $D_0$, the $j$-th coating is $D_j\setminus\overline{D_{j-1}}$, and the background conductivity is normalized to one.
We consider a piecewise constant conductivity distribution.
For positive constants $\sigma_0,\ldots,\sigma_K$, set
\begin{equation}
 \sigma(x)=
 \begin{cases}
  \sigma_0,&x\in D_0,\\
  \sigma_j,&x\in D_j\setminus\overline{D_{j-1}},\quad 1\leq j\leq K,\\
  1,&x\in\mathbb{R}^2\setminus\overline{D_K}.
 \end{cases}
 \label{eq:conductivity}
\end{equation}
For an entire harmonic incident field $H$, let $u$ be the solution of \eqref{eq:conductivity-transmission-problem} corresponding to the conductivity \eqref{eq:conductivity}.
Potential and conormal flux are continuous across every interface.
For \(z_*\in\mathbb{R}^2\) and \(\varepsilon>0\), define the translated and uniformly scaled conductivity
\begin{equation}
 \sigma_{\varepsilon,z_*}(x)
 :=\sigma\left(\frac{x-z_*}{\varepsilon}\right).
 \label{eq:scaled-conductivity}
\end{equation}

Identify $x=(x_1,x_2)$ with $z=x_1+ix_2$ and define the real harmonic polynomials
\begin{equation}
 P_n^c(z)=\operatorname{Re} z^n,
 \qquad
 P_n^s(z)=\operatorname{Im} z^n,
 \qquad n\geq1.
 \label{eq:harmonic-basis}
\end{equation}
The CGPTs describe how each incident harmonic mode generates outgoing multipoles.
Let $u_n^\beta$ solve \eqref{eq:conductivity-transmission-problem} with $H=P_n^\beta$ as defined in \eqref{eq:harmonic-basis}.
For $m,n\geq1$ and $\alpha,\beta\in\{c,s\}$, we define the real CGPTs by
\begin{equation}
 M_{mn}^{\alpha\beta}[\sigma]
 =\int_{\mathbb{R}^2}(\sigma-1)
   \nabla u_n^\beta\cdot\nabla P_m^\alpha\,\mathrm{d} x,
 \label{eq:real-cgpt}
\end{equation}
which is in accordance with the CGPTs defined in \cite{ammari2014reconstruction}.
Note that the integrand in \eqref{eq:real-cgpt} is supported in \(\overline{D_K}\).
Reciprocity gives
\begin{equation}
 M_{mn}^{\alpha\beta}=M_{nm}^{\beta\alpha}.
 \label{eq:reciprocity}
\end{equation}

By the standard expansion of the logarithmic fundamental solution, these coefficients determine the exterior perturbation \cite{ammari2007polarization,ammari2013enhancement}.
Write
\begin{equation}
 H(r,\theta)=H(0)+\sum_{n=1}^{\infty}r^n
  \bigl(a_n^c\cos n\theta+a_n^s\sin n\theta\bigr),
 \label{eq:H-expansion}
\end{equation}
For $H$ as in \eqref{eq:H-expansion} and sufficiently large $r$, the exterior multipole expansion is
\begin{align}
 u-H={}&-\sum_{m=1}^{\infty}\frac{\cos m\theta}{2\pi m r^m}
 \sum_{n=1}^{\infty}
 \bigl(M_{mn}^{cc}a_n^c+M_{mn}^{cs}a_n^s\bigr)\notag\\
 &-\sum_{m=1}^{\infty}\frac{\sin m\theta}{2\pi m r^m}
 \sum_{n=1}^{\infty}
 \bigl(M_{mn}^{sc}a_n^c+M_{mn}^{ss}a_n^s\bigr).
 \label{eq:multipole-expansion}
\end{align}
Here $n$ labels the incident mode and $m$ the outgoing multipole.
For a nonradial structure, an incident mode of order $n$ may generate outgoing modes of different orders $m$.
We therefore cancel the full block $1\leq m,n\leq K$, rather than only its diagonal entries; uniform-field neutrality concerns only $n=1$.

To express the action of rotations, we combine the four real CGPTs into two complex families:
\begin{align}
 \mathbb N_{mn}^{(1)}
 &=(M_{mn}^{cc}-M_{mn}^{ss})
   +i(M_{mn}^{cs}+M_{mn}^{sc}),
 \label{eq:N1}\\
 \mathbb N_{mn}^{(2)}
 &=(M_{mn}^{cc}+M_{mn}^{ss})
   +i(M_{mn}^{cs}-M_{mn}^{sc}).
 \label{eq:N2}
\end{align}

The complex families \eqref{eq:N1}--\eqref{eq:N2} contain the same information as the four real CGPTs and lead to the following cancellation condition.

\begin{definition}
A conductivity structure is \emph{$K$-GPT-vanishing} if its complete $K$-CGPT block satisfies
\begin{equation}
 M_{mn}^{\alpha\beta}=0
 \quad\text{for all }1\leq m,n\leq K,
 \quad\alpha,\beta\in\{c,s\}.
 \label{eq:complete-block}
\end{equation}
Thus all four ordinary real CGPT families vanish for every $1\leq m,n\leq K$, or equivalently the complete $2K\times2K$ real contracted GPT block is zero.
\end{definition}

\subsection{Common conformal level curves and cyclic selection}
\label{sec:cyclic-geometry}

Canceling the complete CGPT block defined above generally requires many coupled conditions.
We reduce them by imposing a finite rotational symmetry, which we first define precisely.

\begin{definition}
Fix an integer $Q\geq3$ and set $\omega=e^{2\pi i/Q}$.
The cyclic group
\[
 C_Q:=\{\omega^q:0\leq q\leq Q-1\}
\]
acts on $\mathbb{C}$ by $z\mapsto\omega^q z$.
A set is called \emph{$C_Q$-symmetric} if it is invariant under this action, and a conductivity $\sigma$ is called \emph{$C_Q$-invariant} if $\sigma(\omega z)=\sigma(z)$ for almost every $z\in\mathbb{C}$.
\end{definition}

A convenient class of $C_Q$-symmetric multilayers is obtained from a common Laurent map.
In a reference $w$-plane, fix radii
\begin{equation}
 0<r_0<r_1<\cdots<r_K
 \label{eq:radii}
\end{equation}
and deform the circles $|w|=r_j$ by the same finite Laurent map.
For an integer $L\geq1$ and parameters $a=(a_1,\ldots,a_L)\in\mathbb{C}^L$, define
\begin{equation}
 \Phi_a(w)=w+\sum_{\ell=1}^{L}a_\ell w^{1-\ell Q}.
 \label{eq:common-map}
\end{equation}
The leading term is the identity map, while the coefficients $a_\ell$ control the noncircular deformation.
The powers $1-\ell Q$ are chosen so that the deformation respects rotations by $2\pi/Q$.
The candidate interfaces are the common level curves
\begin{equation}
 \Gamma_j(a)=\Phi_a(\{|w|=r_j\}),
 \qquad 0\leq j\leq K.
 \label{eq:level-curves}
\end{equation}
Only the restriction of $\Phi_a$ to $|w|\geq r_0$ is used.
The core will be the region enclosed by the innermost interface, so the Laurent series is never evaluated at $w=0$.
To ensure that these images form valid multilayer interfaces, we must exclude self-intersections and crossings between different layers.
The following condition guarantees this on the entire region where $\Phi_a$ is used.

\begin{lemma}\label{lem:global-injectivity}
If
\begin{equation}
 \sum_{\ell=1}^{L}(\ell Q-1)|a_\ell|r_0^{-\ell Q}<1,
 \label{eq:univalence}
\end{equation}
then $\Phi_a$ is injective on $\{|w|\geq r_0\}\cup\{\infty\}$.
In particular, the curves \eqref{eq:level-curves} are analytic, simple, and strictly nested.
\end{lemma}

\begin{proof}
For $p\geq1$ and $|w_1|,|w_2|\geq r_0$, the algebraic difference quotient gives
\[
 \left|\frac{w_1^{-p}-w_2^{-p}}{w_1-w_2}\right| \leq p r_0^{-(p+1)}.
\]
With $p=\ell Q-1$,
\[
 |\Phi_a(w_1)-\Phi_a(w_2)| \geq |w_1-w_2| \left(1-\sum_{\ell=1}^{L}(\ell Q-1)|a_\ell|r_0^{-\ell Q}\right),
\]
which is positive for $w_1\neq w_2$ and proves injectivity.
Letting $w_2\to w_1$ in the same estimate shows that $\Phi_a'$ does not vanish.
The image of each annulus $\{r_j\leq|w|\leq r_{j+1}\}$ is bounded by $\Gamma_j(a)$ and $\Gamma_{j+1}(a)$, so the curves are simple and strictly nested.
\end{proof}

Under \eqref{eq:univalence}, the map $\Phi_a$ is conformal on $|w|>r_0$, and we denote the bounded component enclosed by $\Gamma_j(a)$ by $D_j(a)$.
The special exponents in \eqref{eq:common-map} give
\begin{equation}
 \Phi_a(\omega w)=\omega\Phi_a(w).
 \label{eq:equivariance}
\end{equation}
By \eqref{eq:equivariance}, every interface and every phase region is invariant under rotation through $2\pi/Q$.
This rotational symmetry imposes exact restrictions on the coupling between harmonic modes.

\begin{proposition}\label{prop:selection}
Let $\sigma$ be a bounded, uniformly positive scalar conductivity that equals $1$ outside $D_K$ and is invariant under the $C_Q$ action,
\begin{equation}
 \sigma(\omega z)=\sigma(z)
 \quad\text{for a.e. }z\in\mathbb{C}.
 \label{eq:CQ-conductivity}
\end{equation}
This includes the piecewise constant conductivity \eqref{eq:conductivity} when the domains are the common level-curve domains \eqref{eq:level-curves}.
Then
\begin{align}
 \mathbb N_{mn}^{(1)}&=0
 &&\text{unless }m+n\equiv0\pmod Q,
 \label{eq:selection1}\\
 \mathbb N_{mn}^{(2)}&=0
 &&\text{unless }m-n\equiv0\pmod Q.
 \label{eq:selection2}
\end{align}
If $Q>2K$, the complete block \eqref{eq:complete-block} vanishes if and only if the $K$ quantities
\begin{equation}
 \mathcal G_n:=\frac{1}{2\pi n}\mathbb N_{nn}^{(2)},
 \qquad 1\leq n\leq K,
 \label{eq:reduced-response}
\end{equation}
vanish.
These quantities are real by \eqref{eq:reciprocity} and \eqref{eq:N2}.
\end{proposition}

\begin{proof}
Since $P_n^c(z)+iP_n^s(z)=z^n$, we have
\[
 P_n^c(\omega z)+iP_n^s(\omega z) =\omega^n\bigl(P_n^c(z)+iP_n^s(z)\bigr).
\]
Changing variables \(x=\omega y\) in \eqref{eq:real-cgpt} and using uniqueness of the rotated transmission solution, together with \eqref{eq:CQ-conductivity}, gives
\[
 \mathbb N_{mn}^{(1)}=\omega^{m+n}\mathbb N_{mn}^{(1)},
 \qquad
 \mathbb N_{mn}^{(2)}=\omega^{n-m}\mathbb N_{mn}^{(2)}.
\]
A coefficient can therefore be nonzero only when its phase factor equals one, which proves \eqref{eq:selection1}--\eqref{eq:selection2}.
If $m,n\leq K$ and $Q>2K$, then $0<m+n<Q$ and $|m-n|<Q$.
Thus the first family and the off-diagonal part of the second family vanish.
On the diagonal, \eqref{eq:selection1} and \eqref{eq:N1} give \(M_{nn}^{cc}=M_{nn}^{ss}\) and \(M_{nn}^{cs}=-M_{nn}^{sc}\).
Reciprocity \eqref{eq:reciprocity} also gives \(M_{nn}^{cs}=M_{nn}^{sc}\), so both cross terms vanish.
Moreover, \eqref{eq:N2} and \eqref{eq:reduced-response} show that \(\mathcal G_n=0\) gives \(M_{nn}^{cc}+M_{nn}^{ss}=0\).
Hence all four diagonal entries vanish.
The converse follows immediately from \eqref{eq:N2} and \eqref{eq:reduced-response}.
\end{proof}

\begin{remark}
Without structural symmetry, reciprocity leaves a symmetric real $2K\times2K$ array with $K(2K+1)$ independent real conditions.
The cyclic selection rules reduce this array to $K$ real diagonal conditions, matching the $K$ coating conductivities.
Without the cyclic symmetry, the $K$-parameter problem is generically overdetermined.
\end{remark}

The block reduction requires $Q>2K$.
The stronger condition $Q\geq2K+2$ is used only for the later total-degree estimates, whereas the selection rules themselves hold for every $Q\geq3$.

\section{Main results under cyclic symmetry}
\label{sec:main-results}

This section states the main results obtained under cyclic symmetry and their scattering consequences.
The scalar case \(K=1\) gives a globally unique positive coating conductivity, whereas the higher-order problem \(K\geq2\) is treated in two local regimes: conductivity tuning on a fixed geometry near the homogeneous state, and cancellation for an arbitrary fixed core type when the core is sufficiently small.
The homothetic constructions and those based on common conformal level curves provide concrete geometric classes for these higher-order results.
The resulting GPT cancellations yield the small-inclusion and fixed-scale estimates stated at the end of the section.

The principal result without geometric symmetry is stated later.
Its formulation requires the framework for interface controls developed in Section~\ref{sec:symmetry-free-completion}, where the result is also proved.

\subsection{An independent scalar case: \texorpdfstring{$K=1$}{K=1}}
\begin{theorem}\label{thm:global-one-shell}
	Let \(D\Subset\Omega\subset\mathbb{R}^2\) be bounded domains such that \(\Omega\) is simply connected, \(\partial D\) and \(\partial\Omega\) are of class \(C^{1,\alpha}\) for some \(0<\alpha<1\), and \(\Omega\setminus\overline D\) is nonempty and connected.
	Suppose that \(D\) and \(\Omega\) are invariant under rotation through \(2\pi/Q\) about the same point for an integer \(Q\geq3\).
	For every finite positive core conductivity \(\sigma_{\rm c}\) and background conductivity \(\sigma_{\rm m}\), there exists exactly one finite positive coating conductivity \(\sigma_{\rm sh}^*\).
	For the conductivity normalized by the background value,
	\begin{equation}
		\sigma^*(x):=
		\begin{cases}
			\sigma_{\rm c}/\sigma_{\rm m},&x\in D,\\
			\sigma_{\rm sh}^*/\sigma_{\rm m},
			&x\in\Omega\setminus\overline D,\\
			1,&x\in\mathbb{R}^2\setminus\overline\Omega,
		\end{cases}
		\label{eq:global-one-shell-conductivity}
	\end{equation}
	one has
	\begin{equation}
		M_{11}^{\alpha\beta}[\sigma^*]=0
		\quad\text{for all }\alpha,\beta\in\{c,s\}.
		\label{eq:global-one-shell-neutrality}
	\end{equation}
	Equivalently, the coated inclusion is weakly neutral to every uniform incident field.
	Moreover,
	\begin{align}
		\sigma_{\rm c}>\sigma_{\rm m}
		&\quad\Longrightarrow\quad 0<\sigma_{\rm sh}^*<\sigma_{\rm m},
		\label{eq:one-shell-root-below}\\
		0<\sigma_{\rm c}<\sigma_{\rm m}
		&\quad\Longrightarrow\quad \sigma_{\rm sh}^*>\sigma_{\rm m},
		\label{eq:one-shell-root-above}
	\end{align}
	whereas \(\sigma_{\rm c}=\sigma_{\rm m}\) implies \(\sigma_{\rm sh}^*=\sigma_{\rm m}\).
\end{theorem}

\begin{corollary}\label{cor:one-shell-uniform-suppression}
Under the assumptions of Theorem~\ref{thm:global-one-shell}, let \(u^g\) be the solution at the unique coating conductivity \(\sigma_{\rm sh}^*\) for the uniform incident field \(H_g(x)=g\cdot x\).
Then
\begin{equation}
u^g-H_g=O(|x|^{-(Q-1)}),
\qquad
\nabla(u^g-H_g)=O(|x|^{-Q})
\quad (|x|\to\infty).
\label{eq:one-shell-uniform-decay}
\end{equation}
Consequently, the first \(Q-2\) outgoing multipoles generated by every uniform incident field vanish.
\end{corollary}

\begin{remark}\label{rem:one-shell-uniform-scope}
Corollary~\ref{cor:one-shell-uniform-suppression} concerns only degree-one incident fields.
It implies neither cancellation of a complete higher-order CGPT block nor the same fixed-scale decay for general polynomial or harmonic fields; those conclusions require separate multiparameter arguments.
\end{remark}

Theorem~\ref{thm:global-one-shell} and Corollary~\ref{cor:one-shell-uniform-suppression} are proved in Section~\ref{sec:global-one-shell}.

\subsection{The higher-order regime: \texorpdfstring{$K\geq2$}{K>=2}}

Fix \(K\geq2\) and \(Q>2K\).
The proofs of all results in this subsection are collected in Section~\ref{sec:local-constructions}.

\subsubsection{Fixed geometries near the homogeneous state}

Let
\[
 \mathcal D=(D_0,D_1,\ldots,D_K)
\]
be a fixed collection of bounded simply connected \(C^{1,\alpha}\) domains satisfying
\[
 D_0\Subset D_1\Subset\cdots\Subset D_K
\]
and invariant under rotation through \(2\pi/Q\) about the origin.
Put
\[
 \mathcal A_0=D_0,
 \qquad
 \mathcal A_j=D_j\setminus\overline{D_{j-1}},
 \quad 1\leq j\leq K,
\]
and define
\begin{equation}
 \sigma_{\mathcal D,\delta,\eta}(x)=
 \begin{cases}
  e^\delta,&x\in\mathcal A_0,\\
  e^{\eta_j},&x\in\mathcal A_j,\quad1\leq j\leq K,\\
  1,&x\notin D_K,
 \end{cases}
 \label{eq:fixed-target-conductivity}
\end{equation}
where \(\delta\in\mathbb{R}\) and \(\eta=(\eta_1,\ldots,\eta_K)\in\mathbb{R}^K\).
Define the geometry-dependent material Jacobian \(J^{\mathcal D}\in\mathbb{R}^{K\times K}\) and core vector \(v^{\mathcal D}\in\mathbb{R}^K\) by
\begin{align}
 J^{\mathcal D}_{nj}
 &:=\frac{n}{\pi}\int_{\mathcal A_j}|z|^{2n-2}\,\mathrm{d} x,
 &&1\leq n,j\leq K,
 \label{eq:fixed-target-J}\\
 v^{\mathcal D}_n
 &:=\frac{n}{\pi}\int_{D_0}|z|^{2n-2}\,\mathrm{d} x,
 &&1\leq n\leq K.
 \label{eq:fixed-target-v}
\end{align}

\begin{theorem}\label{thm:fixed-target}
Suppose that the fixed \(C_Q\)-symmetric geometry \(\mathcal D\) satisfies
\begin{equation}
 \det J^{\mathcal D}\neq0.
 \label{eq:fixed-target-nondegeneracy}
\end{equation}
Then there exist \(\delta_*>0\), a neighborhood \(V\) of \(0\) in \(\mathbb{R}^K\), and a real-analytic map
\begin{equation}
 \eta^{\mathcal D}:(-\delta_*,\delta_*)\longrightarrow V,
 \qquad
 \eta^{\mathcal D}(0)=0,
 \label{eq:fixed-target-branch}
\end{equation}
such that, for every \(\lvert\delta\rvert<\delta_*\), the vector \(\eta^{\mathcal D}(\delta)\) is the unique vector in \(V\) for which
\begin{equation}
 M_{mn}^{\alpha\beta}
 [\sigma_{\mathcal D,\delta,\eta^{\mathcal D}(\delta)}]=0
 \quad
 \text{for all }1\leq m,n\leq K,
 \quad\alpha,\beta\in\{c,s\}.
 \label{eq:fixed-target-cancel}
\end{equation}
Moreover,
\begin{equation}
 \eta^{\mathcal D}(\delta)
 =-(J^{\mathcal D})^{-1}v^{\mathcal D}\,\delta+O(\delta^2)
 \qquad(\delta\to0).
 \label{eq:fixed-target-expansion}
\end{equation}
\end{theorem}

Theorem~\ref{thm:fixed-target} shows the fixed geometry need not be close to a concentric circular geometry; only the log-conductivity variables \(\delta\) and \(\eta\) are required to be close to zero.
Besides, it reduces the problem of locally tuning the coating conductivities to the invertibility of the geometry-dependent matrix \(J^{\mathcal D}\). 
This nondegeneracy does not follow from nesting and cyclic symmetry alone, so we next identify a concrete noncircular class for which it can be verified explicitly. 
Consider a homothetic multilayer generated as follows.
Let \(D\subset\mathbb{R}^2\) be a bounded simply connected \(C^{1,\alpha}\) domain that is invariant under rotation through \(2\pi/Q\) and strictly star-shaped with respect to the origin, in the sense that \(x\cdot\nu_D(x)>0\) on \(\partial D\).
Choose
\begin{equation}
 0<\rho_0<\rho_1<\cdots<\rho_K,
 \qquad
 D_j=\rho_jD,
 \quad0\leq j\leq K.
 \label{eq:homothetic-domains}
\end{equation}
Put
\[
 x_j:=\rho_j^2,
 \qquad
 B_{nj}:=x_j^n-x_{j-1}^n,
 \qquad
 b_n:=x_0^n.
\]

\begin{corollary}\label{cor:homothetic}
For every homothetic geometry \eqref{eq:homothetic-domains}, the matrix \(J^{\mathcal D}\) is nonsingular.
Consequently, for every sufficiently small core log-conductivity \(\delta\), there exists a unique coating log-conductivity vector near zero such that the resulting structure is \(K\)-GPT-vanishing.
In fact,
\begin{equation}
 \det J^{\mathcal D}>0,
 \label{eq:homothetic-positive-determinant}
\end{equation}
and
\begin{equation}
 \eta^{\mathcal D}(\delta)
 =-B^{-1}b\,\delta+O(\delta^2).
 \label{eq:homothetic-expansion}
\end{equation}
The leading material compensation law is independent of the shape of \(D\).
In particular, when \(D\) is noncircular, the result does not require it to be close to a disk.
\end{corollary}

Theorem~\ref{thm:fixed-target} and Corollary~\ref{cor:homothetic} are proved in Subsection~\ref{subsec:proof-fixed-geometry}.

\subsubsection{Small cores of fixed material type}

The preceding fixed-geometry theorem assumes that the core conductivity is sufficiently close to the background on an arbitrary fixed geometry.
The next result instead makes the core geometrically small while keeping its type fixed.
The core may have an arbitrary finite positive conductivity, may be insulating, or may be a floating perfect conductor.

Let \(\mathcal B,D_1,\ldots,D_K\subset\mathbb{R}^2\) be bounded simply connected \(C^{1,\alpha}\) domains containing the origin.
Assume that they are invariant under rotation through \(2\pi/Q\), and that
\[
 D_1\Subset D_2\Subset\cdots\Subset D_K.
\]
For sufficiently small \(\varepsilon>0\), put
\begin{equation}
 D_0^\varepsilon:=\varepsilon\mathcal B,
 \qquad
 \mathcal A_1^\varepsilon
 :=D_1\setminus\overline{\varepsilon\mathcal B},
 \qquad
 \mathcal A_j^\varepsilon
 :=D_j\setminus\overline{D_{j-1}},\quad 2\leq j\leq K.
 \label{eq:small-core-regions}
\end{equation}
At \(\varepsilon=0\), set
\begin{equation}
 \mathcal A_1^0:=D_1,
 \qquad
 \mathcal A_j^0:=\mathcal A_j^\varepsilon,
 \quad 2\leq j\leq K.
 \label{eq:small-core-limit-regions}
\end{equation}
Fix \(\kappa\in[0,\infty]\), with the endpoint conventions described below, and let \(\eta=(\eta_1,\ldots,\eta_K)\in\mathbb{R}^K\).
For \(0<\kappa<\infty\), define
\begin{equation}
 \sigma_{\varepsilon,\eta}(x)=
 \begin{cases}
  \kappa,&x\in\varepsilon\mathcal B,\\
  \mathrm{e}^{\eta_j},&x\in\mathcal A_j^\varepsilon,
       \quad 1\leq j\leq K,\\
  1,&x\notin D_K,
 \end{cases}
 \label{eq:small-core-conductivity}
\end{equation}
At \(\kappa=0\), the core is insulating; at \(\kappa=\infty\), it is a floating perfect conductor, meaning that its potential is constant and its total flux is zero.
All coating interfaces remain perfectly bonded.
The endpoint CGPTs are defined through the same exterior multipole expansion as in the finite-conductivity case.
The notation \(\sigma_{\varepsilon,\eta}\) refers to the resulting finite-conductivity or endpoint structure in all three cases.
Put
\begin{align}
 J^0_{nj}
 &:=\frac{n}{\pi}\int_{\mathcal A_j^0}
 |z|^{2n-2}\,\mathrm{d} x,
 &&1\leq n,j\leq K,
 \label{eq:small-core-limit-J}\\
 g_n^{\mathcal B,\kappa}
 &:=\frac{1}{2\pi n}\mathbb N_{nn}^{(2)}
 [\mathcal B;\kappa],
 &&1\leq n\leq K.
 \label{eq:small-core-reference-response}
\end{align}
Here \([\mathcal B;\kappa]\) denotes the single reference inclusion with conductivity \(\kappa\) when \(0<\kappa<\infty\), the insulating problem when \(\kappa=0\), and the floating perfectly conducting problem when \(\kappa=\infty\).

\begin{theorem}\label{thm:small-strong-core}
Suppose that
\begin{equation}
 \det J^0\neq0.
 \label{eq:small-core-nondegeneracy}
\end{equation}
For each fixed \(\kappa\in[0,\infty]\), interpreted as above, there exist \(\varepsilon_*>0\), an open neighborhood \(V\subset\mathbb{R}^K\) of the origin, and \(C_\kappa>0\) such that, for every \(0<\varepsilon<\varepsilon_*\), there is exactly one \(\eta^\varepsilon\in V\) for which
\begin{equation}
 M_{mn}^{\alpha\beta}
 [\sigma_{\varepsilon,\eta^\varepsilon}]=0
 \quad
 \text{for all }1\leq m,n\leq K,
 \quad \alpha,\beta\in\{c,s\}.
 \label{eq:small-core-complete-cancellation}
\end{equation}
The coating conductivities are finite, positive, scalar, isotropic, and piecewise constant.
Moreover,
\begin{equation}
 |\eta^\varepsilon|\leq C_\kappa\varepsilon^2,
 \label{eq:small-core-root-bound}
\end{equation}
and
\begin{equation}
 \eta^\varepsilon
 =-\varepsilon^2g_1^{\mathcal B,\kappa}
   (J^0)^{-1}e_1
   +O_\kappa(\varepsilon^3),
 \qquad e_1=(1,0,\ldots,0)^{\mathsf T}.
 \label{eq:small-core-leading-expansion}
\end{equation}
Here the constants and the admissible core scale may depend on the fixed geometry and on the fixed core type; no uniformity over \([0,\infty]\) is asserted.
The reference core \(\mathcal B\) need not be close to a disk.
At the two endpoints,
\begin{equation}
 g_1^{\mathcal B,0}<0,
 \qquad
 g_1^{\mathcal B,\infty}>0,
 \label{eq:small-core-endpoint-signs}
\end{equation}
and there are constants \(0<c_\kappa<C_\kappa'<\infty\) such that
\begin{equation}
 c_\kappa\varepsilon^2
 \leq|\eta^\varepsilon|
 \leq C_\kappa'\varepsilon^2
 \label{eq:small-core-two-sided-bound}
\end{equation}
for all sufficiently small \(\varepsilon\).
\end{theorem}

If the fixed outer interfaces are homothetic, \(D_j=\rho_jD\) with \(0<\rho_1<\cdots<\rho_K\), the same difference--Vandermonde calculation as in Corollary~\ref{cor:homothetic}, now with the inner endpoint equal to zero, gives \(\det J^0>0\).
Thus this class satisfies the nondegeneracy hypothesis of Theorem~\ref{thm:small-strong-core} automatically.

Theorem~\ref{thm:small-strong-core} is proved in Subsection~\ref{subsec:proof-small-core}.

\subsubsection{Common conformal geometries}

We now specialize to the common conformal geometry \eqref{eq:radii}--\eqref{eq:level-curves} and impose the stronger condition \(Q\geq2K+2\) needed for the total degree cancellation and the estimates for entire harmonic incident fields below.
Parameterize the conductivities logarithmically by
\begin{equation}
 \sigma_{a,\delta,\eta}(x)=
 \begin{cases}
  e^\delta,&x\in D_0(a),\\
  e^{\eta_j},&x\in D_j(a)\setminus\overline{D_{j-1}(a)},
       \quad1\leq j\leq K,\\
  1,&x\notin D_K(a),
 \end{cases}
 \label{eq:log-conductivity}
\end{equation}
where \(\eta=(\eta_1,\ldots,\eta_K)\in\mathbb{R}^K\).
The domains depend on the conformal coefficients \(a\) through \eqref{eq:level-curves}; complex shape coefficients are regarded as pairs of real parameters.

\begin{corollary}\label{cor:common-map-branch}
There exist a neighborhood \(\mathcal U\) of the circular shape, a number \(\delta_*>0\), and a unique real-analytic map
\begin{equation}
 \eta:\mathcal U\times(-\delta_*,\delta_*)\longrightarrow\mathbb{R}^K,
 \qquad \eta(0,0)=0,
 \label{eq:eta-branch}
\end{equation}
such that the complete \(K\)-CGPT block vanishes:
\begin{equation}
 M_{mn}^{\alpha\beta}
 [\sigma_{a,\delta,\eta(a,\delta)}]=0
 \quad
 \text{for all }1\leq m,n\leq K,
 \quad\alpha,\beta\in\{c,s\}.
 \label{eq:main-cancel}
\end{equation}
The resulting conductivities are finite, positive, constant in each phase, and isotropic.
If \(a\neq0\), the interfaces are noncircular, and if \(\delta\neq0\), the resulting \(K\)-GPT-vanishing structure is nontrivial.
The uniqueness is local near the zero log-conductivity vector.
\end{corollary}

\begin{proposition}\label{prop:local-tangent}
Put \(x_j=r_j^2\) and define
\begin{equation}
 A_{nj}=x_j^n-x_{j-1}^n,
 \qquad v_n=x_0^n,
 \qquad 1\leq n,j\leq K.
 \label{eq:A-v}
\end{equation}
Then \(\det A>0\), and
\begin{equation}
 \eta(a,\delta)
 =-A^{-1}v\,\delta
 +O\bigl(\delta^2+|a|^2|\delta|\bigr),
 \label{eq:branch-expansion}
\end{equation}
as \((a,\delta)\to(0,0)\).
Thus the first conductivity correction is explicit, while the first shape-induced correction is quadratic rather than linear.
\end{proposition}

Corollary~\ref{cor:common-map-branch} and Proposition~\ref{prop:local-tangent} are proved in Subsection~\ref{subsec:proof-common-conformal}.
\subsubsection{General core material type}
For $\kappa\in[0,\infty]$ and $\eta=(\eta_1,\ldots,\eta_K)\in\mathbb R^K$, denote by $\sigma_a^\kappa(\eta)$ the structure with core conductivity $\kappa$, coating conductivities $e^{\eta_j}$, and background conductivity $1$.
For $0<\kappa<\infty$, this is $\sigma_{a,\log\kappa,\eta}$ in \eqref{eq:log-conductivity}.
At $\kappa=0$, impose zero exterior normal flux on $\Gamma_0(a)$; at $\kappa=\infty$, impose a constant core potential and zero total flux across $\Gamma_0(a)$.
The endpoint CGPTs are defined by the exterior multipole expansion, as in Lemma~\ref{lem:small-core-endpoint-framework}.
Define the reduced material map
\begin{equation}
	\mathcal G_a^\kappa(\eta)
	:=
	\left(
	\frac{1}{2\pi n}\mathbb N_{nn}^{(2)}
	[\sigma_a^\kappa(\eta)]
	\right)_{n=1}^{K}.
	\label{eq:global-material-map}
\end{equation}
Thus \(\mathcal G_a^\kappa(\eta)=0\) is precisely the reduced
\(K\)-equation system selected by the \(C_Q\) symmetry.
The next theorem removes the restrictions on the core conductivity and core size, provided that the perturbation from the concentric geometry is sufficiently small.
The shape norm \(\left\lVert a\right\rVert _*\) is defined in \eqref{eq:global-shape-norm}.
\begin{theorem}\label{thm:global-material}
	Fix an arbitrary core conductivity \(0\leq\kappa\leq \infty\).
	There exists \(\rho_\kappa>0\) such that, whenever
	\begin{equation}
		\left\lVert a\right\rVert _*<\rho_\kappa,
		\label{eq:global-shape-smallness}
	\end{equation}
	there is at least one vector \(\eta^a\in\mathbb{R}^K\) for which
	\begin{equation}
		M_{mn}^{\alpha\beta}
		[\sigma_a^\kappa(\eta^a)]
		=0
		\quad
		\text{for all }1\leq m,n\leq K,
		\quad \alpha,\beta\in\{c,s\}.
		\label{eq:global-complete-cancellation}
	\end{equation}
	The coating conductivities \(\mathrm{e}^{\eta_j^a}\) are finite and positive and may be chosen in a compact material set depending on \(\kappa\) and the radii.
	If \(a\neq0\), the interfaces are noncircular, and if \(\kappa\neq1\), the structure is nontrivial.
\end{theorem}
Proof of Theorem \ref{thm:global-material} is shown in Subsection \ref{sec:Unif01}.

\subsection{Consequences for scattering}

Proposition~\ref{prop:total-degree} below shows that complete \(K\)-CGPT cancellation and \(Q\geq2K+2\), together with the cyclic selection rules, give
\[
M_{mn}^{\alpha\beta}=0
\quad\text{whenever }m+n\leq2K+1,
\qquad \alpha,\beta\in\{c,s\}.
\]

\begin{theorem}\label{thm:arbitrary-H}
		Let \(K\geq1\) and \(Q\geq2K+2\).
		Fix a \(C_Q\)-invariant reference multilayer
		\(\mathcal D=(D_0,\ldots,D_K)\), centered at the origin, with finite positive conductivities and \(C^{2,\alpha}\) interfaces, and suppose that its complete \(K\)-CGPT block vanishes.
		Let \(\sigma\) be its conductivity and let \(\sigma_{\varepsilon,z_*}\) be defined by \eqref{eq:scaled-conductivity}.
		Let \(\Omega\) be a fixed bounded smooth domain with \(B_R(z_*)\Subset\Omega\), assume \(z_*+\varepsilon\overline{D_K}\subset B_{R/2}(z_*)\), and let \(H\) be an entire harmonic incident field on \(\mathbb R^2\).
		Let \(u_\varepsilon\in H^1(\Omega)\) be the unique weak solution of
		\(\nabla\cdot(\sigma_{\varepsilon,z_*}\nabla u_\varepsilon)=0\) in \(\Omega\), with trace \(u_\varepsilon=H\) on \(\partial\Omega\).
		Then, for every compact \(E\Subset\Omega\setminus\{z_*\}\),
	\begin{equation}
		\left\lVert u_\varepsilon-H\right\rVert _{C^1(E)}
		\leq C_E\varepsilon^{2K+2}
		\left\lVert H\right\rVert _{L^\infty(B_R(z_*))}
		\label{eq:arbitrary-H-bound}
	\end{equation}
	for sufficiently small \(\varepsilon\).
	The constant may depend on \(R\), \(\Omega\), \(\operatorname{dist}(E,z_*)\), the reference geometry and its interface norms, and ellipticity bounds; it is uniform on compact parameter subsets satisfying those bounds.
\end{theorem}

\begin{corollary}\label{cor:polynomial}
Let \(Q\geq2K+2\), and let \(\sigma\) be a \(C_Q\)-invariant multilayer with finite positive conductivities whose complete \(K\)-CGPT block vanishes.
If the unscaled whole-space incident field is a harmonic polynomial of degree at most \(K\), then
\begin{equation}
 u(x)-H(x)=O(|x|^{-(Q-K)})
 \qquad(|x|\to\infty).
 \label{eq:polynomial-decay}
\end{equation}
For \(Q=2K+2\), the decay is \(O(|x|^{-(K+2)})\).
Unlike the small-inclusion estimate of Theorem~\ref{thm:arbitrary-H}, this fixed-scale conclusion is restricted to harmonic polynomials of degree at most \(K\).
\end{corollary}

Theorem~\ref{thm:arbitrary-H} and Corollary~\ref{cor:polynomial} are proved in Section~\ref{sec:harmonic-backgrounds}.

\section{An independent scalar case: global weak neutrality with one coating}
\label{sec:global-one-shell}

This section proves Theorem~\ref{thm:global-one-shell} independently of the constructions based on common conformal level curves and the perturbative results.
The argument first reduces the polarization tensor to a scalar response, then proves strict monotonicity, and finally establishes opposite signs at the insulating and perfectly conducting limits.
The intermediate value theorem then yields the unique coating conductivity.
After translating the common center of rotation to the origin, set \(S:=\Omega\setminus\overline D\) and define
\begin{equation}
	\kappa:=\frac{\sigma_{\rm c}}{\sigma_{\rm m}},
	\qquad
	t:=\frac{\sigma_{\rm sh}}{\sigma_{\rm m}}.
	\label{eq:one-shell-normalized-parameters}
\end{equation}
For \(g\in\mathbb{R}^2\), let \(H_g(x)=g\cdot x\), and let \(u_t^g\) solve the normalized transmission problem with conductivity
\begin{equation}
	\widehat\sigma_t(x)=
	\begin{cases}
		\kappa,&x\in D,\\
		t,&x\in S,\\
		1,&x\in\mathbb{R}^2\setminus\overline\Omega.
	\end{cases}
	\label{eq:normalized-one-shell}
\end{equation}
Define its polarization tensor by
\begin{equation}
	\mathsf M(t)
	:=\bigl(M_{11}^{\alpha\beta}[\widehat\sigma_t]\bigr)_
	{\alpha,\beta\in\{c,s\}}.
	\label{eq:one-shell-Mt}
\end{equation}
By definition, the normalized structure is weakly neutral to every uniform incident field exactly when \(\mathsf M(t)=0\).
The next proposition reduces this matrix equation to a scalar equation.

\begin{proposition}\label{prop:one-shell-scalarization}
	For every \(t>0\),
	\begin{equation}
		\mathsf M(t)=\mu(t)I_2
		\label{eq:one-shell-scalar-PT}
	\end{equation}
	for a real scalar function \(\mu\).
\end{proposition}

\begin{proof}
	Rotational covariance and the identities \(RD=D\) and \(R\Omega=\Omega\) give
	\[
	\mathsf M(t)=R\mathsf M(t)R^{\mathsf T}.
	\]
	Thus \(\mathsf M(t)\) commutes with \(R=R_{2\pi/Q}\).
	Because \(Q\geq3\), the rotation angle is neither \(0\) nor \(\pi\), and every real matrix commuting with \(R\) has the form \(aI_2+bJ\), where \(J\) represents rotation through \(\pi/2\).
	Reciprocity \eqref{eq:reciprocity} makes \(\mathsf M(t)\) symmetric, so \(b=0\).
\end{proof}

By \eqref{eq:one-shell-scalar-PT}, it remains to determine whether \(\mu(t)\) crosses zero and whether the crossing is unique.
The following sensitivity formula provides the required monotonicity.

\begin{proposition}\label{prop:one-shell-monotonicity}
	The map \(t\mapsto\mathsf M(t)\) is locally real analytic on \((0,\infty)\), and
	\begin{equation}
		g^{\mathsf T}\mathsf M'(t)h
		=\int_S\nabla u_t^g\cdot\nabla u_t^h\,\mathrm{d} x
		\qquad(g,h\in\mathbb{R}^2).
		\label{eq:one-shell-sensitivity}
	\end{equation}
	Consequently, for every unit vector \(g\),
	\begin{equation}
		\mu'(t)=\int_S|\nabla u_t^g|^2\,\mathrm{d} x>0.
		\label{eq:one-shell-strict-monotonicity}
	\end{equation}
\end{proposition}

\begin{proof}
	On every compact interval \(I\Subset(0,\infty)\), the variational operators for the correctors associated with \eqref{eq:normalized-one-shell} are uniformly coercive and depend affinely on \(t\).
	The Lax--Milgram theorem and the resolvent identity therefore give locally real analytic dependence of \(u_t^g-H_g\), and hence of the polarization tensor.
	Subtracting the weak formulations at \(t\) and \(q\) and canceling the symmetric cross terms gives the following variational identity for two conductivities, which is the case of the CGPT sensitivity formula corresponding to the polarization tensor \cite{ammari2014reconstruction}:
	\begin{equation}
		g^{\mathsf T}\bigl(\mathsf M(t)-\mathsf M(q)\bigr)h
		=(t-q)\int_S\nabla u_t^g\cdot\nabla u_q^h\,\mathrm{d} x.
		\label{eq:one-shell-difference-identity}
	\end{equation}
	Letting \(q\to t\) proves \eqref{eq:one-shell-sensitivity}.
	
	If the integral in \eqref{eq:one-shell-strict-monotonicity} vanished, then \(u_t^g\) would be constant in the connected coating \(S\).
	Its tangential derivative and its conormal derivative from the coating side would therefore vanish on \(\partial\Omega\).
	The transmission conditions give the same constant Dirichlet data and zero Neumann data on the exterior side.
	Extending the exterior solution by this constant into \(S\) gives a weakly harmonic function on \(\mathbb R^2\setminus\overline D\) because the Dirichlet and normal traces match across \(\partial\Omega\).
	The extended function is constant on the open set \(S\), so unique continuation makes it constant throughout \(\mathbb R^2\setminus\overline D\).
	This contradicts \(u_t^g(x)=g\cdot x+O(|x|^{-1})\) for \(g\neq0\).
\end{proof}

Strict monotonicity shows that \(\mu\) has at most one zero.
To prove existence, it remains to compare the limits as \(t\downarrow0\) and \(t\to\infty\).
Since uniform ellipticity degenerates at these endpoints, we pass to an equation for the trace on the outer boundary.
Let \(\Gamma=\partial\Omega\), with normal \(\nu\) pointing out of \(\Omega\), and set
\begin{equation}
	X_\Gamma:=H^{1/2}(\Gamma)/\mathbb{R}.
	\label{eq:one-shell-trace-quotient}
\end{equation}
The quotient records boundary data modulo additive constants.
For \(f\in X_\Gamma\), let \(w_t^f\) minimize
\begin{equation}
	E_t(w):=
	\kappa\int_D|\nabla w|^2\,\mathrm{d} x
	+t\int_S|\nabla w|^2\,\mathrm{d} x
	\label{eq:one-shell-interior-energy}
\end{equation}
among functions in \(H^1(\Omega)\) whose trace on \(\Gamma\) represents \(f\).
Define the interior Dirichlet-to-Neumann operator \(\Lambda_t^{\rm int}:X_\Gamma\to X_\Gamma^*\) by
\begin{equation}
	\left\langle\Lambda_t^{\rm int}f,h\right\rangle
	:=\kappa\int_D\nabla w_t^f\cdot\nabla w_t^h\,\mathrm{d} x
	+t\int_S\nabla w_t^f\cdot\nabla w_t^h\,\mathrm{d} x.
	\label{eq:one-shell-interior-DN}
\end{equation}

\begin{lemma}\label{lem:one-shell-interior-estimates}
	There are constants \(c,C>0\), independent of \(t>0\), such that
	\begin{equation}
		ct\left\lVert f\right\rVert _{X_\Gamma}^2
		\leq\left\langle\Lambda_t^{\rm int}f,f\right\rangle,
		\qquad
		\left\lVert \Lambda_t^{\rm int}\right\rVert _{\mathcal L(X_\Gamma,X_\Gamma^*)}\leq Ct.
		\label{eq:one-shell-interior-estimates}
	\end{equation}
\end{lemma}

\begin{proof}
	Connectedness of \(S\), Poincar\'e's inequality modulo constants, and the trace theorem give
	\[
	\left\lVert f\right\rVert _{X_\Gamma} \leq C\inf_{c\in\mathbb R}\left\lVert w-c\right\rVert _{H^1(S)} \leq C\left\lVert \nabla w\right\rVert _{L^2(S)}
	\]
	for every admissible \(w\), which proves the lower bound after minimization and squaring.
	For the upper bound, choose the representative of \(f\) with zero mean.
	The trace extension theorem on the domain \(S\), which has two boundary components, provides \(V_f\in H^1(S)\) with trace \(f\) on \(\Gamma\), zero trace on \(\partial D\), and \(\left\lVert V_f\right\rVert _{H^1(S)}\leq C\left\lVert f\right\rVert _{X_\Gamma}\).
	Extending \(V_f\) by zero into \(D\) gives an admissible competitor with energy at most \(Ct\left\lVert f\right\rVert _{X_\Gamma}^2\).
	The bound for the quadratic form, together with polarization, gives the asserted operator norm estimate.
\end{proof}

To couple the interior response to the fixed exterior medium, let \(W_f\) be the unique exterior harmonic function with finite Dirichlet energy whose boundary trace represents \(f\) and which satisfies \(W_f(x)=O(|x|^{-1})\) as \(|x|\to\infty\), and define
\begin{equation}
	\Lambda^{\rm ext}f:=-\partial_\nu W_f.
	\label{eq:one-shell-exterior-DN}
\end{equation}
The exterior Dirichlet-to-Neumann operator is positive, self-adjoint, and coercive:
\begin{equation}
	c_{\rm e}\left\lVert f\right\rVert _{X_\Gamma}^2
	\leq\left\langle\Lambda^{\rm ext}f,f\right\rangle
	=\int_{\mathbb{R}^2\setminus\overline\Omega}|\nabla W_f|^2\,\mathrm{d} x
	\leq C_{\rm e}\left\lVert f\right\rVert _{X_\Gamma}^2.
	\label{eq:one-shell-exterior-coercivity}
\end{equation}

Let
\[
f_t=[u_t^g|_\Gamma],
\qquad
h_g=[H_g|_\Gamma].
\]
In the exterior,
\begin{equation}
	u_t^g=H_g+W_{f_t-h_g}.
	\label{eq:one-shell-exterior-representation}
\end{equation}
Combining the interior and exterior Dirichlet-to-Neumann maps, flux continuity becomes
\begin{equation}
	\bigl(\Lambda_t^{\rm int}+\Lambda^{\rm ext}\bigr)f_t=F_g,
	\qquad
	F_g:=g\cdot\nu+\Lambda^{\rm ext}h_g.
	\label{eq:one-shell-boundary-equation}
\end{equation}
Here \(F_g\in X_\Gamma^*\) because it annihilates constants, as \(\int_\Gamma g\cdot\nu\,\mathrm{d} s=0\).
This boundary equation controls both endpoint limits; the next proposition supplies the opposite polarization signs needed for existence.

\begin{proposition}\label{prop:one-shell-extreme-limits}
	Let \(\mathsf M_{\rm ins}(\Omega)\) and \(\mathsf M_{\rm cond}(\Omega)\) denote the polarization tensors for the insulating cavity \(\Omega\) and the perfectly conducting inclusion \(\Omega\), respectively, with the convention \eqref{eq:multipole-expansion}.
	Then
	\begin{align}
		\mathsf M(t)&\longrightarrow\mathsf M_{\rm ins}(\Omega)
		&&\text{as }t\downarrow0,
		\label{eq:one-shell-insulating-limit}\\
		\mathsf M(t)&\longrightarrow\mathsf M_{\rm cond}(\Omega)
		&&\text{as }t\to\infty.
		\label{eq:one-shell-conducting-limit}
	\end{align}
	More precisely,
	\begin{equation}
		\left\lVert \mathsf M(t)-\mathsf M_{\rm ins}\right\rVert \leq Ct,
		\qquad
		\left\lVert \mathsf M(t)-\mathsf M_{\rm cond}\right\rVert \leq\frac Ct.
		\label{eq:one-shell-limit-rates}
	\end{equation}
	For every \(g\neq0\),
	\begin{equation}
		g^{\mathsf T}\mathsf M_{\rm ins}g<0,
		\qquad
		g^{\mathsf T}\mathsf M_{\rm cond}g>0.
		\label{eq:one-shell-limit-signs}
	\end{equation}
\end{proposition}

\begin{proof}
	For \(t\downarrow0\), set \(f_0=(\Lambda^{\rm ext})^{-1}F_g\).
	The resolvent identity applied to \eqref{eq:one-shell-boundary-equation}, together with \eqref{eq:one-shell-interior-estimates} and \eqref{eq:one-shell-exterior-coercivity}, gives
	\begin{equation}
		\left\lVert f_t-f_0\right\rVert _{X_\Gamma}\leq Ct.
		\label{eq:one-shell-insulating-trace-rate}
	\end{equation}
	Define the exterior field by
	\[
	U_{\rm ins}^g:=H_g+W_{f_0-h_g}
	\]
	It satisfies \(\partial_\nu U_{\rm ins}^g=0\) on \(\Gamma\) and is therefore the solution for the insulating cavity.
	Moreover,
	\begin{equation}
		\left\lVert \nabla(u_t^g-U_{\rm ins}^g)\right\rVert _
		{L^2(\mathbb{R}^2\setminus\overline\Omega)}\leq Ct.
		\label{eq:one-shell-insulating-field-rate}
	\end{equation}
	No convergence inside the degenerating coating is needed.
	
	For \(t\to\infty\), test \eqref{eq:one-shell-boundary-equation} by \(f_t\).
	Since \(\|F_g\|_{X_\Gamma^*}<C\), positivity of \(\Lambda^{\rm ext}\) and Lemma~\ref{lem:one-shell-interior-estimates} give
	\[
	ct\left\lVert f_t\right\rVert _{X_\Gamma}^2 \leq\langle F_g,f_t\rangle \leq C\left\lVert f_t\right\rVert _{X_\Gamma},
	\qquad
	\left\lVert f_t\right\rVert _{X_\Gamma}\leq\frac Ct.
	\]
	Thus \(f_t\to0\) in \(X_\Gamma\), so the trace becomes asymptotically constant.
	Define the limiting exterior field by
	\[
	U_{\rm cond}^g:=H_g+W_{-h_g}
	\]
	It has constant trace on \(\Gamma\).
	Its perturbation has finite Dirichlet energy and no logarithmic term, and hence
	\[
	\int_\Gamma\partial_\nu U_{\rm cond}^g\,\mathrm{d} s=0.
	\]
	It is therefore the perfectly conducting solution, including the condition of zero total flux.
	Moreover,
	\begin{equation}
		\left\lVert \nabla(u_t^g-U_{\rm cond}^g)\right\rVert _
		{L^2(\mathbb{R}^2\setminus\overline\Omega)}\leq\frac Ct.
		\label{eq:one-shell-conducting-field-rate}
	\end{equation}
	
	Because \(F_g\) depends linearly on \(g\), the preceding estimates are uniform for \(|g|=1\).
	On any fixed circle enclosing \(\Omega\), harmonic estimates on the intervening annulus and the trace theorem bound the \(H^{1/2}\) trace by the exterior energy norm.
	The outgoing Fourier coefficient of order \(m=1\) is a bounded linear functional of this trace.
By the normalization \eqref{eq:multipole-expansion}, the coefficient is a linear combination of the entries of the polarization tensor.
	The preceding energy estimates therefore prove \eqref{eq:one-shell-limit-rates}.
	For \(U_{\rm ins}^g\), the Neumann condition eliminates the boundary term involving the total field; for \(U_{\rm cond}^g\), the same term vanishes because the trace is constant and the total flux is zero.
	Applying Green's identity on a truncated exterior domain and using the \(r^{-1}\) far-field behavior yields
	\begin{align}
		g^{\mathsf T}\mathsf M_{\rm ins}g
		&=-\int_{\partial \Omega}U_{\rm ins}^g\partial_\nu H_g\,\mathrm{d}s+\int_{\partial \Omega}\partial_\nu U_{\rm ins}^g H_g\,\mathrm{d}s,\\
		g^{\mathsf T}\mathsf M_{\rm cond}g
			&=-\int_{\partial \Omega}U_{\rm cond}^g\partial_\nu H_g\,\mathrm{d}s+\int_{\partial \Omega}\partial_\nu U_{\rm cond}^g H_g\,\mathrm{d}s.
		\label{eq:one-shell-conducting01}
	\end{align}
	Together with the boundary conditions, these identities yield
	\begin{align}
		g^{\mathsf T}\mathsf M_{\rm ins}g
		&=-|\Omega||g|^2
		-\int_{\mathbb{R}^2\setminus\overline\Omega}
		|\nabla(U_{\rm ins}^g-H_g)|^2\,\mathrm{d} x,
		\label{eq:one-shell-insulating-sign}\\
		g^{\mathsf T}\mathsf M_{\rm cond}g
		&=|\Omega||g|^2
		+\int_{\mathbb{R}^2\setminus\overline\Omega}
		|\nabla(U_{\rm cond}^g-H_g)|^2\,\mathrm{d} x.
		\label{eq:one-shell-conducting-sign}
	\end{align}
	This proves the strict signs.
\end{proof}

The scalarization, strict monotonicity, and endpoint signs now complete the proof of Theorem~\ref{thm:global-one-shell}.

\begin{proof}[Proof of Theorem~\ref{thm:global-one-shell}]
	Dividing all three conductivities by \(\sigma_{\rm m}\) leaves the physical solution unchanged, and \(\mathsf M(t)\) is the polarization tensor for the resulting normalized conductivity.
	Proposition~\ref{prop:one-shell-scalarization} gives \(\mathsf M(t)=\mu(t)I_2\).
	By Proposition~\ref{prop:one-shell-monotonicity}, \(\mu\) is continuous and strictly increasing on \((0,\infty)\), while Proposition~\ref{prop:one-shell-extreme-limits} gives
	\[
	\lim_{t\downarrow0}\mu(t)<0,
	\qquad
	\lim_{t\to\infty}\mu(t)>0.
	\]
	The intermediate value theorem gives a root \(t_*>0\), and strict monotonicity makes it unique.
	
	At \(t=1\), the coating has the same conductivity as the background, so the problem reduces to the single inclusion \(D\) with conductivity \(\kappa\).
	For $\lambda>0$, let $v_\lambda^g$ be the transmission solution for $D$ with conductivity $\lambda$ in a background of conductivity one.
	By the same scalarization argument, its polarization tensor has the form $\widetilde\mu(\lambda)I_2$.
	Since $\lambda=1$ gives the homogeneous configuration, $\widetilde\mu(1)=0$.
	The same sensitivity argument, now varying the conductivity in $D$, gives, for every unit vector $g$,
	\[
	\widetilde\mu'(\lambda)=\int_D|\nabla v_\lambda^g|^2\,\mathrm{d}x.
	\]
	Integrating from $1$ to $\kappa$ therefore yields
	\[
	\mu(1)=\widetilde\mu(\kappa)=\int_1^\kappa\int_D|\nabla v_\lambda^g|^2\,\mathrm{d}x\,\mathrm{d}\lambda.
	\]
	The integrand is strictly positive by the same unique continuation argument.
	Therefore
	\[
	\operatorname{sign}\mu(1)=\operatorname{sign}(\kappa-1).
	\]
	Hence \(\kappa>1\) implies \(0<t_*<1\), whereas \(0<\kappa<1\) implies \(t_*>1\).
	If \(\kappa=1\), then \(t=1\) is the homogeneous configuration and therefore the unique root.
	At \(t=t_*\), the normalized conductivity \(\widehat\sigma_{t_*}\) is exactly \(\sigma^*\) in \eqref{eq:global-one-shell-conductivity}, and \(\mathsf M(t_*)=0\) gives \eqref{eq:global-one-shell-neutrality}.
	Restoring physical units, \(\sigma_{\rm sh}^*=\sigma_{\rm m}t_*\), proves \eqref{eq:one-shell-root-below}--\eqref{eq:one-shell-root-above}.
\end{proof}

We next apply Theorem~\ref{thm:global-one-shell} to a geometry defined by the common conformal level curves introduced in Subsection~\ref{sec:cyclic-geometry}.

\begin{corollary}\label{cor:global-one-shell-common-map}
	Let \(Q\geq3\) and \(L\geq1\) be integers, let \(a_1,\ldots,a_L\in\mathbb{C}\), and let \(0<R_{\rm i}<R_{\rm e}\).
	Define
	\begin{equation}
		\Phi(w)=w+\sum_{\ell=1}^{L}a_\ell w^{1-\ell Q}.
		\label{eq:one-shell-common-map}
	\end{equation}
	If
	\begin{equation}
		\sum_{\ell=1}^{L}(\ell Q-1)|a_\ell|R_{\rm i}^{-\ell Q}<1,
		\label{eq:one-shell-common-map-univalence}
	\end{equation}
	then the images of \(|w|=R_{\rm i}\) and \(|w|=R_{\rm e}\) enclose a nested \(C_Q\)-symmetric core--shell pair.
	For every pair of finite positive conductivities \(\sigma_{\rm c}\) and \(\sigma_{\rm m}\), there exists a unique constant coating conductivity \(\sigma_{\rm sh}^*\in(0,\infty)\) for which the coated inclusion is weakly neutral to every uniform incident field.
\end{corollary}

\begin{proof}
	The proof of Lemma~\ref{lem:global-injectivity}, with inner radius \(R_{\rm i}\), gives global injectivity on \(\{|w|\geq R_{\rm i}\}\).
	The two image curves are therefore simple and nested.
	The identity \(\Phi(\omega w)=\omega\Phi(w)\), where \(\omega=\mathrm{e}^{2\pi i/Q}\), gives their common \(C_Q\) symmetry, so Theorem~\ref{thm:global-one-shell} applies.
	The conformal map generates only the geometry, and the transmission problem is solved again on the image domains.
\end{proof}

We finally combine weak neutrality with the selection rules of Proposition~\ref{prop:selection} to obtain the additional multipole suppression.

\begin{proof}[Proof of Corollary~\ref{cor:one-shell-uniform-suppression}]
	Applying Proposition~\ref{prop:selection} to incident order \(n=1\), the complex CGPT \(\mathbb N_{m1}^{(1)}\) can be nonzero only when \(m+1\equiv0\pmod Q\), while \(\mathbb N_{m1}^{(2)}\) can be nonzero only when \(m-1\equiv0\pmod Q\).
	Weak neutrality cancels the \(m=1\) term in \(\mathbb N_{m1}^{(2)}\), so the first possible surviving outgoing order is \(m=Q-1\).
	The multipole expansion \eqref{eq:multipole-expansion} gives \eqref{eq:one-shell-uniform-decay}.
	Thus all outgoing orders \(1\leq m\leq Q-2\) vanish, proving the final assertion.
\end{proof}

The scalar argument is now complete.
We next turn to the local higher-order constructions under cyclic symmetry.

\section{Proofs of the local \texorpdfstring{$C_Q$}{CQ}-symmetric constructions}
\label{sec:local-constructions}

We first prove the fixed-geometry theorem and its homothetic corollary, then the small-core theorem, and finally the joint shape--material continuation and its linearized correction stated in Section~\ref{sec:main-results}.

\subsection{Fixed geometries near the homogeneous state}
\label{subsec:proof-fixed-geometry}

For the conductivity \eqref{eq:fixed-target-conductivity} and \(1\leq n\leq K\), define the reduced response map by
\begin{equation}
 \mathcal G_n^{\mathcal D}(\delta,\eta)
 :=\frac{1}{2\pi n}
 \mathbb N_{nn}^{(2)}[\sigma_{\mathcal D,\delta,\eta}],
 \qquad
 \boldsymbol{\mathcal G}^{\mathcal D}
 :=(\mathcal G_1^{\mathcal D},\ldots,\mathcal G_K^{\mathcal D}).
 \label{eq:fixed-target-response}
\end{equation}
Under the standing assumption \(Q>2K\), Proposition~\ref{prop:selection} makes complete \(K\)-CGPT cancellation equivalent to \(\boldsymbol{\mathcal G}^{\mathcal D}(\delta,\eta)=0\).
We first establish real-analytic dependence on \((\delta,\eta)\) and then compute the material derivatives at the homogeneous point.

\begin{proof}[Proof of Theorem~\ref{thm:fixed-target}]
At the homogeneous reference point, \(\boldsymbol{\mathcal G}^{\mathcal D}(0,0)=0\).
For the incident harmonic polynomial \(P_n^\beta\), write
\[
 u_n^\beta=P_n^\beta+w_n^\beta.
\]
Here
\[
 \dot H^1(\mathbb{R}^2) :=\{w\in H^1_{\rm loc}(\mathbb{R}^2):\nabla w\in L^2(\mathbb{R}^2)\}/\mathbb{R}.
\]
The scattered field \(w_n^\beta\in\dot H^1(\mathbb{R}^2)\) satisfies
\begin{equation}
 \int_{\mathbb{R}^2}\sigma_{\mathcal D,\delta,\eta}
 \nabla w_n^\beta\cdot\nabla\varphi\,\mathrm{d} x
 =-\int_{D_K}(\sigma_{\mathcal D,\delta,\eta}-1)
 \nabla P_n^\beta\cdot\nabla\varphi\,\mathrm{d} x.
 \label{eq:fixed-target-corrector}
\end{equation}
Equip \(\dot H^1(\mathbb R^2)\) with the gradient norm, or equivalently choose the representative having zero mean on a fixed ball containing \(\overline{D_K}\).
The left-hand side of \eqref{eq:fixed-target-corrector} defines a bounded variational operator \(\mathscr L_{\delta,\eta}:\dot H^1(\mathbb R^2)\to\dot H^1(\mathbb R^2)^*\) whose coercivity constant is bounded away from zero in a material neighborhood of the origin.
Denote the right-hand side by \(F_{\delta,\eta}\); it is a bounded functional because it is supported in the fixed set \(D_K\).
Both \(\mathscr L_{\delta,\eta}\) and \(F_{\delta,\eta}\) depend real analytically on the parameters in their respective operator norms.
Since \(\mathscr L_{0,0}\) is an isomorphism, factor \(\mathscr L_{\delta,\eta}=\mathscr L_{0,0}[I+\mathscr L_{0,0}^{-1}(\mathscr L_{\delta,\eta}-\mathscr L_{0,0})]\).
After shrinking the parameter neighborhood, the second factor is inverted by its Neumann series, so \(w_n^\beta=\mathscr L_{\delta,\eta}^{-1}F_{\delta,\eta}\) is real analytic in \((\delta,\eta)\).
Inserting \(u_n^\beta=P_n^\beta+w_n^\beta\) into \eqref{eq:real-cgpt} proves the same real-analytic dependence for the CGPTs and \(\boldsymbol{\mathcal G}^{\mathcal D}\).

At \((\delta,\eta)=(0,0)\), the medium is homogeneous and \(u_n^\beta=P_n^\beta\).
Differentiating \eqref{eq:real-cgpt} at this point gives
\[
 \left.\partial_{\eta_j}M_{mn}^{\alpha\beta}
 \right|_{(0,0)}
 =\int_{\mathcal A_j}
 \nabla P_n^\beta\cdot\nabla P_m^\alpha\,\mathrm{d} x.
\]
Indeed, the term containing \(\partial_{\eta_j}u_n^\beta\) is multiplied by \(\sigma_{\mathcal D,\delta,\eta}-1\) and therefore vanishes at the homogeneous reference point.
The pointwise identities
\[
 |\nabla P_n^c|^2
 =|\nabla P_n^s|^2
 =n^2|z|^{2n-2},
 \qquad
 \nabla P_n^c\cdot\nabla P_n^s=0
\]
show that the two diagonal derivatives are equal and the two cross derivatives vanish.
Using the normalization in \eqref{eq:N2} and \eqref{eq:fixed-target-response}, we obtain
\begin{align*}
 \left.\partial_{\eta_j}\mathcal G_n^{\mathcal D}
 \right|_{(0,0)}
 &=\frac{n}{\pi}\int_{\mathcal A_j}|z|^{2n-2}\,\mathrm{d} x
 =J^{\mathcal D}_{nj},\\
 \left.\partial_\delta\mathcal G_n^{\mathcal D}
 \right|_{(0,0)}
 &=\frac{n}{\pi}\int_{D_0}|z|^{2n-2}\,\mathrm{d} x
 =v_n^{\mathcal D}.
\end{align*}
These are precisely the matrix and vector defined in \eqref{eq:fixed-target-J}--\eqref{eq:fixed-target-v}, so
\[
 D_\eta\boldsymbol{\mathcal G}^{\mathcal D}(0,0)=J^{\mathcal D},
 \qquad
 D_\delta\boldsymbol{\mathcal G}^{\mathcal D}(0,0)=v^{\mathcal D}.
\]
The analytic implicit-function theorem and \eqref{eq:fixed-target-nondegeneracy} now give the unique local branch \eqref{eq:fixed-target-branch}.
Differentiating
\[
 \boldsymbol{\mathcal G}^{\mathcal D} (\delta,\eta^{\mathcal D}(\delta))=0
\]
at \(\delta=0\) gives \(v^{\mathcal D}+J^{\mathcal D}(\eta^{\mathcal D})'(0)=0\), and hence \eqref{eq:fixed-target-expansion}.
The exponential parameterization gives finite positive conductivities, and the reduction preceding the proof yields the complete cancellation statement \eqref{eq:fixed-target-cancel}.
\end{proof}

For the homothetic class, it remains to factor the moment Jacobian into a positive diagonal matrix and a difference--Vandermonde matrix.

\begin{proof}[Proof of Corollary~\ref{cor:homothetic}]
Strict star-shapedness and \(0<\rho_0<\cdots<\rho_K\) make the domains in \eqref{eq:homothetic-domains} strictly nested.
Put
\[
 \mu_n:=\int_D|y|^{2n-2}\,\mathrm{d} y.
\]
The change of variables \(z=\rho_jy\) gives
\[
 \int_{D_j}|z|^{2n-2}\,\mathrm{d} x
 =\rho_j^{2n}\int_D|y|^{2n-2}\,\mathrm{d} y
 =\mu_nx_j^n.
\]
Subtracting consecutive scaled-domain moments gives
\begin{equation}
 J^{\mathcal D}
 =\operatorname{diag}\left(\frac{n\mu_n}{\pi}\right)_{n=1}^K B,
 \qquad
 v^{\mathcal D}
 =\operatorname{diag}\left(\frac{n\mu_n}{\pi}\right)_{n=1}^K b.
 \label{eq:homothetic-matrix}
\end{equation}
For \(I_j=(x_{j-1},x_j)\), one has
\[
 B_{nj}=\int_{I_j}nt^{n-1}\,\mathrm{d} t.
\]
Multilinearity of the determinant and Fubini's theorem yield
\begin{align*}
 \det B
 &=\int_{I_1\times\cdots\times I_K}
   \det[nt_j^{n-1}]_{n,j=1}^{K}
 \,\mathrm{d} t_1\cdots\,\mathrm{d} t_K\\
 &=\left(\prod_{n=1}^{K}n\right)
   \int_{I_1\times\cdots\times I_K}
   \prod_{1\leq i<j\leq K}(t_j-t_i)
   \,\mathrm{d} t_1\cdots\,\mathrm{d} t_K>0.
\end{align*}
The intervals are strictly ordered, so \(t_1<\cdots<t_K\) almost everywhere and the determinant is positive.
Since every \(\mu_n\) is positive, \eqref{eq:homothetic-matrix} gives \(\det J^{\mathcal D}>0\), proving \eqref{eq:homothetic-positive-determinant}.
The common positive diagonal factor in \(J^{\mathcal D}\) and \(v^{\mathcal D}\) cancels in \eqref{eq:fixed-target-expansion}, proving \eqref{eq:homothetic-expansion}.
\end{proof}

We next turn to the small-core regime, where the core material type is fixed and its geometry shrinks.

\subsection{Small cores of fixed material type}
\label{subsec:proof-small-core}

We use the regions \eqref{eq:small-core-regions}--\eqref{eq:small-core-limit-regions} and the conductivity \eqref{eq:small-core-conductivity}.
The limiting matrix \(J^0\) and reference responses \(g_n^{\mathcal B,\kappa}\) are defined in \eqref{eq:small-core-limit-J}--\eqref{eq:small-core-reference-response}.

Throughout this subsection, \(C_\kappa\) and \(O_\kappa(\,\cdot\,)\) may depend on the fixed parameters, domains, interface regularity, nesting distances, and size bounds, but not on \(\varepsilon\).
At \(\kappa=0,\infty\), this notation means dependence on the fixed endpoint boundary problem.
No uniformity over the extended interval \([0,\infty]\) is asserted.

At \(\varepsilon=0\), the core disappears and the first coating fills \(D_1\).
For \(0\leq\varepsilon\) and \(1\leq n\leq K\), define
\begin{equation}
 \mathcal G_{\varepsilon,n}(\eta)
 :=\frac{1}{2\pi n}\mathbb N_{nn}^{(2)}
 [\sigma_{\varepsilon,\eta}],
 \qquad
 \boldsymbol{\mathcal G}_\varepsilon
 :=(\mathcal G_{\varepsilon,1},\ldots,
       \mathcal G_{\varepsilon,K}),
 \label{eq:small-core-response}
\end{equation}
where at \(\varepsilon=0\) the limiting conductivity assigns \(\mathrm{e}^{\eta_1}\) to all of \(D_1\), retains the remaining coatings, and has unit background conductivity.

The first three lemmas establish the endpoint framework and cyclic selection, identify the limiting material Jacobian, and transfer it uniformly to small positive \(\varepsilon\) while controlling the forcing at \(\eta=0\).
The final lemma is used only for the endpoint lower bound in \eqref{eq:small-core-two-sided-bound}.

\begin{lemma}\label{lem:small-core-endpoint-framework}
Put
\[
 \Omega_\varepsilon :=\mathbb{R}^2\setminus\overline{\varepsilon\mathcal B},
 \qquad
 \Gamma_\varepsilon :=\partial(\varepsilon\mathcal B),
\]
and define on \(\Omega_\varepsilon\)
\[
 a_{\varepsilon,\eta} :=\sum_{j=1}^{K}\mathrm{e}^{\eta_j} \chi_{\mathcal A_j^\varepsilon} +\chi_{\mathbb{R}^2\setminus\overline{D_K}}.
\]
For \(\kappa=0\) and \(\kappa=\infty\), the scattering problem is well posed in the finite-energy class, its CGPTs are reciprocal, and the selection rules \eqref{eq:selection1}--\eqref{eq:selection2} hold.
In particular, if \(Q>2K\), the complete \(K\)-CGPT block vanishes if and only if
\[
 \mathcal G_{\varepsilon,n}(\eta)=0,
 \qquad 1\leq n\leq K.
\]
\end{lemma}

\begin{proof}
Use the Beppo--Levi space
\[
 \dot H^1(\Omega_\varepsilon) := \{v\in H^1_{\rm loc}(\Omega_\varepsilon): \nabla v\in L^2(\Omega_\varepsilon)\}/\mathbb{R}
\]
and its closed subspace
\[
 \mathcal V_\varepsilon := \{v\in\dot H^1(\Omega_\varepsilon): \operatorname{Tr}_{\Gamma_\varepsilon}v \text{ is constant}\}.
\]
Both carry the gradient norm.
Write
\[
 A_{\varepsilon,\eta}(v,\varphi) :=\int_{\Omega_\varepsilon}a_{\varepsilon,\eta} \nabla v\cdot\nabla\varphi\,\mathrm{d} x.
\]
For an incident harmonic polynomial \(H\), the insulating scattering correction \(w=u-H\) is the unique element of \(\dot H^1(\Omega_\varepsilon)\) satisfying
\begin{align}
 A_{\varepsilon,\eta}(w,\varphi)
 ={}&-\sum_{j=1}^{K}(\mathrm{e}^{\eta_j}-1)
       \int_{\mathcal A_j^\varepsilon}
       \nabla H\cdot\nabla\varphi\,\mathrm{d} x\notag\\
 &+\left\langle\partial_\nu H,
       \operatorname{Tr}_{\Gamma_\varepsilon}\varphi
   \right\rangle,
 \qquad
 \varphi\in\dot H^1(\Omega_\varepsilon).
 \label{eq:small-core-insulating-weak}
\end{align}
The boundary functional is independent of the representative because \(\int_{\Gamma_\varepsilon}\partial_\nu H\,\mathrm{d} s=0\), and the trace theorem makes it bounded in the gradient norm.

For the perfectly conducting problem, let
\[
 \mathcal W_{\varepsilon,H} := \{w\in\dot H^1(\Omega_\varepsilon): \operatorname{Tr}_{\Gamma_\varepsilon}(H+w) \text{ is constant}\}.
\]
The correction is the unique \(w\in\mathcal W_{\varepsilon,H}\) satisfying
\begin{equation}
 A_{\varepsilon,\eta}(w,\varphi)
 =-\sum_{j=1}^{K}(\mathrm{e}^{\eta_j}-1)
       \int_{\mathcal A_j^\varepsilon}
       \nabla H\cdot\nabla\varphi\,\mathrm{d} x,
 \qquad
 \varphi\in\mathcal V_\varepsilon.
 \label{eq:small-core-conducting-weak}
\end{equation}
Indeed, a cutoff equal to one near \(\mathcal B\) gives
\[
 \zeta_{\varepsilon,H}(x) :=\chi(x/\varepsilon)\bigl(H(0)-H(x)\bigr) \in\mathcal W_{\varepsilon,H}.
\]
Writing \(w=\zeta_{\varepsilon,H}+v\) reduces \eqref{eq:small-core-conducting-weak} to a coercive problem on \(\mathcal V_\varepsilon\).
Lax--Milgram proves existence and uniqueness in both endpoint cases.
Testing the conducting equation with a compactly supported admissible function whose core trace is one, and then integrating by parts, gives
\[
 \int_{\Gamma_\varepsilon} \mathrm{e}^{\eta_1}\partial_\nu u\,\mathrm{d} s=0.
\]
Thus \eqref{eq:small-core-conducting-weak} is exactly the perfectly conducting problem.
Notice that \eqref{eq:small-core-insulating-weak} and \eqref{eq:small-core-conducting-weak} contain only bounded-support pairings with \(H\); no integral of \(\nabla H\) against an arbitrary finite-energy test function over the unbounded domain is used.
Outside \(D_K\), the correction is a finite-energy harmonic function.
Its logarithmic coefficient is therefore zero; subtracting its constant term gives the normalization \(w=O(|x|^{-1})\).

For either endpoint and harmonic polynomials \(H,L\), define the response pairing from the normalized far field by
\begin{equation}
 \mathscr M_{\varepsilon,\eta}^{\kappa}(H,L)
 :=
 \lim_{R\to\infty}\int_{\partial B_R}
 \bigl(L\partial_r w_H-w_H\partial_rL\bigr)\,\mathrm{d} s.
 \label{eq:small-core-endpoint-response-pairing}
\end{equation}
Green's identity on \(\Omega_\varepsilon\cap B_R\), with inner normal \(-\nu\), shows that for \(H=P_n^\beta\) and \(L=P_m^\alpha\) this limit equals the coefficient \(M_{mn}^{\alpha\beta}\) in \eqref{eq:multipole-expansion}; the core term vanishes by the endpoint boundary condition.
On any fixed annulus outside \(D_K\), say \(B_{2R}\setminus\overline{B_R}\) with \(R\) sufficiently large, the trace theorem, Poincar\'e's inequality modulo constants, and Fourier projection give
\[
 |\mathscr M_{\varepsilon,\eta}^{\kappa}(H,L)| \leq C_L\|\nabla w_H\|_{L^2(B_{2R}\setminus\overline{B_R})},
\]
so the response is a bounded functional of the energy solution.

Let \(a,\widetilde a\) be two positive coating coefficients with the same endpoint core type, let us write $\mathscr M_{\varepsilon,\eta}^{\kappa}$ as $\mathscr M_{a}^{\kappa}$ for short.
Green's identity on \(\Omega_\varepsilon\cap B_R\) first gives
\[
 \int_{\Omega_\varepsilon}(a-\widetilde a) \nabla u_H^{a,\kappa}\cdot \nabla u_L^{\widetilde a,\kappa}\,\mathrm{d} x = \mathscr M_a^\kappa(H,L) -\mathscr M_{\widetilde a}^\kappa(L,H).
\]
The core terms vanish under the two Neumann conditions when \(\kappa=0\).
When \(\kappa=\infty\), each core trace is constant and the remaining factor is the zero total flux.
Taking \(a=\widetilde a\) proves reciprocity; using it in the preceding identity yields
\begin{equation}
 \mathscr M_a^\kappa(H,L)
 -\mathscr M_{\widetilde a}^\kappa(H,L)
 =
 \int_{\Omega_\varepsilon}(a-\widetilde a)
 \nabla u_H^{a,\kappa}\cdot
 \nabla u_L^{\widetilde a,\kappa}\,\mathrm{d} x.
 \label{eq:small-core-endpoint-difference}
\end{equation}

Both endpoint boundary conditions are preserved by the \(C_Q\) action.
Rotating the solution and using uniqueness therefore gives the two phase relations in the proof of Proposition~\ref{prop:selection}, hence \eqref{eq:selection1}--\eqref{eq:selection2}.
On the diagonal, reciprocity gives \(M_{nn}^{cs}=M_{nn}^{sc}\), while \(\mathbb N_{nn}^{(1)}=0\) gives \(M_{nn}^{cc}=M_{nn}^{ss}\) and \(M_{nn}^{cs}=-M_{nn}^{sc}\).
Thus the cross terms vanish and
\[
 \mathbb N_{nn}^{(2)}=2M_{nn}^{cc}\in\mathbb{R}.
\]
The same counting argument as in Proposition~\ref{prop:selection} now proves the final assertion.

The proof is complete.
\end{proof}

The difference identity \eqref{eq:small-core-endpoint-difference} yields the coating derivative at the two endpoints.
The same formula holds for finite \(\kappa\) and for the limiting problem at \(\varepsilon=0\).

\begin{lemma}\label{lem:small-core-sensitivity}
For each fixed \(\kappa\in[0,\infty]\) and \(\varepsilon\geq0\), the map \(\eta\mapsto\boldsymbol{\mathcal G}_\varepsilon(\eta)\) is real analytic near the origin.
Let \(u_{\varepsilon,n}^c(\eta)\) be the corresponding solution generated by \(P_n^c\).
Then
\begin{equation}
 \partial_{\eta_j}\mathcal G_{\varepsilon,n}(\eta)
 =\frac{\mathrm{e}^{\eta_j}}{\pi n}
  \int_{\mathcal A_j^\varepsilon}
  |\nabla u_{\varepsilon,n}^c(\eta)|^2\,\mathrm{d} x,
 \qquad 1\leq n,j\leq K,
 \label{eq:small-core-sensitivity-formula}
\end{equation}
where at \(\varepsilon=0\) the region \(\mathcal A_j^\varepsilon\) is replaced by \(\mathcal A_j^0\).
In particular,
\begin{equation}
 D_\eta\boldsymbol{\mathcal G}_0(0)=J^0.
 \label{eq:small-core-limit-derivative}
\end{equation}
\end{lemma}

\begin{proof}
For \(0<\kappa<\infty\), the coercive variational operator depends real analytically on \(\eta\) in operator norm.
Factoring it at \(\eta=0\) and inverting the resulting perturbation by a Neumann series proves real-analytic dependence of the solution and hence of the response.
If a uniformly positive conductivity \(\sigma\) is varied in the direction \(q\), differentiating \eqref{eq:real-cgpt} gives
\begin{equation}
 D M^{\alpha\beta}_{mn}[\sigma](q)
 =\int_{\mathbb{R}^2}q\,
   \nabla u_n^\beta\cdot\nabla u_m^\alpha\,\mathrm{d} x.
 \label{eq:material-derivative-general}
\end{equation}
Indeed, the derivative of \(u_n^\beta\) solves the linearized weak equation with right-hand side \(-\int q\nabla u_n^\beta\cdot\nabla\varphi\), and testing this equation against \(u_m^\alpha-P_m^\alpha\) converts the field-derivative term in \eqref{eq:real-cgpt} into the missing part of \(\nabla u_m^\alpha\).

For \(\kappa=0,\infty\), the spaces in \eqref{eq:small-core-insulating-weak} and \eqref{eq:small-core-conducting-weak} do not depend on \(\eta\), and their coercive operators depend real analytically on \(\eta\) in operator norm.
Analytic inversion therefore gives analytic dependence of the energy solution.
The fixed-annulus estimate in the proof of Lemma~\ref{lem:small-core-endpoint-framework} transfers this analyticity to every far-field coefficient.
Differentiating \eqref{eq:small-core-endpoint-difference} gives the endpoint analogue of \eqref{eq:material-derivative-general}, with the integral taken over \(\Omega_\varepsilon\).
Since the material variation is supported in the coatings, the finite-conductivity and endpoint problems have the same coating derivative.

For \(n\leq K\), the selection rules and reciprocity give
\[
 \mathbb N_{nn}^{(1)}=0,
 \qquad
 M_{nn}^{cc}=M_{nn}^{ss},
 \qquad
 M_{nn}^{cs}=M_{nn}^{sc}=0,
\]
and hence
\[
 \mathcal G_{\varepsilon,n}=\frac{1}{\pi n}M_{nn}^{cc}.
\]
Taking \(q=\mathrm{e}^{\eta_j}\chi_{\mathcal A_j^\varepsilon}\) in \eqref{eq:material-derivative-general}, or in its endpoint analogue, proves \eqref{eq:small-core-sensitivity-formula}.
At \((\varepsilon,\eta)=(0,0)\), the medium is homogeneous and \(u_{0,n}^c=P_n^c\).
Since \(|\nabla P_n^c|^2=n^2|z|^{2n-2}\), the same formula gives \eqref{eq:small-core-limit-derivative}.
\end{proof}

At \(\varepsilon=0\) the core is absent, and \eqref{eq:small-core-limit-derivative} identifies the candidate limiting Jacobian.
Because the endpoint domains vary with \(\varepsilon\), the required stability is a uniform \(C^1\) estimate rather than an implicit-function argument posed directly at the degenerate geometry.

\begin{lemma}\label{lem:small-core-C1-limit}
There exist \(\rho>0\), \(\varepsilon_0>0\), and \(C_\kappa>0\) such that, for every \(0<\varepsilon<\varepsilon_0\),
\begin{align}
 \sup_{|\eta|\leq\rho}
 |\boldsymbol{\mathcal G}_\varepsilon(\eta)
  -\boldsymbol{\mathcal G}_0(\eta)|
 &\leq C_\kappa\varepsilon,
 \label{eq:small-core-C0-convergence}\\
 \sup_{|\eta|\leq\rho}
 \|D_\eta\boldsymbol{\mathcal G}_\varepsilon(\eta)
  -D_\eta\boldsymbol{\mathcal G}_0(\eta)\|
 &\leq C_\kappa\varepsilon.
 \label{eq:small-core-C1-convergence}
\end{align}
Moreover, the response at the homogeneous coating vector satisfies the exact identities
\begin{equation}
 \mathcal G_{\varepsilon,n}(0)
 =\varepsilon^{2n}g_n^{\mathcal B,\kappa},
 \qquad 1\leq n\leq K.
 \label{eq:small-core-exact-forcing}
\end{equation}
In particular,
\begin{equation}
 |\boldsymbol{\mathcal G}_\varepsilon(0)|
 \leq C_\kappa\varepsilon^2.
 \label{eq:small-core-forcing-bound}
\end{equation}
\end{lemma}

\begin{proof}
We first prove the exact scaling at \(\eta=0\).
All coatings and the background then have conductivity one, so the only contrast is \(\varepsilon\mathcal B\).
If \(U_n^\beta\) is the corresponding finite-conductivity or endpoint solution for the reference core \(\mathcal B\), homogeneity of \(P_n^\beta\) gives
\[
 u_{\varepsilon,n}^\beta(x) =\varepsilon^nU_n^\beta(x/\varepsilon).
\]
For \(0<\kappa<\infty\), a change of variables in \eqref{eq:real-cgpt} gives the factor \(\varepsilon^{m+n}\) for \(M_{mn}^{\alpha\beta}\).
The same factor follows at \(\kappa=0,\infty\) from \eqref{eq:small-core-endpoint-response-pairing}; both endpoint boundary conditions and the zero-total-flux constraint are preserved by the scaling.
Hence
\[
 M_{mn}^{\alpha\beta}[\sigma_{\varepsilon,0}] =\varepsilon^{m+n} M_{mn}^{\alpha\beta}[\mathcal B;\kappa].
\]
This proves \eqref{eq:small-core-exact-forcing} and \eqref{eq:small-core-forcing-bound}.

We next prove the uniform Jacobian convergence.
Fix \(\rho>0\) small enough that the material ball under consideration lies in the analyticity neighborhood.
On the coatings and the background, the coefficients have common positive lower and upper bounds for \(|\eta|\leq\rho\).
The finite-conductivity problem and the two coercive endpoint formulations therefore satisfy uniform energy estimates.

Choose \(r_*>0\) such that \(\overline{B_{4r_*}}\subset D_1\).
Fix a ball \(U\) centered at the origin with \(\overline{\mathcal B}\Subset U\), and decrease \(\varepsilon_0\) so that \(\varepsilon\overline U\subset B_{r_*}\) for every \(0<\varepsilon<\varepsilon_0\).
The limiting field \(u_{0,n}^c(\eta)\) is harmonic in \(B_{4r_*}\), because the limiting conductivity is constant there.
Interior estimates and the uniform energy bound give
\begin{equation}
 \sup_{|\eta|\leq\rho}
 \|\nabla u_{0,n}^c(\eta)\|_{L^\infty(B_{2r_*})}
 \leq C_\kappa.
 \label{eq:small-core-interior-gradient}
\end{equation}
Set \(v_{\varepsilon,n}(\eta):=u_{\varepsilon,n}^c(\eta)-u_{0,n}^c(\eta)\), with the difference restricted to \(\Omega_\varepsilon\) for the endpoint problems.

If \(0<\kappa<\infty\), subtracting the two full-space weak equations gives
\[
 \int_{\mathbb{R}^2}\sigma_{\varepsilon,\eta}
 \nabla v_{\varepsilon,n}\cdot\nabla\varphi\,\mathrm{d} x
 =\int_{\varepsilon\mathcal B}
 (\mathrm{e}^{\eta_1}-\kappa)
 \nabla u_{0,n}^c(\eta)\cdot\nabla\varphi\,\mathrm{d} x.
\]
Testing with \(v_{\varepsilon,n}\) and using \eqref{eq:small-core-interior-gradient} gives the required \(O_\kappa(\varepsilon)\) energy bound.

For \(\kappa=0\), subtraction of the insulating weak problem and the restriction of the limiting full-space problem gives
\begin{equation}
 A_{\varepsilon,\eta}(v_{\varepsilon,n},\varphi)
 =\mathrm{e}^{\eta_1}
 \left\langle\partial_\nu u_{0,n}^c,
   \operatorname{Tr}_{\Gamma_\varepsilon}\varphi\right\rangle.
 \label{eq:small-core-insulating-difference}
\end{equation}
With this fixed \(U\), put \(\widetilde\varphi(y)=\varphi(\varepsilon y)\), and set
\[
 g_\varepsilon(y) :=\partial_{\nu_{\mathcal B}}u_{0,n}^c(\varepsilon y).
\]
Harmonicity inside \(\varepsilon\mathcal B\) gives \(\int_{\partial\mathcal B}g_\varepsilon\,\mathrm{d} s=0\), while \eqref{eq:small-core-interior-gradient} gives a uniform \(H^{-1/2}(\partial\mathcal B)\) bound.
If \(c\) is the average of \(\widetilde\varphi\) on \(U\setminus\overline{\mathcal B}\), the scaled trace and Poincar\'e inequalities yield
\begin{align*}
 \left|
 \left\langle\partial_\nu u_{0,n}^c,
   \operatorname{Tr}_{\Gamma_\varepsilon}\varphi\right\rangle
 \right|
 &=\varepsilon\left|
   \left\langle g_\varepsilon,
     \widetilde\varphi-c\right\rangle_{\partial\mathcal B}
   \right|\\
 &\leq C_\kappa\varepsilon
   \|\nabla\varphi\|_
   {L^2(\varepsilon(U\setminus\overline{\mathcal B}))}.
\end{align*}
Testing \eqref{eq:small-core-insulating-difference} with \(v_{\varepsilon,n}\) proves the same energy bound without extending the insulating solution into the core.

For \(\kappa=\infty\), let \(\mathcal V_\varepsilon\) be the constant-trace space in Lemma~\ref{lem:small-core-endpoint-framework}.
Every \(\varphi\in\mathcal V_\varepsilon\) extends across \(\varepsilon\mathcal B\) by its constant trace, with zero gradient there.
Using this extension in the limiting full-space equation and subtracting \eqref{eq:small-core-conducting-weak} gives the finite-energy identity
\begin{equation}
 A_{\varepsilon,\eta}(v_{\varepsilon,n},\varphi)=0,
 \qquad \varphi\in\mathcal V_\varepsilon.
 \label{eq:small-core-conducting-difference}
\end{equation}
Take a fixed cutoff \(\chi\in C_c^\infty(U)\) equal to one near \(\overline{\mathcal B}\), and put
\[
 \zeta_{\varepsilon,n}(x) :=\chi(x/\varepsilon) \bigl(u_{0,n}^c(\eta)(0)-u_{0,n}^c(\eta)(x)\bigr).
\]
The cutoff \(\chi(\cdot/\varepsilon)\) is supported in \(\varepsilon U\subset B_{r_*}\).
Then \(v_{\varepsilon,n}-\zeta_{\varepsilon,n}\in \mathcal V_\varepsilon\), and the interior gradient bound gives
\[
 \|\nabla\zeta_{\varepsilon,n}\|_{L^2(\Omega_\varepsilon)} \leq C_\kappa\varepsilon.
\]
Using \(v_{\varepsilon,n}-\zeta_{\varepsilon,n}\) in \eqref{eq:small-core-conducting-difference}, followed by coercivity and Cauchy--Schwarz, proves the same estimate.
The freely varying core constant is essential here: it removes the constant mismatch that would otherwise produce the two-dimensional logarithmic-capacity scale.

Thus all three fixed core types satisfy
\[
 \sup_{|\eta|\leq\rho}
 \|\nabla v_{\varepsilon,n}(\eta)\|_{L^2(\Omega_\varepsilon)}
 \leq C_\kappa\varepsilon.
\]

For \(j\geq2\), the coating region is fixed, so
\begin{align*}
 &\left|
 \int_{\mathcal A_j^0}|\nabla u_{\varepsilon,n}^c|^2\,\mathrm{d} x
 -\int_{\mathcal A_j^0}|\nabla u_{0,n}^c|^2\,\mathrm{d} x
 \right|
 \leq C_\kappa\varepsilon.
\end{align*}
For \(j=1\), the same estimate and the omitted-core bound from \eqref{eq:small-core-interior-gradient} give
\begin{align*}
 &\left|
 \int_{D_1\setminus\overline{\varepsilon\mathcal B}}
 |\nabla u_{\varepsilon,n}^c|^2\,\mathrm{d} x
 -\int_{D_1}|\nabla u_{0,n}^c|^2\,\mathrm{d} x
 \right|
 \leq C_\kappa\varepsilon.
\end{align*}
Here the integral of \(|\nabla u_{0,n}^c|^2\) over \(\varepsilon\mathcal B\) is in fact \(O_\kappa(\varepsilon^2)\).
Combining these estimates with \eqref{eq:small-core-sensitivity-formula} proves \eqref{eq:small-core-C1-convergence}.

Finally, \(\boldsymbol{\mathcal G}_0(0)=0\), and the fundamental theorem of calculus in the log-conductivity variables gives
\begin{align*}
 \boldsymbol{\mathcal G}_\varepsilon(\eta)
 -\boldsymbol{\mathcal G}_0(\eta)
 ={}&\boldsymbol{\mathcal G}_\varepsilon(0)+\int_0^1
 \bigl[D_\eta\boldsymbol{\mathcal G}_\varepsilon(t\eta)
       -D_\eta\boldsymbol{\mathcal G}_0(t\eta)\bigr]
 \eta\,\mathrm{d} t.
\end{align*}
Equations \eqref{eq:small-core-forcing-bound} and \eqref{eq:small-core-C1-convergence} prove \eqref{eq:small-core-C0-convergence}, after decreasing \(\varepsilon_0\) so that \(\varepsilon_0\leq1\).

The proof is complete.
\end{proof}

Only the two-sided endpoint bound remains; it requires the leading forcing coefficient to be nonzero.

\begin{lemma}\label{lem:small-core-endpoint-signs}
For the unscaled reference core,
\[
 g_1^{\mathcal B,0}<0,
 \qquad
 g_1^{\mathcal B,\infty}>0.
\]
\end{lemma}

\begin{proof}
Let \(H_g(x)=g\cdot x\) with \(|g|=1\), and write the endpoint solution as \(u_\kappa^g=H_g+w_\kappa^g\).
The normal \(\nu\) points from \(\mathcal B\) into its exterior, so the exterior domain has inner outward normal \(-\nu\).
Green's identity and the far-field normalization give
\begin{align*}
 g^{\mathsf T}\mathsf M_0(\mathcal B)g
 &=-|\mathcal B|
   -\int_{\mathbb{R}^2\setminus\overline{\mathcal B}}
     |\nabla w_0^g|^2\,\mathrm{d} x<0,\\
 g^{\mathsf T}\mathsf M_\infty(\mathcal B)g
 &=|\mathcal B|
   +\int_{\mathbb{R}^2\setminus\overline{\mathcal B}}
     |\nabla w_\infty^g|^2\,\mathrm{d} x>0.
\end{align*}
For the first identity, use \(\partial_\nu w_0^g=-\partial_\nu H_g\).
For the second, use that \(w_\infty^g+H_g\) is constant on the core and that \(\int_{\partial\mathcal B}\partial_\nu u_\infty^g\,\mathrm{d} s=0\).
Since \(Q>2K\) implies \(Q\geq3\), cyclic covariance and reciprocity give \(\mathsf M_\kappa(\mathcal B)=\mu_\kappa I_2\).
Moreover,
\[
 g_1^{\mathcal B,\kappa} =\frac{M_{11}^{cc}+M_{11}^{ss}}{2\pi} =\frac{\mu_\kappa}{\pi}.
\]
The two strict signs follow.
\end{proof}

\begin{proof}[Proof of Theorem~\ref{thm:small-strong-core}]
Lemma~\ref{lem:small-core-sensitivity} gives
\[
 J:=J^0=D_\eta\boldsymbol{\mathcal G}_0(0).
\]
By \eqref{eq:small-core-nondegeneracy}, the matrix \(J\) is invertible.
Take the log-conductivity radius \(\rho\) supplied by Lemma~\ref{lem:small-core-C1-limit}.
Since \(\boldsymbol{\mathcal G}_0\) is real analytic, we may shrink this same \(\rho\) so that
\begin{equation}
 \|D_\eta\boldsymbol{\mathcal G}_0(\eta)-J\|
 \leq C_\kappa|\eta|,
 \qquad |\eta|\leq\rho,
 \label{eq:small-core-pointwise-J-Lipschitz}
\end{equation}
and
\[
 \|J^{-1}\|C_\kappa\rho\leq\frac14.
\]
By Lemma~\ref{lem:small-core-C1-limit}, after decreasing \(\varepsilon_*\),
\begin{equation}
 \sup_{\substack{0<\varepsilon<\varepsilon_*\\|\eta|\leq\rho}}
 \left\|J^{-1}
 \bigl(D_\eta\boldsymbol{\mathcal G}_\varepsilon(\eta)-J\bigr)
 \right\|
 \leq\frac12.
 \label{eq:small-core-contraction-derivative}
\end{equation}
The forcing estimate \eqref{eq:small-core-forcing-bound} allows us to decrease \(\varepsilon_*\) once more so that
\[
 \left|J^{-1}\boldsymbol{\mathcal G}_\varepsilon(0)\right|
 \leq\frac{\rho}{4}.
\]

Define
\[
 \mathcal T_\varepsilon(\eta)
 :=\eta-J^{-1}\boldsymbol{\mathcal G}_\varepsilon(\eta).
\]
The mean-value formula and \eqref{eq:small-core-contraction-derivative} show that
\[
 |\mathcal T_\varepsilon(\eta) -\mathcal T_\varepsilon(\widetilde\eta)| \leq\frac12|\eta-\widetilde\eta|
\]
on \(\overline{B_\rho(0)}\).
Moreover, for \(|\eta|\leq\rho\),
\[
 |\mathcal T_\varepsilon(\eta)| \leq|\mathcal T_\varepsilon(0)|+\frac12|\eta| \leq\frac{\rho}{4}+\frac{\rho}{2} =\frac{3\rho}{4}<\rho.
\]
Thus \(\mathcal T_\varepsilon\) maps the closed ball strictly into the open ball and is a contraction.
The Banach fixed-point theorem gives a unique fixed point \(\eta^\varepsilon\in B_\rho(0)\), which is exactly the unique zero of \(\boldsymbol{\mathcal G}_\varepsilon\) in the open neighborhood \(V:=B_\rho(0)\).

At the fixed point,
\[
 |\eta^\varepsilon| \leq\left|J^{-1}\boldsymbol{\mathcal G}_\varepsilon(0)\right| +\frac12|\eta^\varepsilon|,
\]
so \eqref{eq:small-core-forcing-bound} proves \eqref{eq:small-core-root-bound}.
To obtain the leading term, write
\[
 0=\boldsymbol{\mathcal G}_\varepsilon(0)
   +J\eta^\varepsilon+R_\varepsilon,
\]
where
\[
 R_\varepsilon
 :=\int_0^1
 \bigl[D_\eta\boldsymbol{\mathcal G}_\varepsilon
       (t\eta^\varepsilon)-J\bigr]
 \eta^\varepsilon\,\mathrm{d} t.
\]
The uniform Jacobian convergence and \eqref{eq:small-core-pointwise-J-Lipschitz} give
\[
 \|D_\eta\boldsymbol{\mathcal G}_\varepsilon (t\eta^\varepsilon)-J\| \leq C_\kappa\varepsilon +C_\kappa t|\eta^\varepsilon|.
\]
Consequently,
\begin{align*}
 |R_\varepsilon|
 &\leq C_\kappa\varepsilon|\eta^\varepsilon|
      +\frac{C_\kappa}{2}|\eta^\varepsilon|^2
 =O_\kappa(\varepsilon^3).
\end{align*}
It follows that
\begin{equation}
 \eta^\varepsilon
 =-J^{-1}\boldsymbol{\mathcal G}_\varepsilon(0)
  +O_\kappa(\varepsilon^3).
 \label{eq:small-core-root-vector-relation}
\end{equation}
By the exact scaling identity,
\[
 \boldsymbol{\mathcal G}_\varepsilon(0) =\bigl(\varepsilon^2g_1^{\mathcal B,\kappa}, \varepsilon^4g_2^{\mathcal B,\kappa},\ldots, \varepsilon^{2K}g_K^{\mathcal B,\kappa}\bigr)^{\mathsf T}.
\]
All components after the first are \(O_\kappa(\varepsilon^4)\), so \eqref{eq:small-core-root-vector-relation} proves \eqref{eq:small-core-leading-expansion}.
For \(\kappa=0,\infty\), Lemma~\ref{lem:small-core-endpoint-signs} proves \eqref{eq:small-core-endpoint-signs}.
Hence
\[
 -g_1^{\mathcal B,\kappa}(J^0)^{-1}e_1\neq0,
\]
and the leading expansion, after decreasing \(\varepsilon_*\), gives \eqref{eq:small-core-two-sided-bound}.

Finally, exact \(C_Q\) symmetry and \(Q>2K\) allow Proposition~\ref{prop:selection} in the finite-conductivity case and Lemma~\ref{lem:small-core-endpoint-framework} at the endpoints to convert the \(K\) reduced equalities into \eqref{eq:small-core-complete-cancellation}.
The logarithmic parameterization makes every coating conductivity finite and positive; only the core is represented by an endpoint boundary condition.
\end{proof}

The two local conductivity regimes are now established.

\subsection{Joint shape--material continuation near the concentric geometry}
\label{subsec:proof-common-conformal}

We first identify the material Jacobian from the fixed-geometry calculation and establish analytic dependence for perturbations through common conformal level curves.
We then use rotational covariance to eliminate the linear shape term and derive the expansion in Proposition~\ref{prop:local-tangent}.

\begin{proof}[Proof of Corollary~\ref{cor:common-map-branch}]
For the conductivity \eqref{eq:log-conductivity}, Proposition~\ref{prop:selection} reduces the problem to the $K$ real equations
\begin{equation}
 \boldsymbol{\mathcal G}(a,\delta,\eta)
 =(\mathcal G_1,\ldots,\mathcal G_K)=0,
 \label{eq:response-system}
\end{equation}
where $\mathcal G_n$ is defined by \eqref{eq:reduced-response}.
At \(a=0\), the interfaces are the concentric circles \(|x|=r_j\).
Specializing the homogeneous-point calculation in the proof of Theorem~\ref{thm:fixed-target} to these disks and using \(x_j=r_j^2\) gives
\begin{equation}
 D_\eta\boldsymbol{\mathcal G}(0,0,0)=A,
 \qquad D_\delta\boldsymbol{\mathcal G}(0,0,0)=v.
 \label{eq:jacobian}
\end{equation}
Here \(A\) and \(v\) are the matrix and vector defined in \eqref{eq:A-v}.
The matrix \(A\) is the difference--Vandermonde matrix \(B\) in the proof of Corollary~\ref{cor:homothetic}, so the determinant computation there gives \(\det A>0\).

Choose \(\chi\in C_c^\infty((0,\infty))\) which equals one on neighborhoods of all radii \(r_j\), vanishes near the origin, and vanishes outside a fixed ball.
On \(\mathbb C\setminus\{0\}\), set
\[
 V_a(w):=\chi(|w|) \sum_{\ell=1}^{L}a_\ell w^{1-\ell Q},
 \qquad
 \Psi_a(w):=w+V_a(w),
\]
and set \(V_a=0\) near the origin.
Then \(\Psi_a=\Phi_a\) on every reference circle, \(\Psi_a=\operatorname{Id}\) outside a fixed ball, and
\[
 \left\lVert DV_a\right\rVert _{L^\infty}\leq C|a|.
\]
Shrink the parameter neighborhood so that \(\left\lVert DV_a\right\rVert _{L^\infty}<1\).
The mean-value formula then gives
\[
 |\Psi_a(x)-\Psi_a(y)| \geq(1-\left\lVert DV_a\right\rVert _{L^\infty})|x-y|.
\]
Thus \(\Psi_a\) is injective and bi-Lipschitz onto its image.
Invariance of domain makes its image open, while properness, which follows because \(\Psi_a\) is the identity outside a fixed ball, makes the image closed.
Hence the image is all of \(\mathbb R^2\); since \(D\Psi_a=I+DV_a\) is invertible, \(\Psi_a\) is a global \(C^{1,\alpha}\) diffeomorphism.
Moreover, \(\Psi_a(\omega w)=\omega\Psi_a(w)\), and its dependence on the real and imaginary parts of \(a\) is real analytic.

Writing the scattered correction in \(\dot H^1(\mathbb R^2)\) and pulling back by \(x=\Psi_a(y)\) gives a variational problem on the fixed radial partition.
Its coefficient matrix is, phase by phase,
\[
 \sigma_{0,\delta,\eta}(y)\, \det D\Psi_a(y)\, D\Psi_a(y)^{-1}D\Psi_a(y)^{-\mathsf T}.
\]
It is uniformly elliptic for small \(a\) and depends real analytically on \((a,\delta,\eta)\) in \(L^\infty\), hence in operator norm from \(\dot H^1\) to its dual.
The pulled-back right-hand side has the same analytic dependence.
Analytic inversion of these uniformly coercive operators therefore gives real-analytic dependence of the transmission solutions and of the bounded CGPT functionals \(\boldsymbol{\mathcal G}\).
Since $\boldsymbol{\mathcal G}(0,0,0)=0$ and $D_\eta\boldsymbol{\mathcal G}(0,0,0)=A$ is invertible, the real-analytic implicit-function theorem gives the unique branch \eqref{eq:eta-branch} satisfying \eqref{eq:response-system}.
Its conductivities are finite and positive, and Proposition~\ref{prop:selection} converts \eqref{eq:response-system} into the complete cancellation statement \eqref{eq:main-cancel}.
A \(C_Q\)-invariant circle with \(Q\geq3\) must be centered at the origin.
If \(\Gamma_j(a)\) were such a circle, then
\[
 \left|r_j+\sum_{\ell=1}^L a_\ell r_j^{1-\ell Q}\zeta^\ell\right|
 \quad (|\zeta|=1)
\]
would be constant.
Any polynomial of constant modulus on the unit circle is a monomial, as follows from the polynomial identity \(p(z)p^*(z)=c z^{\deg p}\), where \(p^*(z)=z^{\deg p}\overline{p(1/\overline z)}\).
Since the displayed polynomial has the nonzero constant term \(r_j\), it must be constant, forcing every \(a_\ell=0\).
Thus \(a\neq0\) makes the interfaces noncircular, and uniqueness holds in the material neighborhood supplied by the implicit-function theorem.
\end{proof}

\begin{proof}[Proof of Proposition~\ref{prop:local-tangent}]
For every shape $a$, the zero log-conductivity vector \((\delta,\eta)=(0,0)\) gives the homogeneous medium, so
\begin{equation}
 \boldsymbol{\mathcal G}(a,0,0)=0,
 \qquad \eta(a,0)=0.
 \label{eq:zero-contrast}
\end{equation}
For $\alpha\in\mathbb{R}$, define the action
\begin{equation}
 (\mathcal R_\alpha a)_\ell=e^{i\ell Q\alpha}a_\ell.
 \label{eq:shape-rotation-action}
\end{equation}
The identity
\[
 \Phi_{\mathcal R_\alpha a}(w) =e^{i\alpha}\Phi_a(e^{-i\alpha}w)
\]
shows that $\mathcal R_\alpha a$ represents a physical rotation of the entire multilayer.
A diagonal second-type CGPT is invariant under physical rotation, and hence
\begin{equation}
 \boldsymbol{\mathcal G}(\mathcal R_\alpha a,\delta,\eta)
 =\boldsymbol{\mathcal G}(a,\delta,\eta).
 \label{eq:response-shape-invariance}
\end{equation}
At $a=0$, no nonzero real linear functional on a complex weight-$\ell Q$ parameter can be invariant under every $\alpha$.
Indeed, if \(L_\ell\) denotes the restriction of the shape derivative to the \(\ell\)-th complex coordinate, invariance gives
\[
 L_\ell(e^{i\ell Q\alpha}h)=L_\ell(h)
 \quad\text{for every }\alpha.
\]
Taking \(\alpha=\pi/(\ell Q)\) and using real linearity yields \(L_\ell(-h)=L_\ell(h)=-L_\ell(h)\), hence \(L_\ell=0\).
Differentiating \eqref{eq:response-shape-invariance} therefore gives
\begin{equation}
 D_a\boldsymbol{\mathcal G}(0,\delta,\eta)=0
 \label{eq:no-linear-shape}
\end{equation}
for every nearby log-conductivity vector $(\delta,\eta)$ at the concentric radial reference geometry.

Let $\eta_0(\delta)=\eta(0,\delta)$.
Differentiating $\boldsymbol{\mathcal G}(a,\delta,\eta(a,\delta))=0$ in $a$ at $a=0$ gives
\[
 D_\eta\boldsymbol{\mathcal G}(0,\delta,\eta_0(\delta))D_a\eta(0,\delta)
 +D_a\boldsymbol{\mathcal G}(0,\delta,\eta_0(\delta))=0.
\]
The first factor is invertible by continuity from $A$, and the second term is zero by \eqref{eq:no-linear-shape}; hence $D_a\eta(0,\delta)=0$.
By \eqref{eq:zero-contrast} and the analytic Hadamard lemma, there is an analytic vector function \(h\) such that
\[
 \eta(a,\delta)=\delta h(a,\delta).
\]
For \(\delta\neq0\), the identity \(D_a\eta(0,\delta)=0\) gives \(D_a h(0,\delta)=0\); analyticity extends this equality to \(\delta=0\).
Differentiating \(\boldsymbol{\mathcal G}(0,\delta,\eta(0,\delta))=0\) at \(\delta=0\) and using \eqref{eq:jacobian} gives
\[
 v+A\,\partial_\delta\eta(0,0)=0,
 \qquad
 h(0,0)=\partial_\delta\eta(0,0)=-A^{-1}v.
\]
Taylor expansion, together with \(D_a h(0,\delta)=0\), now yields
\[
 h(a,\delta)=-A^{-1}v+O(|\delta|+|a|^2).
\]
Multiplying by \(\delta\) proves \eqref{eq:branch-expansion}.
Since $A^{-1}v\neq0$, the branch is nontrivial for $\delta\neq0$.
\end{proof}

\begin{remark}
For $K=2$, formula \eqref{eq:branch-expansion} gives, with $x_j=r_j^2$,
\begin{align*}
 \eta_1
 &=-\frac{x_0(x_2+x_1-x_0)}{(x_1-x_0)(x_2-x_0)}\,\delta
 +O(\delta^2+|a|^2|\delta|),\\
 \eta_2
 &=\frac{x_0x_1}{(x_2-x_1)(x_2-x_0)}\,\delta
 +O(\delta^2+|a|^2|\delta|).
\end{align*}
A core with conductivity slightly above the normalized background is therefore compensated first by a lower-conductivity inner coating and then by a higher-conductivity outer coating.
\end{remark}

\subsection{Uniform shape continuity and degree persistence}\label{sec:Unif01}

\begin{lemma}\label{lem:uniform-shape-continuity}
	Fix $\kappa\in[0,\infty]$ and let $B\Subset\mathbb R^K$ be compact.
	Then
	\begin{equation}
		\lim_{a\to0}
		\sup_{\eta\in B}
		\left|
		\mathcal G_a^\kappa(\eta)
		-\mathcal G_0^\kappa(\eta)
		\right|
		=0.
		\label{eq:uniform-shape-continuity}
	\end{equation}
	The convergence is also uniform for $ta$ with $0\leq t\leq1$.
\end{lemma}

\begin{proof}
	For small $a$, let $\Psi_a$ be a $C^{1,\alpha}$ diffeomorphism from the radial partition to the deformed partition that is the identity outside a fixed ball.
	First suppose $0<\kappa<\infty$.
	Since $B$ is compact in log-conductivity space, the pulled-back coefficients have ellipticity constants independent of $(a,\eta)$ in a neighborhood of $\{0\}\times B$.
	Let \(\mathscr L_{a,\eta}\) denote the corresponding operator from
	\(\dot H^1\) to its dual.
The \(C^1\) convergence of the pullback maps
	and compactness of \(B\) give
	\[
	\sup_{\eta\in B}
	\left\lVert \mathscr L_{a,\eta}-\mathscr L_{0,\eta}\right\rVert 
	\longrightarrow0.
	\]
	Uniform coercivity bounds both inverses independently of
	\((a,\eta)\), and the resolvent identity
	\[
	\mathscr L_{a,\eta}^{-1}-\mathscr L_{0,\eta}^{-1}
	=\mathscr L_{a,\eta}^{-1}
	(\mathscr L_{0,\eta}-\mathscr L_{a,\eta})
	\mathscr L_{0,\eta}^{-1}
	\]
	therefore proves uniform convergence of the inverses on \(B\).
	The pulled-back forcing functionals converge uniformly as well.
	In particular, for each of the finitely many incident modes used in
	\eqref{eq:global-material-map}, the corresponding pulled-back scattered
	corrections satisfy an estimate of the form
	\[
	\sup_{\eta\in B}
	\|w_{a,\eta}^{(n)}-w_{0,\eta}^{(n)}\|_{\dot H^1(\mathbb R^2)}
	\leq C_B\|\Psi_a-\operatorname{Id}\|_{C^1}.
	\]
	The CGPT integrals are supported in one fixed bounded set after
	pullback.
Cauchy--Schwarz, the uniform ellipticity bounds, and the
	preceding corrector estimate show that they are uniformly continuous
	functionals of the coefficients and correctors.
Hence they inherit
	the stated uniform convergence.
	For $\kappa=0,\infty$, use \eqref{eq:small-core-insulating-weak}--\eqref{eq:small-core-conducting-weak} with the core replaced by $D_0(a)$.
	The pullback identifies the test spaces with the fixed spaces
	\[
	\begin{gathered}
	\Omega_0:=\mathbb R^2\setminus\overline{D_0(0)},
	\qquad X_0:=\dot H^1(\Omega_0),\\
	X_\infty:=\{v\in\dot H^1(\Omega_0):\operatorname{Tr}_{\Gamma_0(0)}v\text{ is constant}\}.
	\end{gathered}
	\]
	In the conducting case, subtract the lift
	\[
	\zeta_{a,H}(x):=\chi(x)\bigl(H(0)-H(\Psi_a(x))\bigr),
	\]
	where $\chi\in C_c^\infty(\mathbb R^2)$ is a fixed cutoff equal to one near $\overline{D_0(0)}$, so that the pulled-back correction minus $\zeta_{a,H}$ lies in $X_\infty$.
	Since $\Psi_a\to\operatorname{Id}$ in $C^1$ and all conductivities outside the core remain bounded above and away from zero, the pulled-back forms on these spaces are uniformly coercive and converge in operator norm uniformly for $\eta\in B$.
	The bounded forcing functionals, including the insulating boundary term and the conducting lift, converge uniformly as well.
	The same resolvent argument therefore gives uniform convergence in the gradient norm on $\Omega_0$.
	On a fixed annulus outside the support of $\Psi_a-\operatorname{Id}$, the far-field pairing estimate in Lemma~\ref{lem:small-core-endpoint-framework} transfers this energy convergence to every required CGPT.

	The pullback diffeomorphisms can be chosen so that
	\[
	\sup_{0\leq t\leq1}
	\|\Psi_{ta}-\operatorname{Id}\|_{C^{1,\alpha}}
	\leq C\|a\|_*.
	\]
	The same operator estimate is therefore uniform for $0\leq t\leq1$ in all three cases, with constants allowed to depend on the fixed $\kappa$.

The proof is complete.
\end{proof}

We finally come to the general existence of non-circular structure with \(K\)-CGPT vanishing.
For later use, set
\begin{equation}
	\left\lVert a\right\rVert _*
	:=\sum_{\ell=1}^{L}
	(\ell Q-1)|a_\ell|r_0^{-\ell Q}.
	\label{eq:global-shape-norm}
\end{equation}
For a prescribed radial core conductivity $\kappa\in[0,\infty]$, also set
\begin{equation}
	\mathcal G_0^\kappa(\eta)
	:=\left(
	\frac{1}{2\pi n}\mathbb N_{nn}^{(2)}
	[\sigma_0^\kappa(\eta)]
	\right)_{n=1}^K,
	\qquad \eta\in\mathbb R^K.
	\label{eq:main-radial-material-map}
\end{equation}

\begin{proof}[Proof of Theorem \ref{thm:global-material}]
We use the homotopy invariance of Brouwer degree \cite{Ciarlet13} and the radial degree argument in \cite{sun2026existence}.
We first consider $0<\kappa<\infty$ and define
$$\gamma_n=\frac{\sigma_{n-1}-\sigma_{n}}{\sigma_{n-1}+\sigma_{n}}, \quad n=1,2,\dots,K+1,$$
where $\sigma_{K+1}:=1$ is the background conductivity.
We also set $\gamma:=(\gamma_2, \ldots, \gamma_{K+1})\in(-1,1)^K$.
The relations between $\gamma_1$, $\gamma$ and $\kappa$, $\eta$ are
$$
\kappa=\prod_{n=1}^{K+1}\frac{1+\gamma_n}{1-\gamma_n}, \quad \eta_m=\sum_{n=m+1}^{K+1}\log\frac{1+\gamma_n}{1-\gamma_n}, \quad m=1, 2, \ldots, K.
$$
For fixed $0<\kappa<\infty$, set
\[
P(\gamma):=\prod_{n=2}^{K+1}\frac{1+\gamma_n}{1-\gamma_n},
\qquad
\gamma_1^\kappa(\gamma):=\frac{\kappa-P(\gamma)}{\kappa+P(\gamma)}.
\]
Thus the fixed-core response is $F_\kappa(\gamma):=\mathcal G_0^\kappa(\eta(\gamma))$, with the innermost contrast chosen as $\gamma_1^\kappa(\gamma)$.
After reversing the layer order and dividing the $n$th radial response by $\pi n$, the fixed-core construction in \cite[proof of Theorem~3.1, equations~(5.10)--(5.12)]{sun2026existence} gives a continuous extension of $F_\kappa$ to $[-1,1]^K$ satisfying
\begin{equation}
	0\notin F_\kappa(\partial[-1,1]^K),
	\qquad
	\deg(F_\kappa,(-1,1)^K,0)\neq0.
	\label{eq:contrast-nonzero-degree}
\end{equation}
The boundary extension uses the independence of the exterior response from the innermost contrast on $\partial[-1,1]^K$, as established in \cite[proof of Theorem~5.5]{sun2026existence}.
At $\kappa=0$ and $\kappa=\infty$, take $\gamma_1=-1$ and $\gamma_1=1$, respectively, and define $F_\kappa(\gamma):=\mathcal G_0^\kappa(\eta(\gamma))$ using the endpoint responses.
The same boundary exclusion and nonzero degree in \eqref{eq:contrast-nonzero-degree} follow directly from \cite[proof of Theorem~5.5, equation~(5.9)]{sun2026existence} at these two fixed contrasts.
For every fixed $\kappa\in[0,\infty]$, the zero set of $F_\kappa$ is therefore compactly contained in $(-1,1)^K$, and the map $\gamma\mapsto\eta(\gamma)$ is a diffeomorphism from $(-1,1)^K$ onto $\mathbb R^K$.
Excision and change of variables for Brouwer degree then give a bounded open set $\mathcal O_\kappa\Subset\mathbb R^K$ such that
\begin{equation}
	0\notin\mathcal G_0^\kappa(\partial\mathcal O_\kappa),
	\qquad
	\deg(\mathcal G_0^\kappa,\mathcal O_\kappa,0)\neq0.
	\label{eq:radial-nonzero-degree}
\end{equation}
Boundedness
in log-conductivity coordinates makes every conductivity represented by
a vector in \(\overline{\mathcal O_\kappa}\) finite and uniformly
positive.	
The number
	\begin{equation}
		\mu_\kappa
		:=\min_{\eta\in\partial\mathcal O_\kappa}
		|\mathcal G_0^\kappa(\eta)|
		\label{eq:degree-boundary-gap}
	\end{equation}
	is strictly positive.
	Lemma~\ref{lem:uniform-shape-continuity} permits $\rho_\kappa$ to be chosen smaller than the injectivity threshold in \eqref{eq:univalence} and so that
	\begin{equation}
		\sup_{0\leq t\leq1}
		\sup_{\eta\in\partial\mathcal O_\kappa}
		\left|
		\mathcal G_{ta}^\kappa(\eta)
		-\mathcal G_0^\kappa(\eta)
		\right|
		<\mu_\kappa
		\label{eq:degree-homotopy-gap}
	\end{equation}
	whenever \eqref{eq:global-shape-smallness} holds.
	Consequently,
	\[
	0\notin\mathcal G_{ta}^\kappa
	(\partial\mathcal O_\kappa)
	\qquad 0\leq t\leq1.
	\]
	Homotopy invariance of the Brouwer degree gives
	\begin{equation}
		\deg(\mathcal G_a^\kappa,\mathcal O_\kappa,0)
		=\deg(\mathcal G_0^\kappa,\mathcal O_\kappa,0)
		\neq0.
		\label{eq:shape-degree-invariance}
	\end{equation}
	Hence $\mathcal G_a^\kappa$ has a zero $\eta^a\in\mathcal O_\kappa$.
	The compact inclusion $\overline{\mathcal O_\kappa}\Subset\mathbb{R}^K$ gives the asserted finite positive material bounds.
	Finally, Proposition~\ref{prop:selection} for $0<\kappa<\infty$ and the endpoint selection argument in Lemma~\ref{lem:small-core-endpoint-framework}, applied to the core $D_0(a)$, convert the $K$ reduced equalities into the complete block cancellation \eqref{eq:global-complete-cancellation}.
	The constant-modulus argument in the proof of
	Corollary~\ref{cor:common-map-branch} shows that \(a\neq0\) makes every interface
	noncircular, while \(\kappa\neq1\) already makes the core differ from
	the background.
\end{proof}

\section{Main result without geometric symmetry}
\label{sec:symmetry-free-completion}

The constructions in Section~\ref{sec:local-constructions} retain an exact cyclic symmetry, which reduces a complete CGPT block to \(K\) scalar equations.
In this section, we vary the interfaces and conductivities simultaneously, without imposing symmetry on the perturbed structure.
The relevant target is the total-degree triangle rather than the square block; its connection with arbitrary entire harmonic incident fields is established in Section~\ref{sec:harmonic-backgrounds}.

\subsection{The total degree triangle and its frequency count}

Reciprocity \eqref{eq:reciprocity} gives
\begin{equation}
 \mathbb N_{mn}^{(1)}=\mathbb N_{nm}^{(1)},
 \qquad
 \mathbb N_{nm}^{(2)}=\overline{\mathbb N_{mn}^{(2)}},
 \qquad
 \mathbb N_{nn}^{(2)}\in\mathbb R.
 \label{eq:symfree-complex-reciprocity}
\end{equation}
For a fixed ordered pair, the two complex quantities in \eqref{eq:N1}--\eqref{eq:N2}, together with \eqref{eq:symfree-complex-reciprocity} on the diagonal, contain exactly the independent information in the four real CGPTs.
Consequently an independent target for
\begin{equation}
 M_{mn}^{\alpha\beta}=0
 \quad(m+n\leq2K+1),
 \qquad \alpha,\beta\in\{c,s\},
 \label{eq:symfree-triangle-target}
\end{equation}
is
\begin{align}
 \mathscr T_K^{(1)}
 &=\{\mathbb N_{mn}^{(1)}:1\leq m\leq n,\ m+n\leq2K+1\},
 \label{eq:symfree-target-N1}\\
 \mathscr T_K^{(2,\mathrm{off})}
 &=\{\mathbb N_{mn}^{(2)}:1\leq m<n,\ m+n\leq2K+1\},
 \label{eq:symfree-target-N2off}\\
 \mathscr T_K^{(2,\mathrm{diag})}
 &=\{\mathbb N_{nn}^{(2)}:1\leq n\leq K\}.
 \label{eq:symfree-target-N2diag}
\end{align}

For \(1\leq q\leq2K+1\), define
\begin{align}
 I_q^+
 &=\{(m,n):1\leq m\leq n,\ m+n=q\},
 \label{eq:symfree-Iplus}\\
 I_q^-
 &=\{(m,n):1\leq m<n,\ n-m=q,\ m+n\leq2K+1\}.
 \label{eq:symfree-Iminus}
\end{align}
The first set indexes the first-type responses of Fourier weight \(q\), and the second indexes the second-type off-diagonal responses of the same weight.
Direct counting gives
\begin{equation}
 \#I_q^+=\left\lfloor\frac q2\right\rfloor,
 \qquad
 \#I_q^-=\left\lfloor\frac{2K+1-q}{2}\right\rfloor,
 \qquad
 \boxed{\#I_q^++\#I_q^-=K}.
 \label{eq:symfree-frequency-count}
\end{equation}
Thus the triangle contains \(K(2K+1)\) complex conditions and \(K\) real diagonal conditions, or \(4K^2+3K\) real conditions in total.
For each \(q\), the \(K\) complex conditions in \eqref{eq:symfree-frequency-count} can therefore be paired with \(K\) interface Fourier coefficients.
This count alone does not imply solvability: the corresponding derivative blocks must still be shown to be invertible.

\begin{remark}\label{rem:symfree-sharp-cutoff}
The cutoff is sharp for the present architecture, in which the core shape is prescribed and only the \(K\) coating interfaces and \(K\) coating conductivities are controls.
At total degree \(2K+2\), every even frequency block has \(K+1\) complex conditions, and the diagonal target has \(K+1\) real conditions.
For \(K=3\), the \(q=2\) block would already contain
\[
 \mathbb N_{11}^{(1)},\quad \mathbb N_{13}^{(2)},\quad \mathbb N_{24}^{(2)},\quad \mathbb N_{35}^{(2)},
\]
whereas the three coating boundaries provide only three complex coefficients of frequency two.
The number seven is therefore a control threshold, not a physical impossibility: an additional coating, an additional controlled interface, or another selection mechanism can move the threshold.
\end{remark}

\subsection{Analytic shape dependence and the Hadamard derivative}

Fix concentric reference interfaces \(\Gamma_j^0=\partial B_{r_j}\), \(0\leq j\leq K\), with \(0<r_0<\cdots<r_K\).
The core conductivity \(\sigma_0=\kappa\) is fixed, the coating conductivities are \(\sigma_j=e^{\eta_j}\), \(1\leq j\leq K\), and \(\sigma_{K+1}=1\).
Write \(\sigma_{0,\eta}\) for this concentric reference conductivity.
Let \(\mathbb T:=\mathbb R/(2\pi\mathbb Z)\) denote the one-dimensional torus, define
\begin{equation}
 \mathcal X:=\prod_{j=0}^{K}C^{2,\alpha}(\mathbb T;\mathbb R),
 \qquad 0<\alpha<1,
 \label{eq:symfree-shape-space}
\end{equation}
and, for \(h=(h_0,\ldots,h_K)\), set
\begin{equation}
 \Gamma_j(h)
 =\{(r_j+h_j(\theta))\mathrm{e}^{i\theta}:\theta\in\mathbb T\}.
 \label{eq:symfree-normal-graphs}
\end{equation}
Our Fourier convention is
\begin{equation}
 \widehat h_j(q)=\frac1{2\pi}\int_0^{2\pi}
 h_j(\theta)\mathrm{e}^{iq\theta}\,\mathrm{d}\theta,
 \qquad
 h_j(\theta)=\widehat h_j(0)
 +2\operatorname{Re}\sum_{q\geq1}\widehat h_j(q)\mathrm{e}^{-iq\theta}.
 \label{eq:symfree-Fourier-convention}
\end{equation}

The implicit-function argument requires analytic dependence of the target CGPTs on these shape coordinates and on the conductivities.
The following lemma supplies this dependence on a fixed reference partition.

\begin{lemma}\label{lem:symfree-shape-analyticity}
Near every concentric configuration \((0,\bar\eta)\in\mathcal X\times\mathbb R^K\) representing finite positive conductivities, the map from \((h,\eta)\) to any finite family of ordinary CGPTs is real analytic.
After shrinking the neighborhood, the interfaces remain strictly nested and the conductivity remains uniformly elliptic.
\end{lemma}

\begin{proof}
Choose mutually disjoint radial collars of the reference circles and smooth radial cutoffs \(\chi_j\) supported in those collars.
The map
\[
 \Psi_h(r\mathrm{e}^{i\theta}) =\left(r+\sum_{j=0}^{K}\chi_j(r)h_j(\theta)\right)\mathrm{e}^{i\theta}
\]
equals the identity near the origin and outside a fixed ball and sends \(\Gamma_j^0\) to \(\Gamma_j(h)\).
If \(\left\lVert D\Psi_h-I\right\rVert _{L^\infty}<1\), then
\[
 |\Psi_h(x)-\Psi_h(y)| \geq(1-\left\lVert D\Psi_h-I\right\rVert _{L^\infty})|x-y|,
\]
so \(\Psi_h\) is injective and its inverse is Lipschitz on its image.
Because \(\Psi_h\) equals the identity outside a fixed ball, it is proper.
Invariance of domain makes its image open, while properness makes the image closed.
Hence the image is all of \(\mathbb R^2\).
Since \(D\Psi_h\) is invertible everywhere, the local inverse theorem and the preceding estimate show that \(\Psi_h\) is a global bi-Lipschitz \(C^{2,\alpha}\) diffeomorphism.
Pullback to the fixed radial partition gives
\[
 A_{h,\eta} =\sigma_{0,\eta}\det D\Psi_h D\Psi_h^{-1}D\Psi_h^{-\mathsf T}.
\]
This coefficient is real analytic in \((h,\eta)\) as an \(L^\infty\)-valued map and is uniformly coercive nearby.
For the incident polynomial \(P_n^\beta\), the finite-energy unknown is the pulled-back corrector
\[
 v_n^\beta=u_n^\beta\circ\Psi_h-P_n^\beta\in\dot H^1(\mathbb R^2),
\]
not the total polynomially growing field itself.
Changing variables in the weak formulation gives
\[
 \int_{\mathbb R^2}A_{h,\eta}\nabla(u_n^\beta\circ\Psi_h)\cdot\nabla\varphi\,\mathrm{d}x=0,
 \qquad \varphi\in C_c^\infty(\mathbb R^2).
\]
Since \(P_n^\beta\) is harmonic, the corrector satisfies
\[
 \int_{\mathbb R^2}A_{h,\eta}\nabla v_n^\beta\cdot\nabla\varphi\,\mathrm{d}x
 =-\int_{\mathbb R^2}(A_{h,\eta}-I)\nabla P_n^\beta\cdot\nabla\varphi\,\mathrm{d}x.
\]
Because \(A_{h,\eta}-I\) is compactly supported, the right-hand side defines a bounded functional on \(\dot H^1(\mathbb R^2)\), and this identity extends by density to every \(\varphi\in\dot H^1(\mathbb R^2)\).
Analytic inversion of the uniformly coercive variational operator proves analytic dependence of the corrector.
Pulling \eqref{eq:real-cgpt} back to the fixed partition then proves the assertion that the CGPTs are analytic with respect to $h$ and $\eta$.
\end{proof}

To identify the derivative with respect to the interface variables, it is convenient to use the bilinear response associated with two harmonic loadings.
For two real harmonic loadings \(H,F\), put
\begin{equation}
 \mathscr M_\sigma(F,H)
 :=\int_{\mathbb R^2}(\sigma-1)\nabla u_H\cdot\nabla F\,\mathrm{d} x
 \label{eq:symfree-bilinear-response}
\end{equation}
and extend it complex bilinearly.
Thus
\begin{equation}
 \mathbb N_{mn}^{(1)}=\mathscr M_\sigma(z^m,z^n),
 \qquad
 \mathbb N_{mn}^{(2)}=\mathscr M_\sigma(\overline{z}^{\,m},z^n).
 \label{eq:symfree-complex-response}
\end{equation}
Orient \(\nu_j\) from phase \(\sigma_j\) toward phase \(\sigma_{j+1}\).
If \(u_H,u_F\) are the unperturbed total fields, write their common conormal fluxes as
\[
 q_H^{(j)}=\sigma_j\partial_{\nu_j}u_H^- =\sigma_{j+1}\partial_{\nu_j}u_H^+,
 \qquad
 q_F^{(j)}=\sigma_j\partial_{\nu_j}u_F^- =\sigma_{j+1}\partial_{\nu_j}u_F^+.
\]

\begin{proposition}\label{prop:symfree-Hadamard}
If positive \(h_j\) expands the inner phase across \(\Gamma_j^0\), then
\begin{align}
 D\mathscr M_\sigma(F,H)[h]
 =\sum_{j=0}^{K}\int_{\Gamma_j^0}h_j
 \bigg[& (\sigma_j-\sigma_{j+1})
       \partial_\tau u_H\,\partial_\tau u_F\notag\\
 &+\left(\frac1{\sigma_{j+1}}-\frac1{\sigma_j}\right)
       q_H^{(j)}q_F^{(j)}\bigg]\,\mathrm{d} s.
 \label{eq:symfree-Hadamard}
\end{align}
\end{proposition}

\begin{proof}
Choose a compactly supported \(C^{2,\alpha}\) vector field \(V\) whose normal trace on \(\Gamma_j^0\) is \(V\cdot\nu_j=h_j\), and let \(T_t=I+tV\).
Pulling the transmission problems back by \(T_t\), differentiating at \(t=0\), and using the second state as the adjoint eliminates the material derivatives of the two states.
Since the correctors have finite energy and \(V\) is compactly supported, this standard calculation gives
\[
 D\mathscr M_\sigma(F,H)[V]
 =\sum_E\int_E \mathsf S_H^F:DV\,\mathrm{d} x,
 \]
where the sum runs over the constant-conductivity phases and
\[
 \mathsf S_H^F
 :=\sigma\bigl((\nabla u_H\cdot\nabla u_F)I
 -\nabla u_H\otimes\nabla u_F
 -\nabla u_F\otimes\nabla u_H\bigr).
\]
This is the usual polarization stress.
It is divergence-free in each phase because
\(\sigma\) is constant there and both states satisfy the conductivity equation.
Phasewise integration by parts therefore yields
\[
 D\mathscr M_\sigma(F,H)[V]
 =\sum_{j=0}^{K}\int_{\Gamma_j^0}h_j\,
 \nu_j\cdot(\mathsf S_{H,j}^F-\mathsf S_{H,j+1}^F)\nu_j\,\mathrm{d} s.
\]
The tangential derivatives are continuous across \(\Gamma_j^0\), and
\(\partial_{\nu_j}u_H^-=q_H^{(j)}/\sigma_j\) and
\(\partial_{\nu_j}u_H^+=q_H^{(j)}/\sigma_{j+1}\), with the analogous identities for
\(u_F\).
Writing \(\mathsf S_{H,\ell}^F\) for the trace of the stress from the phase
of conductivity \(\sigma_\ell\), we obtain
\[
 \nu_j\cdot\mathsf S_{H,\ell}^F\nu_j
 =\sigma_\ell\partial_\tau u_H\,\partial_\tau u_F
 -\frac{q_H^{(j)}q_F^{(j)}}{\sigma_\ell},
 \qquad \ell=j,j+1.
\]
Taking the inner-minus-outer jump gives exactly \eqref{eq:symfree-Hadamard}.
For the corresponding single-interface GPT derivative, see
\cite[Section~4]{ammari2012generalized}; the use of this formula for multicoated
GPT-vanishing structures appears in \cite{feng2017construction}.
\end{proof}

At a concentric reference configuration, angular Fourier modes diagonalize the Hadamard formula.
The resulting frequency blocks can be written explicitly in terms of the radial transmission fields.
For a radial structure and incident \(z^n\), write in phase \(\ell\)
\begin{equation}
 U_{\ell,n}(r,\theta)
 =(a_{\ell,n}r^n+b_{\ell,n}r^{-n})\mathrm{e}^{in\theta},
 \qquad b_{0,n}=0,\quad a_{K+1,n}=1.
 \label{eq:symfree-radial-mode}
\end{equation}
At \(r=r_j\), define the common trace and flux amplitudes
\begin{align}
 d_{j,n}&=a_{j,n}r_j^n+b_{j,n}r_j^{-n},
 \label{eq:symfree-djn}\\
 p_{j,n}&=\sigma_jn
 (a_{j,n}r_j^{n-1}-b_{j,n}r_j^{-n-1}).
 \label{eq:symfree-pjn}
\end{align}
For \(m\leq n\), set
\begin{align}
 C_{j,mn}^{+}
 &=r_j\left[
 \left(\frac1{\sigma_{j+1}}-\frac1{\sigma_j}\right)p_{j,m}p_{j,n}
 -(\sigma_j-\sigma_{j+1})\frac{mn}{r_j^2}d_{j,m}d_{j,n}
 \right],
 \label{eq:symfree-Cplus}\\
 C_{j,mn}^{-}
 &=r_j\left[
 \left(\frac1{\sigma_{j+1}}-\frac1{\sigma_j}\right)p_{j,m}p_{j,n}
 +(\sigma_j-\sigma_{j+1})\frac{mn}{r_j^2}d_{j,m}d_{j,n}
 \right].
 \label{eq:symfree-Cminus}
\end{align}
Substitution in \eqref{eq:symfree-Hadamard} gives, for \(h_j=c_{j,q}\mathrm{e}^{-iq\theta}+\overline{c_{j,q}}\mathrm{e}^{iq\theta}\),
\begin{align}
 D\mathbb N_{mn}^{(1)}[h]
 &=2\pi\sum_{j=0}^{K}C_{j,mn}^{+}c_{j,m+n},
 \label{eq:symfree-N1-shape-derivative}\\
 D\mathbb N_{mn}^{(2)}[h]
 &=2\pi\sum_{j=0}^{K}C_{j,mn}^{-}c_{j,n-m},
 \qquad m<n.
 \label{eq:symfree-N2-shape-derivative}
\end{align}
The opposite tangential signs follow from \(\partial_\tau\mathrm{e}^{im\theta}\partial_\tau\mathrm{e}^{in\theta} =-mn r^{-2}\mathrm{e}^{i(m+n)\theta}\) and \(\partial_\tau\mathrm{e}^{-im\theta}\partial_\tau\mathrm{e}^{in\theta} =mn r^{-2}\mathrm{e}^{i(n-m)\theta}\).

Thus a shape coefficient of frequency \(q\) affects precisely the first-type targets with \(m+n=q\) and the off-diagonal second-type targets with \(n-m=q\).
Grouping these targets by \(q\) gives the control matrices below.
For each \(q\), choose \(K\) controlled interfaces
\begin{equation}
 \mathcal J_q\subset\{0,\ldots,K\},
 \qquad \#\mathcal J_q=K,
 \label{eq:symfree-control-interface-set}
\end{equation}
and define the real \(K\times K\) matrix \(S_q\), with rows ordered by \(I_q^+\) followed by \(I_q^-\), by
\begin{equation}
 (S_q)_{(+,m,n),j}=C_{j,mn}^+,
 \qquad
 (S_q)_{(-,m,n),j}=C_{j,mn}^-,
 \qquad j\in\mathcal J_q.
 \label{eq:symfree-Sq}
\end{equation}
Accordingly, the complex derivative block is \(2\pi S_q\).

\subsection{Conditional cancellation by interface correction}

The \(K\) diagonal second-type responses are assigned to the \(K\) conductivity variables, while the nonzero-frequency responses are assigned to the interface blocks \(S_q\).
At a radial log-conductivity root \(\bar\eta\), define
\begin{equation}
 \mathcal G_n(\eta)
 =\frac{\mathbb N_{nn}^{(2)}(\eta)}{2\pi n},
 \qquad
 J_{\bar\eta}=D_\eta(\mathcal G_1,\ldots,\mathcal G_K)(\bar\eta).
 \label{eq:symfree-material-block}
\end{equation}
Here the responses are evaluated at the concentric configuration \(h=0\).
The required nondegeneracy condition is
\begin{equation}
 \det J_{\bar\eta}\neq0,
 \qquad
 \det S_q\neq0,
 \quad 1\leq q\leq2K+1.
 \label{eq:symfree-ND}
\end{equation}

Let \(E\) insert the selected Fourier coefficients from \(\prod_{q=1}^{2K+1}\mathbb C^K\) into \(\mathcal X\), and let \(P_{\rm ctrl}\) extract those same coefficients.
By construction, \(P_{\rm ctrl}E=I\), and hence
\begin{equation}
 \mathcal X_{\rm free}:=\ker P_{\rm ctrl},
 \qquad
 \mathcal X=\mathcal X_{\rm free}\oplus\operatorname{ran}E.
 \label{eq:symfree-free-space}
\end{equation}

Under \eqref{eq:symfree-ND}, these material and interface blocks form an invertible derivative with respect to all control variables.

\begin{theorem}\label{thm:symfree-general}
Let \(\bar\eta\) be a radial \(K\)-GPT-vanishing root representing finite positive conductivities, and suppose \eqref{eq:symfree-ND} holds.
There exist a neighborhood \(\mathcal U\subset\mathcal X_{\rm free}\) of zero and a unique real-analytic map
\begin{equation}
 g\longmapsto(c(g),\eta(g)),
 \qquad
 c(0)=0,\quad\eta(0)=\bar\eta,
 \label{eq:symfree-IFT-map}
\end{equation}
from \(\mathcal U\) to \(\prod_{q=1}^{2K+1}\mathbb C^K\times\mathbb R^K\), such that the multilayer with interfaces \(h(g)=g+Ec(g)\) is strictly nested, has finite, positive, isotropic conductivities, and satisfies
\begin{equation}
 M_{mn}^{\alpha\beta}[h(g),\eta(g)]=0
 \quad\text{for every }m+n\leq2K+1,
 \quad\alpha,\beta\in\{c,s\}.
 \label{eq:symfree-exact-triangle}
\end{equation}
For each \(g\in\mathcal U\), the controls are unique in a fixed neighborhood of \((0,\bar\eta)\).
If, in addition,
\begin{equation}
 \widehat g_j(q)=0
 \quad(0\leq q\leq2K+1,\ 0\leq j\leq K),
 \label{eq:symfree-high-free-shape}
\end{equation}
then
\begin{equation}
 \left\lVert c(g)\right\rVert +\left\lvert \eta(g)-\bar\eta\right\rvert 
 \leq C\left\lVert g\right\rVert _{\mathcal X}^2.
 \label{eq:symfree-quadratic-compensation}
\end{equation}
\end{theorem}

\begin{proof}
Let \(\mathcal F(h,\eta)\) be the realification of the complex targets in \eqref{eq:symfree-target-N1}--\eqref{eq:symfree-target-N2off}, together with the normalized diagonal targets
\[
 \frac{\mathbb N_{nn}^{(2)}(h,\eta)}{2\pi n},
 \qquad 1\leq n\leq K.
\]
Vanishing of \(\mathcal F\) is equivalent to \eqref{eq:symfree-triangle-target}.
At \((h,\eta)=(0,\bar\eta)\), rotation invariance makes every first-type response and every off-diagonal second-type response zero, while the diagonal targets vanish because \(\bar\eta\) is a radial \(K\)-GPT-vanishing root.
A material variation remains radial, a positive-frequency shape variation cannot change a diagonal response to first order, and different shape frequencies decouple by \eqref{eq:symfree-N1-shape-derivative}--\eqref{eq:symfree-N2-shape-derivative}.
Define
\[
 \mathcal H(c,\eta,g):=\mathcal F(g+Ec,\eta).
\]
After ordering the controls as \((\eta,c_1,\ldots,c_{2K+1})\), the derivative of \(\mathcal H\) with respect to \((\eta,c)\) at \((0,\bar\eta,0)\) is the realification of
\[
 \operatorname{diag}(J_{\bar\eta},2\pi S_1,\ldots,2\pi S_{2K+1}).
\]
\begin{equation}
 \left|\det_{\mathbb R}
 D_{(\eta,c)}\mathcal H(0,\bar\eta,0)\right|
 =
 (2\pi)^{2K(2K+1)}\left|\det J_{\bar\eta}\right|
 \prod_{q=1}^{2K+1}(\det S_q)^2,
 \label{eq:symfree-control-determinant}
\end{equation}
which is nonzero by \eqref{eq:symfree-ND}.
Lemma~\ref{lem:symfree-shape-analyticity} and the analytic implicit-function theorem now give \eqref{eq:symfree-IFT-map}--\eqref{eq:symfree-exact-triangle}.

For a direction \(g\) satisfying \eqref{eq:symfree-high-free-shape}, formulas \eqref{eq:symfree-N1-shape-derivative}--\eqref{eq:symfree-N2-shape-derivative} and angular orthogonality give
\[
 D_g\mathcal H(0,\bar\eta,0)[g]=0.
\]
Differentiating the identity
\[
 \mathcal H(c(g),\eta(g),g)=0
\]
at \(g=0\) and using the invertibility of the control derivative yield \(Dc(0)[g]=D\eta(0)[g]=0\).
Taylor's theorem on this high-frequency subspace proves \eqref{eq:symfree-quadratic-compensation}.
\end{proof}

\begin{remark}\label{rem:symfree-translation-gauge}
For the full \(q=1\) derivative block containing all \(K+1\) interface columns, simultaneous translation gives the complex kernel vector \((1,\ldots,1)\).
This unavoidable gauge kernel does not obstruct full row rank.
Once one specified \(K\times K\) minor is nonzero, the full \(K\times(K+1)\) block has rank \(K\), and its kernel is exactly the translation line.
The ordinary CGPT translation formula is lower triangular in total degree.
Hence simultaneous translation cannot recreate a coefficient in the target from lower total degree tensors after the whole triangle has been canceled.
No assertion is made that every minor is invertible.
\end{remark}

\subsection{Certified three-coating interface derivative blocks}

To verify the nondegeneracy required by Theorem~\ref{thm:symfree-general}, consider the following three-coating data and write
\begin{equation}
 \mathcal G_a^4(\eta)
 :=\left(
 \frac{1}{2\pi n}\mathbb N_{nn}^{(2)}
 [\sigma_{a,\log4,\eta}]
 \right)_{n=1}^{3}.
 \label{eq:symfree-reduced-material-map}
\end{equation}
The radial material block is certified in \texttt{gate1} and the interface derivative blocks in \texttt{gate2}, using the common box
\begin{equation}
 \bar\eta_{\rm rad}:=
 \begin{pmatrix}
 -1.53006096729191691732657047148667374\\
 \phantom{-}0.80245726807094701062267600365567964\\
 -0.16225781746837391133029569801616697
 \end{pmatrix},
 \qquad
 B_{\rm rad}:=\bar\eta_{\rm rad}+[-10^{-20},10^{-20}]^3.
 \label{eq:symfree-radial-box}
\end{equation}
For mode \(n\), the inward transfer matrix across \(r=r_j\) is
\begin{equation}
 T_{j,n}=\frac1{2\sigma_j}
 \begin{pmatrix}
 \sigma_j+\sigma_{j+1}
 & (\sigma_j-\sigma_{j+1})r_j^{-2n}\\
 (\sigma_j-\sigma_{j+1})r_j^{2n}
 & \sigma_j+\sigma_{j+1}
 \end{pmatrix}.
 \label{eq:symfree-transfer}
\end{equation}
With \(P_n=T_{0,n}\cdots T_{3,n}\), exterior coefficients \((1,b_n)^{\mathsf T}\), and core regularity, one has
\begin{equation}
 b_n=-\frac{(P_n)_{21}}{(P_n)_{22}}.
 \label{eq:symfree-scattering-coefficient}
\end{equation}
Only \(1\leq n\leq6\) enters the total-degree-seven target.

\begin{proposition}\label{prop:symfree-shape-certificate}
Take
\[
 K=3,\quad \kappa=4,\quad (r_0,r_1,r_2,r_3)=\left(\frac12,\frac{33}{50},\frac{41}{50},1\right),
\]
with the box \(B_{\rm rad}\) in \eqref{eq:symfree-radial-box}.
The radial map \(\mathcal G_0^4\) has exactly one zero \(\eta_{\rm rad}^*\) in \(B_{\rm rad}\), and
\begin{equation}
 \det D_\eta\mathcal G_0^4(B_{\rm rad})
 \subset
 [0.01830739599226524591860,
  0.01830739599226524599060].
 \label{eq:symfree-radial-determinant}
\end{equation}
For every \(q=1,\ldots,7\), choose the interface control columns \(\mathcal J_q=\{1,2,3\}\).
Then every \(S_q\) is invertible.
More precisely, the following outward decimal intervals contain \(\det S_q(B_{\rm rad})\):
\begin{table}[H]
\centering
\small
\begin{tabular}{c@{\qquad}l}
\toprule
\(q\) & certified interval for \(\det S_q\)\\
\midrule
1 & \([6.636901510952761,\ 6.636901510952765]\)\\
2 & \([1.908901271610178,\ 1.908901271610180]\)\\
3 & \([-1.051256939049554,\ -1.051256939049552]\)\\
4 & \([-0.01708012219352023,\ -0.01708012219352011]\)\\
5 & \([0.007250882067103773,\ 0.007250882067103875]\)\\
6 & \([-0.000236806328747981,\ -0.000236806328747955]\)\\
7 & \([-0.000116447075501611,\ -0.000116447075501558]\)\\
\bottomrule
\end{tabular}
\caption{Strict enclosures of the determinants of the interface derivative blocks restricted to the three coating interfaces.
The common factor \(2\pi\) in the physical derivative is not included.}
\label{tab:symfree-certified-Sq}
\end{table}
\end{proposition}

\begin{proof}
The radial certificates in \texttt{gate1} evaluate the transfer matrices \eqref{eq:symfree-transfer} and their analytic log-conductivity derivatives with outward-rounded Arb ball arithmetic at \(512\) and \(768\) bits.
The computation uses a fixed nonsingular dyadic preconditioner and gives a Krawczyk image strictly inside \(B_{\rm rad}\), with each boundary margin larger than \(9.999999999999996\times10^{-21}\).
The standard Krawczyk criterion therefore gives the unique zero \(\eta_{\rm rad}^*\), while direct outward evaluation of the Jacobian gives \eqref{eq:symfree-radial-determinant}.
The accompanying source and verifier are supplied with the certificate.

The \texttt{gate2} calculation treats all radii and log-conductivity box endpoints as exact rationals and rounds every algebraic operation outwards on a \(10^{-90}\) grid.
Exponentials are enclosed by the Taylor polynomial of degree \(120\); for \(x\geq0\) the tail is bounded by
\[
 0\leq\mathrm{e}^x-\sum_{k=0}^{120}\frac{x^k}{k!}
 \leq\frac{x^{121}}{121!}\frac1{1-x/122},
\]
and negative arguments are handled by reciprocal intervals.
The largest exponential tail is below \(2.8\times10^{-179}\).
All six denominator intervals \((P_n)_{22}(B_{\rm rad})\) are positive and lie in \([1.0822,2.1143]\).
Equations \eqref{eq:symfree-transfer}--\eqref{eq:symfree-Cminus} give every entry, and direct outward evaluation of the seven \(3\times3\) determinants gives Table~\ref{tab:symfree-certified-Sq}.
The source and JSON certificate are supplied in \texttt{gate2}, where the verifier rebuilds the certificate and checks its source hash.
\end{proof}

The \(q=1\) certificate concerns the specific minor with columns \(\Gamma_1,\Gamma_2,\Gamma_3\).
Hence the full \(3\times4\) interface derivative block has full row rank.
At the exact radial root its kernel is precisely the simultaneous-translation line, so fixing the core \(q=1\) coefficient gives a transverse gauge.

\begin{theorem}\label{thm:symfree-K3}
At the certified radial root \(\eta_{\rm rad}^*\) of Proposition~\ref{prop:symfree-shape-certificate}, let the free space consist of all small \(C^{2,\alpha}\) interface perturbations whose selected coating coefficients satisfy
\[
 \widehat g_j(q)=0,
 \qquad j=1,2,3,\quad q=1,\ldots,7.
\]
There is a neighborhood of zero and a unique local real-analytic choice of those \(21\) complex coating coefficients and the three coating log-conductivities for which
\begin{equation}
 M_{mn}^{\alpha\beta}=0
 \quad(m+n\leq7),
 \qquad \alpha,\beta\in\{c,s\}.
 \label{eq:symfree-K3-triangle}
\end{equation}
In particular, the entire core perturbation may be prescribed.
The resulting interfaces are strictly nested and all conductivities remain finite, positive, scalar, and isotropic.
\end{theorem}

\begin{proof}
The conductivity Jacobian determinant is enclosed away from zero by \eqref{eq:symfree-radial-determinant}, and all interface derivative block determinants are enclosed away from zero by Proposition~\ref{prop:symfree-shape-certificate}.
Theorem~\ref{thm:symfree-general} applies.
The assertion is local: no explicit shape radius, large-deformation continuation, or global uniqueness is claimed.
\end{proof}

\subsection{An analytic family with no Euclidean symmetry}

The preceding theorem allows the core perturbation to be prescribed freely.
A convenient way to exclude rotations about an unknown center and reflections across an unknown line is to prescribe a support function whose two leading noncircular harmonics have coprime orders.
Put
\begin{equation}
 p_s(\theta)=\frac12+s\left[
 \cos8\theta+\frac35\cos\left(9\theta+\frac\pi{16}\right)\right]
 \label{eq:symfree-support-function}
\end{equation}
and
\begin{equation}
 X_s(\theta)=p_s(\theta)n(\theta)+p_s'(\theta)t(\theta),
 \quad
 n=(\cos\theta,\sin\theta),\quad
 t=(-\sin\theta,\cos\theta).
 \label{eq:symfree-support-parametrization}
\end{equation}
The radius of curvature is
\begin{equation}
 p_s+p_s''=\frac12-63s\cos8\theta
 -48s\cos\left(9\theta+\frac\pi{16}\right).
 \label{eq:symfree-curvature-radius}
\end{equation}
If \(0<|s|<1/444\), then \(p_s+p_s''\geq1/2-111|s|>1/4\), so the core is smooth and strictly convex.
Near zero, this curve has a unique analytic radial-graph representation \(r=1/2+h_0(s,\varphi)\), with
\begin{equation}
 h_0'(0,\varphi)=\cos8\varphi
 +\frac35\cos\left(9\varphi+\frac\pi{16}\right).
 \label{eq:symfree-core-first-jet}
\end{equation}

\begin{lemma}\label{lem:symfree-no-Euclidean-symmetry}
For every sufficiently small \(s\neq0\), the convex curve \eqref{eq:symfree-support-parametrization} has no nontrivial rotational symmetry about any center and no reflection symmetry across any line.
\end{lemma}

\begin{proof}
Translation of the origin changes a support function only by a first Fourier harmonic.
If rotation through \(\gamma\) about any center preserved the curve, its nonzero Fourier coefficients of orders eight and nine would give
\[
 \mathrm{e}^{-8i\gamma}=\mathrm{e}^{-9i\gamma}=1.
\]
Since \(\gcd(8,9)=1\), the rotation is trivial.
If reflection across a line of direction \(\beta\) preserved the curve, the support coefficients \(\widehat p(k)=(2\pi)^{-1}\int p(\theta)\mathrm{e}^{-ik\theta} \,\mathrm{d}\theta\) would satisfy, for \(k\geq2\),
\[
 \widehat p(k)=\mathrm{e}^{-2ik\beta}\overline{\widehat p(k)}.
\]
The real nonzero eighth coefficient gives \(\beta=\ell\pi/8\).
The ninth coefficient, whose phase is \(\pi/16\) modulo \(\pi\), then requires \((1+18\ell)/16\in\mathbb Z\), which is impossible because the numerator is odd.
\end{proof}

Applying Theorem~\ref{thm:symfree-K3} to this core perturbation gives the symmetry-free family below.

\begin{corollary}\label{cor:symfree-asymmetric-family}
There is \(s_*>0\) such that every \(0<|s|<s_*\) determines a strictly nested three-coating conductivity with the core \eqref{eq:symfree-support-parametrization}, no nontrivial rotation or reflection symmetry, and the exact cancellation \eqref{eq:symfree-K3-triangle}.
Its controlled coating coefficients and log-conductivity vector satisfy
\begin{equation}
 c_{j,q}(s)=O(s^2),
 \qquad
 \eta(s)-\eta_{\rm rad}^*=O(s^2).
 \label{eq:symfree-squared-retuning}
\end{equation}
\end{corollary}

\begin{proof}
Take the free perturbation \(g(s)=(h_0(s),0,0,0)\) in Theorem~\ref{thm:symfree-K3}.
Equation \eqref{eq:symfree-core-first-jet} shows that the first derivative of the target along this curve is zero, because the first jet contains no Fourier order \(0,\ldots,7\).
Differentiating the implicit identity at \(s=0\) gives \(c'(0)=0\) and \(\eta'(0)=0\), proving \eqref{eq:symfree-squared-retuning}.
Although the radial graph may acquire low harmonics at order \(s^2\), the argument uses only its first jet and does not require those harmonics to vanish for finite \(s\).

The core has no Euclidean symmetry by Lemma~\ref{lem:symfree-no-Euclidean-symmetry}.
At the certified radial root, the five phase conductivities are pairwise distinct, and they remain so after shrinking the analytic neighborhood.
Any Euclidean symmetry of the full conductivity would therefore preserve the core phase, contradicting the lemma.
\end{proof}

\section{Entire harmonic incident fields}
\label{sec:harmonic-backgrounds}

Complete \(K\)-CGPT cancellation controls only a finite block of incident and outgoing modes, whereas an entire harmonic incident field contains infinitely many incident modes.
The cyclic selection rules bridge this gap by extending the block cancellation to all pairs of sufficiently low total degree.
The resulting powers of the inclusion size are then summed by means of a uniform interface estimate and transferred from free space to the bounded domain by a harmonic boundary correction.

\begin{proposition}\label{prop:total-degree}
Let \(Q\geq2K+2\), and let $\sigma$ be any $C_Q$-invariant multilayer with finite positive conductivities for which the complete block \eqref{eq:complete-block} vanishes.
Then
\begin{equation}
 M_{mn}^{\alpha\beta}=0
 \quad\text{whenever }m+n\leq2K+1,
 \quad \alpha,\beta\in\{c,s\}.
 \label{eq:total-degree-zero}
\end{equation}
\end{proposition}

\begin{proof}
Suppose that \(m+n\leq2K+1<Q\).
The first selection rule \eqref{eq:selection1} gives \(\mathbb N_{mn}^{(1)}=0\), because a positive multiple of \(Q\) cannot be smaller than \(Q\).
Moreover, \(|m-n|<m+n<Q\), so \eqref{eq:selection2} allows \(\mathbb N_{mn}^{(2)}\) to be nonzero only when \(m=n\).
In that case \(2m\leq2K+1\), hence \(m=n\leq K\), and \eqref{eq:complete-block} gives \(\mathbb N_{nn}^{(2)}=0\).
Thus both complex families vanish, and \eqref{eq:N1}--\eqref{eq:N2} give the claim for all four real CGPTs.
\end{proof}

The proposition identifies the cancellation region required for an \(O(\varepsilon^{2K+2})\) estimate.
The exact scaling law below converts this total-degree condition into powers of \(\varepsilon\).
For the remainder of the section, let
\(\mathcal D=(D_0,\ldots,D_K)\) be a fixed reference multilayer centered at the origin, with finite positive conductivities and \(C^{2,\alpha}\) interfaces, and write
\[
 \mathcal D_{\varepsilon,z_*}:=z_*+\varepsilon\mathcal D.
\]
For this scaled multilayer, define the centered harmonic basis
\begin{equation}
 P_{n,z_*}^c(z)=\operatorname{Re}(z-z_*)^n,
 \qquad
 P_{n,z_*}^s(z)=\operatorname{Im}(z-z_*)^n,
 \label{eq:centered-basis}
\end{equation}
and denote the tensors computed with this basis by $M_{mn,z_*}^{\alpha\beta}$.
The basis is centered at $z_*$ because ordinary CGPTs mix under translation of the origin.

Uniqueness of the transmission solution and a change of variables in \eqref{eq:real-cgpt} give the exact scaling law
\begin{equation}
 M_{mn,z_*}^{\alpha\beta}[z_*+\varepsilon\mathcal D]
 =\varepsilon^{m+n}M_{mn}^{\alpha\beta}[\mathcal D].
 \label{eq:cgpt-scaling}
\end{equation}

For a general local harmonic field, the modal expansion is infinite.
To apply \eqref{eq:cgpt-scaling} term by term with constants independent of \(\varepsilon\), we first obtain an \(\varepsilon\)-uniform estimate for the interface densities.
Put $\Gamma_j=\partial D_j$, orient $\nu_j$ from phase $\sigma_j$ toward phase $\sigma_{j+1}$, and set $\sigma_{K+1}=1$ and
\[
 \Gamma_{j,\varepsilon}=z_*+\varepsilon\Gamma_j.
\]
For a function $h$ harmonic in a neighborhood of the closed scaled structure, define
\begin{equation}
 \mathcal S_\varepsilon[h]
 =\sum_{j=0}^{K}\mathcal S_{\Gamma_{j,\varepsilon}}
       [\varphi_{\varepsilon,j}],
 \qquad
 \mathcal S_\Gamma[\varphi](x)
 =\frac{1}{2\pi}\int_\Gamma\log|x-y|\varphi(y)\,\mathrm{d} s_y,
 \label{eq:local-scattering-definition}
\end{equation}
where the mean-zero densities are chosen so that $h+\mathcal S_\varepsilon[h]$ satisfies the flux transmission condition on every $\Gamma_{j,\varepsilon}$.
Their existence and uniqueness follow from Lemma~\ref{lem:uniform-interface-system}.

Let
\[
 X=\prod_{j=0}^{K}H^{-1/2}_0(\Gamma_j),
\]
where the subscript denotes zero mean.
With
\begin{equation}
 \left.\partial_{\nu_j}\mathcal S_{\Gamma_j}[\phi]\right|_{\pm} =\left(\pm\frac12I+\mathcal K_j^*\right)\phi,
 \label{eq:single-layer-normal-jump}
\end{equation}
where \(+\) denotes the \(\sigma_{j+1}\)-side and \(-\) the \(\sigma_j\)-side,
define $\mathcal V_{jj}=\mathcal K_j^*$ and, for $j\neq k$,
\[
 \mathcal V_{jk}\phi =\left.\partial_{\nu_j}\mathcal S_{\Gamma_k}[\phi] \right|_{\Gamma_j}.
\]
After the pullback $\widehat\varphi_{\varepsilon,j}(y)=\varphi_{\varepsilon,j}(z_*+\varepsilon y)$, the density equation on the fixed interfaces is
\begin{equation}
 \mathbb A_{\boldsymbol\sigma}
 \widehat{\boldsymbol\varphi}_\varepsilon
 =\left((\sigma_j-\sigma_{j+1})
   \partial_{\nu_j}h(z_*+\varepsilon\,\cdot)\right)_{j=0}^{K},
 \label{eq:rescaled-interface-system}
\end{equation}
where
\begin{equation}
 (\mathbb A_{\boldsymbol\sigma}\phi)_j
 =\frac{\sigma_j+\sigma_{j+1}}{2}\phi_j
  -(\sigma_j-\sigma_{j+1})
    \sum_{k=0}^{K}\mathcal V_{jk}\phi_k.
 \label{eq:rescaled-interface-operator}
\end{equation}

\begin{lemma}\label{lem:uniform-interface-system}
The operator
\(\mathbb A_{\boldsymbol\sigma}:X\to X\) in \eqref{eq:rescaled-interface-operator}
is an isomorphism.
For every \(h\) harmonic in a neighborhood of the closed scaled multilayer, the unique solution of \eqref{eq:rescaled-interface-system} satisfies
\begin{equation}
 \left\lVert \widehat{\boldsymbol\varphi}_\varepsilon\right\rVert _{X}
 \leq C_{\mathcal D,\boldsymbol\sigma}
 \sum_{j=0}^{K}
 \left\lVert \partial_{\nu_j}h(z_*+\varepsilon\,\cdot)\right\rVert 
      _{H^{-1/2}(\Gamma_j)}.
 \label{eq:uniform-rescaled-inverse}
\end{equation}
The inverse bound is uniform on compact parameter families with uniform ellipticity, interface separation, \(C^{2,\alpha}\) interface bounds, and outer radius.
\end{lemma}

\begin{proof}
Uniform scaling leaves every \(\mathcal V_{jk}\) unchanged.
First, $\mathbb A_{\boldsymbol\sigma}$ maps $X$ into itself.
Indeed,
\[
 \int_{\Gamma_j}\mathcal K_j^*\phi_j\,\mathrm{d} s =\frac12\int_{\Gamma_j}\phi_j\,\mathrm{d} s=0.
\]
If $k>j$, then $\mathcal S_{\Gamma_k}[\phi_k]$ is harmonic in $D_j$; if $k<j$, the divergence theorem across the nested interfaces gives
\[
 \int_{\Gamma_j}\mathcal V_{jk}\phi_k\,\mathrm{d} s =\int_{\Gamma_k}\phi_k\,\mathrm{d} s=0.
\]
Also $\int_{\Gamma_j}\partial_{\nu_j}h\,\mathrm{d} s=0$ because $h$ is harmonic in $D_j$.
For $C^{2,\alpha}$ interfaces, each diagonal $\mathcal K_j^*$ is compact on $H^{-1/2}_0(\Gamma_j)$ and every off-diagonal $\mathcal V_{jk}$ is smoothing.
Thus $\mathbb A_{\boldsymbol\sigma}$ is Fredholm of index zero.
If \(\mathbb A_{\boldsymbol\sigma}\phi=0\), then \(\sum_k\mathcal S_{\Gamma_k}[\phi_k]\) is a homogeneous transmission solution.
Because every \(\phi_j\) has zero mean, the logarithmic term at infinity cancels and this field is \(O(|x|^{-1})\).
Writing this field as \(W\), integration by parts over all phases in a large disk and passage to infinity give
\[
 \int_{\mathbb{R}^2}\sigma|\nabla W|^2\,\mathrm{d} x=0.
\]
Indeed, the interface terms cancel by continuity of the potential and conormal flux, while the outer boundary term tends to zero because \(W=O(|x|^{-1})\) and \(\nabla W=O(|x|^{-2})\).
Thus \(W\) is constant, and its decay at infinity gives \(W=0\).
The normal-derivative jump across each interface then gives \(\phi_j=0\).
Hence $\mathbb A_{\boldsymbol\sigma}:X\to X$ is bijective and has a bounded inverse.
Estimate \eqref{eq:uniform-rescaled-inverse} follows from the bounded inverse theorem.

For a compact parameter family, identify the varying trace spaces with fixed reference spaces by the chosen \(C^{2,\alpha}\) interface parametrizations.
The pulled-back operators then depend continuously in operator norm on the interfaces and conductivities.
Since every operator in the family is invertible by the preceding argument and inversion is continuous on the open set of bounded invertible operators, compactness gives a uniform bound for their inverses.
Equation \eqref{eq:rescaled-interface-operator} remains regular if two adjacent conductivities coincide: the corresponding row simply gives $\sigma_j\phi_j=0$.
\end{proof}

The uniform inverse bound justifies summing the interface densities generated by the harmonic expansion.
Combined with total-degree cancellation, it yields the following free-space estimate.

\begin{lemma}\label{lem:small-volume}
Let $d\geq2$, let $h$ be harmonic and bounded in $B_R(z_*)$, and suppose that the CGPTs of the fixed reference multilayer $\mathcal D$, computed about the origin using \eqref{eq:harmonic-basis}, satisfy
\begin{equation}
 M_{mn}^{\alpha\beta}[\mathcal D]=0
 \quad\text{for }m+n\leq d-1,
 \qquad \alpha,\beta\in\{c,s\}.
 \label{eq:free-space-hypothesis}
\end{equation}
Let $E_0\Subset\mathbb{R}^2\setminus\{z_*\}$, and put
\[
 a_*:=\max_{y\in\overline{D_K}}|y|,
 \qquad
 \delta_0:=\operatorname{dist}(E_0,z_*).
\]
If
\begin{equation}
 \varepsilon a_*\leq\min\{R/2,\delta_0/2\},
 \label{eq:small-volume-separation}
\end{equation}
then
\begin{equation}
 \left\lVert \mathcal S_\varepsilon[h]\right\rVert _{C^1(E_0)}
 \leq C\varepsilon^d\left\lVert h\right\rVert _{L^\infty(B_R(z_*))}.
 \label{eq:free-space-estimate}
\end{equation}
The constant depends on $d$, $R$, $\delta_0$, the reference interfaces and fixed ellipticity bounds, but not on $\varepsilon$ or $h$.
It is uniform on the compact parameter families specified in Lemma~\ref{lem:uniform-interface-system}.
\end{lemma}

\begin{proof}
Set $\rho=3R/4$ and $M=\left\lVert h\right\rVert _{L^\infty(B_R(z_*))}$.
The harmonic expansion about $z_*$ is
\begin{equation}
 h(z_*+re^{i\theta})
 =h(z_*)+\sum_{n=1}^{\infty}r^n
   \bigl(a_n^c\cos n\theta+a_n^s\sin n\theta\bigr).
 \label{eq:local-harmonic-expansion}
\end{equation}
Taking Fourier coefficients on $|z-z_*|=\rho$ gives
\begin{equation}
 |a_n^c|+|a_n^s|\leq4M\rho^{-n}.
 \label{eq:cauchy-harmonic-coefficients}
\end{equation}
For the $n$-th centered harmonic polynomial,
\[
 \partial_{\nu_j}P_{n,z_*}^\beta(z_*+\varepsilon y) =\varepsilon^{n-1}\partial_{\nu_j}P_n^\beta(y).
\]
The fixed-interface trace bound
\[
 \left\lVert \partial_{\nu_j}P_n^\beta\right\rVert _{H^{-1/2}(\Gamma_j)} \leq Cn a_*^{n-1}
\]
and \eqref{eq:uniform-rescaled-inverse} show that the density series generated by \eqref{eq:local-harmonic-expansion} converges absolutely in $X$, because
\[
 \sum_{n=1}^{\infty} (|a_n^c|+|a_n^s|)\varepsilon^{n-1}n a_*^{n-1} \leq\frac{CM}{\rho} \sum_{n=1}^{\infty}n \left(\frac{\varepsilon a_*}{\rho}\right)^{n-1}<\infty.
\]
Here $\varepsilon a_*/\rho\leq2/3$ by \eqref{eq:small-volume-separation}.
The boundedness of the inverse in Lemma~\ref{lem:uniform-interface-system} therefore permits termwise application of the scattering operator to the harmonic series.
The constant term produces no density.
By linearity and uniqueness of the interface system, the remaining rescaled density is the absolutely convergent sum of the densities generated on the reference interfaces by the polynomial loadings \(a_n^\beta\varepsilon^{n-1}P_n^\beta\).
The standard GPT multipole representation \eqref{eq:multipole-expansion}, translated to \(z_*\) and combined with the exact scaling law \eqref{eq:cgpt-scaling}, shows that the coefficient of the \(m\)-th outgoing \(\alpha\)-mode generated by \(a_n^\beta P_{n,z_*}^\beta\) is \(-\varepsilon^{m+n}M_{mn}^{\alpha\beta}a_n^\beta/(2\pi m)\) \cite[Chapter~6]{ammari2007polarization}; the identification of these coefficients with GPTs is also developed in \cite{ammari2014reconstruction}.
The zero-mean normalization of the interface densities removes the additive \(\log\varepsilon\) term produced when the two-dimensional single-layer potential is rescaled.

It remains to control the growth of the reference CGPTs.
Write $u_n^\beta=P_n^\beta+w_n^\beta$.
Coercivity of the whole-space variational equation gives
\[
 \left\lVert \nabla w_n^\beta\right\rVert _{L^2(\mathbb{R}^2)} \leq C\left\lVert \nabla P_n^\beta\right\rVert _{L^2(D_K)}.
\]
Since $D_K\subset B_{a_*}$,
\[
 \left\lVert \nabla P_n^\beta\right\rVert _{L^2(D_K)} \leq(\pi n)^{1/2}a_*^n.
\]
The variational definition \eqref{eq:real-cgpt} and Cauchy--Schwarz therefore yield
\begin{equation}
 |M_{mn}^{\alpha\beta}[\mathcal D]|
 \leq C_0\sqrt{mn}\,a_*^{m+n}.
 \label{eq:cgpt-growth-bound}
\end{equation}

For $x-z_*=r(\cos\theta,\sin\theta)$, put $\chi_m^c(\theta)=\cos m\theta$ and $\chi_m^s(\theta)=\sin m\theta$.
For \(|x-z_*|>\varepsilon a_*\), the preceding density decomposition and the termwise logarithmic-kernel expansion therefore give
\begin{equation}
 \mathcal S_\varepsilon[h](x)
 =-\sum_{\alpha,\beta\in\{c,s\}}
   \sum_{m,n\geq1}
   \frac{\varepsilon^{m+n}M_{mn}^{\alpha\beta}[\mathcal D]
         a_n^\beta}
        {2\pi m r^m}\,
   \chi_m^\alpha(\theta).
 \label{eq:local-double-multipole-series}
\end{equation}
The outgoing modes satisfy
\begin{equation}
 \left\lVert \frac{\chi_m^\alpha(\theta)}{2\pi m r^m}\right\rVert 
       _{C^1(E_0)}
 \leq C_{\delta_0}\delta_0^{-m}.
 \label{eq:outgoing-mode-c1}
\end{equation}
Hypothesis \eqref{eq:free-space-hypothesis} removes all pairs with $m+n\leq d-1$.
With
\[
 L=a_*\max\{\rho^{-1},\delta_0^{-1}\},
 \qquad \varepsilon L\leq\frac23,
\]
equations \eqref{eq:cauchy-harmonic-coefficients}, \eqref{eq:cgpt-growth-bound}, and \eqref{eq:outgoing-mode-c1} give
\begin{align}
 \left\lVert \mathcal S_\varepsilon[h]\right\rVert _{C^1(E_0)}
 &\leq CM
 \sum_{\substack{m,n\geq1\\m+n\geq d}}
 \sqrt{mn}\,(\varepsilon L)^{m+n}\notag\\
 &\leq CM\sum_{p=d}^{\infty}p^2(\varepsilon L)^p\notag\\
 &\leq CM\varepsilon^dL^d
       \sum_{p=d}^{\infty}p^2(2/3)^{p-d}.
 \label{eq:double-series-sum}
\end{align}
The final series is finite; hence the multipole expansion converges absolutely in \(C^1(E_0)\) and satisfies \eqref{eq:free-space-estimate}.

The proof is complete.
\end{proof}

The free-space scattered field does not itself satisfy the boundary condition on \(\partial\Omega\).
A harmonic correction restores that condition and, by the preceding estimate, remains of the same order in \(\varepsilon\).

\begin{proof}[Proof of Theorem~\ref{thm:arbitrary-H}]
Set $d=2K+2$.
Proposition~\ref{prop:total-degree} verifies \eqref{eq:free-space-hypothesis}.
Let
\[
 \mathcal H_R =\{h\in C(\overline{B_R(z_*)}):\Delta h=0 \text{ in }B_R(z_*)\}
\]
with the supremum norm, and let $\mathcal E_\Omega$ be the harmonic Dirichlet extension from $\partial\Omega$ into $\Omega$.
Define the feedback operator on $\mathcal H_R$ by
\begin{equation}
 \mathcal T_\varepsilon h
 =\left.\mathcal E_\Omega\bigl(
   \mathcal S_\varepsilon[h]|_{\partial\Omega}\bigr)
   \right|_{B_R(z_*)}.
 \label{eq:feedback-operator}
\end{equation}
Lemma~\ref{lem:small-volume}, applied on a fixed compact tubular neighborhood of \(\partial\Omega\) disjoint from \(z_*\), and the maximum principle give
\begin{equation}
 \left\lVert \mathcal T_\varepsilon h\right\rVert _{L^\infty(B_R)}
 \leq
 \left\lVert \mathcal S_\varepsilon[h]\right\rVert _{C^0(\partial\Omega)}
 \leq C\varepsilon^d\left\lVert h\right\rVert _{L^\infty(B_R)}.
 \label{eq:feedback-operator-bound}
\end{equation}
For sufficiently small $\varepsilon$, $\left\lVert \mathcal T_\varepsilon\right\rVert <1/2$, so $I+\mathcal T_\varepsilon$ is invertible by its Neumann series.
Define the effective local incident field by
\begin{equation}
 h_\varepsilon=(I+\mathcal T_\varepsilon)^{-1}(H|_{B_R}).
 \label{eq:effective-incident-fixed-point}
\end{equation}
Then
\begin{equation}
 \left\lVert h_\varepsilon\right\rVert _{L^\infty(B_R)}
 \leq2\left\lVert H\right\rVert _{L^\infty(B_R)}.
 \label{eq:effective-incident-bound}
\end{equation}
Put
\[
 s_\varepsilon=\mathcal S_\varepsilon[h_\varepsilon],
 \qquad
 v_\varepsilon=\mathcal E_\Omega(s_\varepsilon|_{\partial\Omega}).
\]
Since \(v_\varepsilon|_{B_R}=\mathcal T_\varepsilon h_\varepsilon\), equation \eqref{eq:effective-incident-fixed-point} is equivalent to
\[
 h_\varepsilon+v_\varepsilon|_{B_R}=H|_{B_R}.
\]
Hence
\[
 \widetilde u_\varepsilon:=H-v_\varepsilon+s_\varepsilon
\]
satisfies all transmission conditions near the inclusion, is harmonic elsewhere in $\Omega$, and has the prescribed boundary value because $v_\varepsilon=s_\varepsilon$ on $\partial\Omega$.
Uniqueness gives $\widetilde u_\varepsilon=u_\varepsilon$.

For $E\Subset\Omega\setminus\{z_*\}$, take $\varepsilon$ smaller if necessary so that \eqref{eq:small-volume-separation} holds for both \(E\) and the chosen tubular neighborhood of \(\partial\Omega\).
Lemma~\ref{lem:small-volume} gives
\[
 \left\lVert s_\varepsilon\right\rVert _{C^1(E)} +\left\lVert s_\varepsilon\right\rVert _{C^0(\partial\Omega)} \leq C_E\varepsilon^d \left\lVert h_\varepsilon\right\rVert _{L^\infty(B_R)}.
\]
The maximum principle followed by the interior harmonic estimate gives
\[
 \left\lVert v_\varepsilon\right\rVert _{C^1(E)} \leq C_E\left\lVert v_\varepsilon\right\rVert _{L^\infty(\Omega)} \leq C_E\left\lVert s_\varepsilon\right\rVert _{C^0(\partial\Omega)}.
\]
Using \eqref{eq:effective-incident-bound} in $u_\varepsilon-H=s_\varepsilon-v_\varepsilon$ proves \eqref{eq:arbitrary-H-bound}.
\end{proof}

The preceding argument uses symmetry only through the total-degree cancellation in Proposition~\ref{prop:total-degree}.
It therefore applies without change to the symmetry-free family once the same cancellation condition is available.

\begin{corollary}\label{cor:symfree-arbitrary-H}
Let \(\mathcal D_s\) be any fixed nonzero member of the exact symmetry-free family in Corollary~\ref{cor:symfree-asymmetric-family}, centered at the origin, and set \(\mathcal D_{s,\varepsilon}=z_*+\varepsilon\mathcal D_s\).
Let \(\Omega\) be a fixed bounded smooth domain with \(B_R(z_*)\Subset\Omega\), assume \(\mathcal D_{s,\varepsilon}\subset B_{R/2}(z_*)\), and let \(H\) be harmonic in \(\mathbb{R}^2\).
Let \(\sigma_{\varepsilon,z_*}\) be the scaled conductivity associated with \(\mathcal D_s\), as in \eqref{eq:scaled-conductivity}, and let \(u_\varepsilon\in H^1(\Omega)\) be the corresponding weak solution with boundary trace \(H\).
Then, for every \(E\Subset\Omega\setminus\{z_*\}\),
\begin{equation}
 \left\lVert u_\varepsilon-H\right\rVert _{C^1(E)}
 \leq C_{E,s}\varepsilon^8
 \left\lVert H\right\rVert _{L^\infty(B_R(z_*))}
 \label{eq:symfree-arbitrary-H-bound}
\end{equation}
for sufficiently small \(\varepsilon\).
The constant is uniform when \(s\) ranges over a compact subinterval of the local analytic family on which the interface separation and ellipticity bounds are uniform.
\end{corollary}

\begin{proof}
Equation \eqref{eq:symfree-K3-triangle} gives the hypothesis \eqref{eq:free-space-hypothesis} of Lemma~\ref{lem:small-volume} with \(d=8\).
Define \(\mathcal T_\varepsilon,h_\varepsilon,s_\varepsilon\), and \(v_\varepsilon\) as in the proof of Theorem~\ref{thm:arbitrary-H}.
Lemma~\ref{lem:small-volume}, followed by the maximum principle and the interior harmonic estimate, gives \eqref{eq:symfree-arbitrary-H-bound}.
\end{proof}

The fixed-scale polynomial statement follows directly from the selection rules and does not require the small-volume argument.

\begin{proof}[Proof of Corollary~\ref{cor:polynomial}]
The constant component of the incident polynomial produces no scattered field.
Fix an incident order \(1\leq n\leq K\) and an observation order \(m<Q-K\).
Then \(m+n<Q\), so the first selection rule \eqref{eq:selection1} gives \(\mathbb N_{mn}^{(1)}=0\).
Moreover, \(m<Q\) and \(n<Q\), hence \(|m-n|<Q\); the second selection rule \eqref{eq:selection2} can therefore allow a response only when \(m=n\).
In that case \(m=n\leq K\), and the complete block cancellation gives \(\mathbb N_{nn}^{(2)}=0\).
Thus both complex CGPT families, and hence all four real CGPTs, vanish for every \(n\leq K\) and \(m<Q-K\).
By linearity, the multipole expansion \eqref{eq:multipole-expansion} for an arbitrary harmonic polynomial of degree at most \(K\) starts at observation order \(m=Q-K\), which proves the stated decay.
\end{proof}

\begin{remark}
The restriction to incident polynomials of degree at most $K$ is essential: the incident mode $n=Q-1$ may couple to the dipole $m=1$ because $m+n=Q$.
After scaling, this contribution has total degree $Q\geq2K+2$ and is absorbed by \eqref{eq:arbitrary-H-bound}, which is therefore a small-inclusion estimate rather than a fixed-scale invisibility statement.
\end{remark}

\section{Computer-assisted and numerical examples}
\label{sec:examples}

\subsection{A certified chiral \texorpdfstring{$C_8$}{C8} structure}
\label{sec:explicit-example}

For one fixed noncircular structure without reflection symmetry, injectivity and chirality are proved analytically, while existence and uniqueness of the exact root of the response map inside a stated box are certified by a strict Krawczyk inclusion.
The certificate combines infinite-dimensional Faber--Grunsky formulas, analytic projection-tail estimates, and outward-rounded ball arithmetic.
A Nystr\"om computation in the physical domain provides an independent consistency check without interval arithmetic.

\subsubsection{Geometry and approximate center}

Take $K=3$, $Q=8$, background conductivity $1$, core conductivity $4$, and the common map
\begin{equation}
 \Phi(w)=w+\frac{3}{20000}w^{-7}
             +\frac{i}{20000000}w^{-15}.
 \label{eq:numerical-map}
\end{equation}
Equivalently, in the parametrization by the common conformal map \eqref{eq:common-map}, the prescribed shape parameter is
\begin{equation}
 a^\star
 :=\left(\frac{3}{20000},\frac{i}{20000000}\right)\neq0.
 \label{eq:certified-noncircular-shape-parameter}
\end{equation}
The level radii, indexed from the core toward the exterior, are
\begin{equation}
 (r_0,r_1,r_2,r_3)
 =\left(\frac12,\frac{33}{50},\frac{41}{50},1\right).
 \label{eq:numerical-radii}
\end{equation}
The univalence index in \eqref{eq:univalence} is
\begin{equation}
 \frac{7(1.5\times10^{-4})}{0.5^8}
 +\frac{15(5.0\times10^{-8})}{0.5^{16}}
 =0.317952<1.
 \label{eq:numerical-univalence}
\end{equation}
The estimate certifies injectivity and nesting of the interfaces.

The geometry is chiral.
With the physical rotation convention \eqref{eq:shape-rotation-action}, a rotation through $\alpha$ sends
\[
 (a_1,a_2)\longmapsto (e^{i8\alpha}a_1,e^{i16\alpha}a_2).
\]
Making the first transformed coefficient real forces $e^{i8\alpha}=\pm1$ and hence $e^{i16\alpha}=1$, so the second coefficient remains purely imaginary.
Equivalently, $a_2/a_1^2$ is rotation invariant and has phase $\pi/2$.
A reflection symmetry of the multilayer would in particular preserve its bounded outer domain.
The exact \(C_8\) symmetry forces the centroid of that domain to be the origin, and every reflection symmetry fixes the centroid; hence any reflection axis must pass through the origin.
A reflection axis would, after rotation to the real axis, leave the exterior domain invariant under \(z\mapsto\overline z\).
The maps \(\Phi(w)\) and \(\overline{\Phi(\overline w)}\) would then be two normalized exterior maps of the same domain.
Uniqueness of the normalization \(\Phi(w)=w+O(w^{-1})\) would make them equal, forcing every Laurent coefficient to be real.
This contradicts the purely imaginary nonzero coefficient \(a_2\), so no reflection axis exists.

The noncircular solve is initialized by the certified radial root \(\eta_{\rm rad}^*\) of Proposition~\ref{prop:symfree-shape-certificate}.
Starting from that root, a high-precision Faber--Grunsky solve provides the approximate center in inner-to-outer coating order,
\begin{equation}
\boxed{(\sigma_1,\sigma_2,\sigma_3)
\approx(0.2173691055307019,
  2.2297628798180334,
  0.8502996736218210).}
\label{eq:noncircular-root}
\end{equation}
These decimals locate the certified root; the core and background conductivities are \(\sigma_0=4\) and \(\sigma_4=1\), respectively.

For the rigorous enclosure, write
\[
 \eta=(\log\sigma_1,\log\sigma_2,\log\sigma_3)
\]
in inner-to-outer coating order and define the following exact rational center by its terminating decimal representation:
\begin{equation}
 \bar\eta=
 \begin{pmatrix}
 -1.52615842326786706722987361029159674329936613882185\\
 \phantom{-}0.80189524789822358780381229273728032672229722508971\\
 -0.16216643442931959760313595606642700681881532393627
 \end{pmatrix}
 \label{eq:certified-eta-center}
\end{equation}
and the exact rational box
\begin{equation}
 B_{\rm cert}:=\bar\eta+[-10^{-15},10^{-15}]^3.
 \label{eq:certified-eta-box}
\end{equation}

\begin{theorem}\label{thm:certified-C8}
For the geometry \eqref{eq:numerical-map}--\eqref{eq:numerical-radii}, with core conductivity \(4\) and background conductivity \(1\), there is exactly one \(\eta\in B_{\rm cert}\) for which the complete \(3\)-CGPT block vanishes.
The associated finite positive coating conductivities satisfy the outward decimal enclosures in Table~\ref{tab:certified-conductivities}.
Moreover, on the whole certified box,
\begin{equation}
 \det D_\eta\mathcal G_{a^\star}^{4}(B_{\rm cert})
 \subset
 [0.018232258368660,\;0.018232258368668],
 \label{eq:certified-noncircular-log-determinant}
\end{equation}
and hence
\begin{equation}
 \deg\!\left(
  \mathcal G_{a^\star}^{4},
  \operatorname{int}B_{\rm cert},0
 \right)=+1.
 \label{eq:certified-noncircular-local-degree}
\end{equation}
If
\[
 \tau_j:=\frac{\sigma_j-\sigma_{j+1}} {\sigma_j+\sigma_{j+1}},
 \qquad
 \boldsymbol\tau:=(\tau_3,\tau_2,\tau_1),
 \qquad \sigma_4=1,
\]
\(\Theta(\boldsymbol\tau)=\eta\) denotes the corresponding coordinate map, and
\[
 \mathcal M_{a^\star}^{4}(\boldsymbol\tau) :=(M_{11}^{cc},M_{22}^{cc},M_{33}^{cc})^{\mathsf T} =\operatorname{diag}(\pi,2\pi,3\pi) \mathcal G_{a^\star}^{4}(\Theta(\boldsymbol\tau)),
\]
then
\begin{equation}
 \det D_{\boldsymbol\tau}\mathcal M_{a^\star}^{4}
  \bigl(\Theta^{-1}(B_{\rm cert})\bigr)
 \subset[-105.538312937861,\;-105.538312937737].
 \label{eq:certified-noncircular-contrast-determinant}
\end{equation}
Consequently,
\begin{equation}
 \deg\!\left(
  \mathcal M_{a^\star}^{4},
  \Theta^{-1}(\operatorname{int}B_{\rm cert}),0
 \right)=-1.
 \label{eq:certified-noncircular-contrast-degree}
\end{equation}
At the unique log-conductivity vector,
\begin{equation}
 M_{mn}^{\alpha\beta}=0
 \qquad
 \text{for all }m,n\geq1,\quad m+n\leq7,
 \quad\alpha,\beta\in\{c,s\}.
 \label{eq:certified-total-degree-seven}
\end{equation}
At the same log-conductivity vector, the second-type response of total degree eight satisfies the outward decimal enclosure
\begin{equation}
 \mathbb N_{44}^{(2)}
 \in[-0.684735501505,-0.684735501504].
 \label{eq:certified-degree-eight}
\end{equation}
In particular, the first total degree beyond the certified cancellation range contains a rigorously nonzero response that is permitted by symmetry.
The uniqueness and Brouwer degree assertions are local to the displayed log-conductivity boxes; no global uniqueness in the conductivity parameters is claimed.
\end{theorem}

Exact \(C_8\) symmetry removes the forbidden low-order responses, and the three coating conductivities cancel the remaining diagonal terms \(\mathbb N_{nn}^{(2)}\), \(1\leq n\leq3\).
The nonzero enclosure \eqref{eq:certified-degree-eight} shows that cancellation stops at the first total-degree stratum beyond the certified range \(m+n\leq7\).
Thus the example is a finite-order near-neutral structure rather than perfectly invisible; after scaling, Theorem~\ref{thm:arbitrary-H} gives the corresponding \(O(\varepsilon^8)\) exterior estimate.

\begin{figure}[t]
\centering
\begin{tikzpicture}[scale=3.45]
  \newcommand{\levelcurve}[5]{%
    \draw[#1,thick,fill=#2]
      plot[smooth cycle,samples=241,domain=0:360]
      ({#3*cos(\x)+#4*cos(7*\x)+#5*sin(15*\x)},
       {#3*sin(\x)-#4*sin(7*\x)+#5*cos(15*\x)});
  }
  \levelcurve{outerblue}{outerblue!18}{1.00}{0.00015}{0.00000005}
  \levelcurve{middleblue}{middleblue!24}{0.82}{0.000601720}{0.000000981}
  \levelcurve{innergold}{innergold!28}{0.66}{0.002749690}{0.000025457}
  \levelcurve{coregray}{coregray!24}{0.50}{0.0192}{0.0016384}
  \draw[->,thin] (1.08,0.50)--(0.79,0.36)
    node[pos=0,anchor=west,align=left,font=\footnotesize]
    {$\sigma_{\rm out}=0.85030$};
  \draw[->,thin] (1.08,0.18)--(0.66,0.12)
    node[pos=0,anchor=west,align=left,font=\footnotesize]
    {$\sigma_{\rm mid}=2.22976$};
  \draw[->,thin] (1.08,-0.16)--(0.52,-0.12)
    node[pos=0,anchor=west,align=left,font=\footnotesize]
    {$\sigma_{\rm in}=0.21737$};
  \draw[->,thin] (1.08,-0.49)--(0.28,-0.25)
    node[pos=0,anchor=west,align=left,font=\footnotesize]
    {core $\sigma_0=4$};
  \node[font=\small,align=center] at (0,-1.19)
    {common $C_8$ level curves; the small second Laurent harmonic breaks every reflection symmetry};
\end{tikzpicture}
\caption{The certified chiral three-coating geometry.
The displayed conductivity decimals indicate the location of the rigorously enclosed log-conductivity root.
The global injectivity of every interface follows from \eqref{eq:numerical-univalence}.}
\label{fig:chiral-geometry}
\end{figure}

\subsubsection{Validated Faber--Grunsky and Krawczyk certificate}

The validation concerns the inner-to-outer log-conductivity center \eqref{eq:certified-eta-center} and the exact rational box \eqref{eq:certified-eta-box}.

The validation evaluates the semi-infinite Faber--Grunsky representation of the reduced map \(\mathcal G_{a^\star}^{4}\) in \eqref{eq:symfree-reduced-material-map}, using the physical multilayer response formula of \cite[Theorem~5.2]{choi2023geometric}.
For the present map, \(\mathsf F_n(z)=z^n\) for \(1\leq n\leq7\); hence its first four diagonal Faber responses are exactly the corresponding normalized second-type complex CGPT responses \(\mathbb N_{nn}^{(2)}/(2\pi n)\) in the ordinary harmonic basis.

At the prescribed value \(a=a^\star\neq0\), all material transfer factors in the Faber--Grunsky representation are differentiated analytically with respect to the log-conductivities.
Thus no finite-difference derivative enters the response, Jacobian, determinant, or Krawczyk certificate.

The production root certificate uses the truncation order \(N_{\rm tr}=512\), \(512\)-bit Arb/Acb arithmetic, and the three \(64\times64\) invariant residue blocks
\[
 (1,7),\qquad(2,6),\qquad(3,5).
\]
The certificate for total degree eight uses the additional self-paired \(64\times64\) residue block \((4,4)\) on the same log-conductivity box.
All finite Grunsky coefficients, matrix products, derivatives with respect to the log-conductivities, and linear solves are enclosed by outward-rounded balls.
The implementation uses \texttt{python-flint 0.8.0} with \texttt{FLINT 3.3.1} and one computational thread.
Analytic estimates for the omitted modes give the uniform response and Jacobian tail radii
\begin{align}
 \max_{1\leq n\leq3}\varepsilon_{0,n}&<1.577\times10^{-20},
 &
 \max_{\substack{1\leq n\leq3\\1\leq p\leq3}}\varepsilon_{1,np}
 &<1.712\times10^{-18}.
 \label{eq:certificate-tail-summary}
\end{align}
For the additional order-four response, the same estimates give
\begin{equation}
 \varepsilon_{0,4}<7.855\times10^{-23}.
 \label{eq:certificate-degree-eight-tail}
\end{equation}
These tail radii are added componentwise to the finite outward-rounded matrix balls and therefore enclose the exact infinite-dimensional response and its log-conductivity Jacobian on \(B_{\rm cert}\).

Let \(Y\) be the fixed exact \(256\)-bit dyadic preconditioner recorded in the machine-readable certificate.
Its determinant is enclosed away from zero and has midpoint
\[
 54.8478405570814036\ldots.
\]
Here the bracketed response and Jacobian intervals denote these enlarged \(N_{\rm tr}=512\) Arb/Acb enclosures.
In particular, with rows indexed by the response order \(n\) and columns by the inner-to-outer log-conductivities, define the displayed center matrix
\begin{equation}
 J_\star:=
 \begin{pmatrix}
  0.296487546388217&0.316278187638896&0.332968736844520\\
  0.096535915282060&0.336467076198360&0.559472575130880\\
  0.032897665930071&0.250094473122290&0.715663833751530
 \end{pmatrix}.
 \label{eq:certified-noncircular-log-jacobian}
\end{equation}
Then the componentwise interval enclosure is
\begin{equation}
 D_\eta\mathcal G_{a^\star}^{4}(B_{\rm cert})
 \subset
 J_\star+[-3\times10^{-14},3\times10^{-14}]^{3\times3}.
 \label{eq:certified-noncircular-log-jacobian-box}
\end{equation}
The full-precision componentwise matrix enclosure stored in the certificate is contained in the coarser displayed box \eqref{eq:certified-noncircular-log-jacobian-box}.
Direct outward-rounded evaluation of the determinant of that stored full-precision enclosure gives \eqref{eq:certified-noncircular-log-determinant}.

For completeness, the change from \(\boldsymbol\tau=(\tau_3,\tau_2,\tau_1)\) to \(\eta=(\log\sigma_1,\log\sigma_2,\log\sigma_3)\) has the explicit Jacobian
\begin{equation}
 D_{\boldsymbol\tau}\Theta=
 \begin{pmatrix}
  \dfrac{2}{1-\tau_3^2}&
  \dfrac{2}{1-\tau_2^2}&
  \dfrac{2}{1-\tau_1^2}\\[1.1ex]
  \dfrac{2}{1-\tau_3^2}&
  \dfrac{2}{1-\tau_2^2}&0\\[1.1ex]
  \dfrac{2}{1-\tau_3^2}&0&0
 \end{pmatrix},
 \qquad
 \det D_{\boldsymbol\tau}\Theta
 =-\prod_{j=1}^{3}\frac{2}{1-\tau_j^2}<0.
 \label{eq:contrast-to-log-jacobian}
\end{equation}
Thus the final contrast-coordinate matrix and determinant are
\begin{align}
 D_{\boldsymbol\tau}\mathcal M_{a^\star}^{4}
 &=
 \operatorname{diag}(\pi,2\pi,3\pi)\,
 D_\eta\mathcal G_{a^\star}^{4}\,
 D_{\boldsymbol\tau}\Theta,
 \label{eq:noncircular-final-contrast-jacobian}\\
 \det D_{\boldsymbol\tau}\mathcal M_{a^\star}^{4}
 &=
 6\pi^3\,
 \det D_\eta\mathcal G_{a^\star}^{4}\,
 \det D_{\boldsymbol\tau}\Theta.
 \label{eq:noncircular-final-contrast-determinant}
\end{align}
The certificate evaluates both the matrix in \eqref{eq:noncircular-final-contrast-jacobian} directly and the factorization in \eqref{eq:noncircular-final-contrast-determinant}; the resulting determinant intervals overlap and are strictly negative, yielding \eqref{eq:certified-noncircular-contrast-degree}.
These intervals satisfy
\begin{equation}
 \bar\eta-Y[\mathcal G_{a^\star}^{4}(\bar\eta)]
 +(I-Y[D_\eta\mathcal G_{a^\star}^{4}(B_{\rm cert})])
 (B_{\rm cert}-\bar\eta)
 \subset\operatorname{int}B_{\rm cert}.
 \label{eq:certificate-krawczyk}
\end{equation}
The three certified distances from the Krawczyk image to the boundary of \(B_{\rm cert}\) are bounded below by
\begin{equation}
 (9.9989444,\;9.9992948,\;9.9997967)\times10^{-16}.
 \label{eq:certificate-margins}
\end{equation}

\begin{proof}[Proof of Theorem~\ref{thm:certified-C8}]
Because \(\mathsf F_n(z)=z^n\) for \(1\leq n\leq7\), the first three components of the semi-infinite response are precisely the normalized second-type complex CGPT responses in the ordinary harmonic basis that define \(\mathcal G_{a^\star}^{4}\).
The tail radii in \eqref{eq:certificate-tail-summary} make the finite response and Jacobian balls rigorous enclosures of that exact map.
For a continuously differentiable map on a box, the standard Krawczyk criterion states that a strict inclusion of its Krawczyk image in the interior of the box gives exactly one zero there \cite{krawczyk1969newton,moore2009introduction,rump2010verification}.
Thus \eqref{eq:certificate-krawczyk} gives existence and uniqueness in \(B_{\rm cert}\).
The determinant enclosure \eqref{eq:certified-noncircular-log-determinant} shows that the unique zero is regular and has local index \(+1\).
Since the strict Krawczyk inclusion places it in \(\operatorname{int}B_{\rm cert}\) and excludes every other zero in \(B_{\rm cert}\), the regular-value formula for Brouwer degree proves \eqref{eq:certified-noncircular-local-degree}.
The coordinate factorization \eqref{eq:noncircular-final-contrast-determinant} and the negative orientation in \eqref{eq:contrast-to-log-jacobian} give the contrast-coordinate determinant enclosure and degree in \eqref{eq:certified-noncircular-contrast-degree}.
Since \(8>2\cdot3\), Proposition~\ref{prop:selection} also shows that, throughout \(B_{\rm cert}\), a zero of \(\mathcal G_{a^\star}^{4}\) is equivalent to cancellation of the complete \(3\)-CGPT block.
Finally, Proposition~\ref{prop:selection} makes every first-type tensor with \(m+n\leq7\) and every off-diagonal second-type tensor with \(m+n\leq7\) vanish.
The only remaining second-type cases are \(m=n\leq3\), which are exactly the three certified equations.
For \(n=4\), the same identity \(\mathsf F_4(z)=z^4\) identifies the semi-infinite response with \(\mathbb N_{44}^{(2)}/(8\pi)\), while \eqref{eq:certificate-degree-eight-tail} encloses the exact response uniformly on \(B_{\rm cert}\).
Outward-rounded evaluation gives \eqref{eq:certified-degree-eight}, whose upper endpoint is negative.
The unique exact root in \(B_{\rm cert}\) therefore has \(\mathbb N_{44}^{(2)}\neq0\).
Finally, exponentiating the three coordinate intervals of \(B_{\rm cert}\) with outward rounding gives the conductivity enclosures in Table~\ref{tab:certified-conductivities}.
\end{proof}

\begin{table}[H]
\centering
\caption{Outward decimal enclosures of the unique certified coating conductivities, ordered from the core toward the exterior.
The full ball endpoints are stored in the machine-readable certificate.}
\label{tab:certified-conductivities}
\begin{tabular}{@{}ccc@{}}
\toprule
coating & lower bound & upper bound\\
\midrule
\(\sigma_1\)
& \(0.217369105530701687\)
& \(0.217369105530702123\)\\
\(\sigma_2\)
& \(2.229762879818031155\)
& \(2.229762879818035615\)\\
\(\sigma_3\)
& \(0.850299673621820147\)
& \(0.850299673621821848\)\\
\bottomrule
\end{tabular}
\end{table}

The supplementary computational archive contains the exact decimal inputs, the exact dyadic preconditioner, the full outward-rounded balls, independent precision runs, and the verification procedure used for the preceding certificate.

\subsubsection{Quadratic prediction from the radial reference root}

\begin{remark}\label{rem:quadratic-coefficient-reading}
The certified nondegeneracy in Proposition~\ref{prop:symfree-shape-certificate}, Lemma~\ref{lem:symfree-shape-analyticity}, and the implicit-function theorem give a local analytic branch \(a\mapsto\eta(a)\) through \(\eta(0)=\eta_{\rm rad}^*\).
Set \(\sigma(a):=(e^{\eta_j(a)})_{j=1}^3\) and \(\bar\sigma:=\sigma(0)\), so that
\[
 \bar\sigma\approx
 (0.2165224661251919,\,
  2.231016403756756,\,
  0.8502219742037971).
\]
With \(\gamma=r_0=1/2\), put \(\widehat a_\ell:=\gamma^{-8\ell}a_\ell\) for \(\ell=1,2\).
Let \(e_\ell\) be the \(\ell\)-th coordinate vector in \(\mathbb C^2\), and define the exact coefficient matrix by
\[
 (C_\sigma)_{j\ell}:=
 \left.\frac12\frac{\mathrm{d}^2}{\mathrm{d}t^2}
 \sigma_j(\gamma^{8\ell}t e_\ell)\right|_{t=0},
 \qquad j=1,2,3,\quad \ell=1,2,\quad t\in\mathbb R.
\]
Rotation invariance gives the expansion
\begin{equation}
 \sigma(a)=\bar\sigma+
 C_\sigma
 \begin{pmatrix}|\widehat a_1|^2\\|\widehat a_2|^2\end{pmatrix}
 +O(\lVert\widehat a\rVert^3).
 \label{eq:certified-physical-quadratic-expansion}
\end{equation}
Numerically,
\[
 C_\sigma\approx
 \begin{pmatrix}
  0.5590159854459432& 1.1626320762841366\\
 -0.8493244551757674&-1.3365036556017672\\
  0.0532708102338975& 0.0760586836617055
 \end{pmatrix}.
\]
For the fixed chiral map in \eqref{eq:numerical-map},
\[
 \widehat a_1=0.0384,
 \qquad
 \widehat a_2=0.0032768\,i.
\]
The corresponding quadratic prediction in log-conductivity coordinates is
\begin{equation}
 \begin{aligned}
 \Delta\eta_{\rm quad}
 &:=\operatorname{diag}(\bar\sigma)^{-1}C_\sigma
 \begin{pmatrix}|\widehat a_1|^2\\|\widehat a_2|^2\end{pmatrix}\\
 &\approx(
  0.00386466260677840,\,
 -0.000567781781084374,\,
  0.000093349363159051)^{\mathsf T}.
 \end{aligned}
 \label{eq:target-quadratic-prediction}
\end{equation}
Let $\eta_{\rm chiral}^*$ denote the unique exact root in $B_{\rm cert}$ from Theorem~\ref{thm:certified-C8}.
Comparison of the two independently certified root boxes gives
\begin{equation}
 \eta_{\rm chiral}^*-\eta_{\rm rad}^*
 -\Delta\eta_{\rm quad}
 \in
 \begin{pmatrix}
 [3.7881417270448,3.7881417272449]\times10^{-5}\\
 [5.761608359950,5.761608361951]\times10^{-6}\\
 [-1.966324105738,-1.966324103737]\times10^{-6}
 \end{pmatrix}.
 \label{eq:target-quadratic-comparison}
\end{equation}
Componentwise, the second-order prediction differs from the certified material displacement by at most $2.2\%$.
This is an a posteriori comparison of two certified roots, not a proof that the two roots lie on one validated continuation branch.

For the two weights $1$ and $2$, rotation and reflection invariance show that the only possible cubic invariant is a multiple of $\operatorname{Re}(\widehat a_1^2\overline{\widehat a_2})$.
It vanishes for the phase choice above.
Hence the Taylor expansion along the corresponding radial parameter ray has no cubic term, although no numerical fourth-order remainder bound is asserted here.
\end{remark}

\subsubsection{Independent Nystr\"om cross-check in the physical domain}

Reindex the four interfaces from inside to outside as $\Gamma_0,\ldots,\Gamma_3$, set $\sigma_4=1$, and orient $\Gamma_j$ by the normal pointing from the inner phase $\sigma_j$ to the outer phase $\sigma_{j+1}$.
We use the convention
\begin{equation}
 \mathcal S_\Gamma[\varphi](x)
 =\frac{1}{2\pi}\int_\Gamma\log|x-y|\varphi(y)\,\mathrm{d} s_y,
 \qquad
 \left.\partial_\nu\mathcal S_\Gamma[\varphi]\right|_{\pm}
 =\left(\pm\frac12I+\mathcal K_\Gamma^*\right)\varphi,
 \label{eq:jump-convention}
\end{equation}
where $+$ is the exterior trace and $\nu$ is the outward normal.
We solve directly on the four physical curves using
\begin{equation}
 u=H+\sum_{j=0}^{3}\mathcal S_{\Gamma_j}[\varphi_j].
 \label{eq:single-layer-ansatz}
\end{equation}
The densities satisfy the block system
\begin{equation}
 \left(\lambda_j I-\mathcal K_{\Gamma_j}^*\right)\varphi_j
 -\sum_{\ell\neq j}
  \partial_{\nu_j}\mathcal S_{\Gamma_\ell}[\varphi_\ell]
 =\partial_{\nu_j}H,
 \qquad
 \lambda_j=\frac{\sigma_j+\sigma_{j+1}}
 {2(\sigma_j-\sigma_{j+1})}.
 \label{eq:BIE-system}
\end{equation}
For each incident field \(P_n^\beta\), \(1\leq n\leq8\) and \(\beta\in\{c,s\}\), we compute the complex moments
\begin{equation}
 Q_{mn}^{\beta}=\sum_{j=0}^{3}
 \int_{\Gamma_j}z^m\varphi_{j,n}^{\beta}\,\mathrm{d} s
 \label{eq:computed-moment}
\end{equation}
which, in the present normalization, satisfy
\begin{equation}
 Q_{mn}^{c}=M_{mn}^{cc}+iM_{mn}^{sc},
 \qquad
 Q_{mn}^{s}=M_{mn}^{cs}+iM_{mn}^{ss}.
 \label{eq:Q-to-M}
\end{equation}
Thus the two loadings and the real and imaginary observation components check all four real CGPT families.
A periodic Nystr\"om discretization is used independently on the four physical interfaces.
This solver does not use the Faber--Grunsky linear system or its projection-tail bounds.

For comparison with the reduced response, define
\begin{equation}
 \boldsymbol R^{\rm raw}
 :=(\operatorname{Re} Q_{11}^{c},\operatorname{Re} Q_{22}^{c},\operatorname{Re} Q_{33}^{c}).
 \label{eq:raw-response}
\end{equation}
The exact selection rules give
\begin{equation}
 \mathcal G_n=\frac{R_n^{\rm raw}}{\pi n}.
 \label{eq:raw-to-reduced}
\end{equation}

At the high-precision center \eqref{eq:certified-eta-center}, the analytic Faber--Grunsky Jacobian in raw units and outer-to-inner material order is
\begin{equation}
 J_{\rm raw}\approx
 \begin{pmatrix}
 1.0460521375458165&0.9936172307770452&0.9314430976140895\\
 3.5152698638322530&2.1140849895192020&0.6065530445154076\\
 6.7449727276611290&2.3570848783931897&0.3100531968184864
 \end{pmatrix}.
 \label{eq:crosscheck-raw-J}
\end{equation}
It has
\begin{equation}
 \det J_{\rm raw}\approx-3.39188668491305,
 \qquad
 s(J_{\rm raw})\approx(8.34708,1.22748,0.33105).
 \label{eq:crosscheck-raw-J-data}
\end{equation}
The independent analytic Nystr\"om Jacobian at \(160\) nodes per interface agrees with this matrix to \(1.10\times10^{-14}\) entrywise.
Its maximum spurious imaginary component is below \(5.3\times10^{-17}\).
At the same resolution the maximum low-order physical moment is \(3.11\times10^{-15}\).
These comparisons are diagnostics without interval certification and are not used in the proof of Theorem~\ref{thm:certified-C8}.

\begin{table}[t]
\centering
\caption{Independent Nystr\"om resolution check without interval arithmetic at the rounded center \eqref{eq:noncircular-root}.
The second column is the maximum magnitude over all complex moments encoding the four real families with \(1\leq m,n\leq3\).
The third checks every pair with \(m+n\leq7\).
Once roundoff is reached, the residual is not expected to decrease monotonically.}
\label{tab:convergence}
\begin{tabular}{@{}rrr@{}}
\toprule
nodes/interface & $\max_{m,n\leq3,\,\beta}|Q_{mn}^{\beta}|$
& $\max_{m+n\leq7,\,\beta}|Q_{mn}^{\beta}|$\\
\midrule
 96  & $2.7631\times10^{-9}$  & $2.7631\times10^{-9}$\\
128  & $4.3572\times10^{-12}$ & $4.3572\times10^{-12}$\\
160, 192, 256 & $<10^{-14}$ & $<10^{-14}$\\
\bottomrule
\end{tabular}
\end{table}

\subsubsection{A certified response beyond the cancellation range}

The interval enclosure \eqref{eq:certified-degree-eight} proves that \(\mathbb N_{44}^{(2)}\), at total degree \(m+n=8\), is nonzero at the certified log-conductivity root.
This is the first total degree beyond the certified range \(m+n\leq7\), although lower diagonal responses of total degrees two, four, and six are also permitted by symmetry and vanish here because of the material tuning.
Table~\ref{tab:first-shell} gives an independent Nystr\"om evaluation in the physical domain of several entries at total degree eight.

\begin{table}[t]
\centering
\caption{Largest of the cosine- and sine-loading complex moments at the first total degree beyond the certified range, \(m+n=8\), evaluated with \(256\) nodes per interface.
Reciprocal pairs have the same magnitude up to discretization error.}
\label{tab:first-shell}
\begin{tabular}{@{}ccr@{}}
\toprule
$m$ & $n$ & $\max_{\beta\in\{c,s\}}|Q_{mn}^{\beta}|$\\
\midrule
1&7&$5.88166235698\times10^{-3}$\\
2&6&$1.03752894259\times10^{-2}$\\
3&5&$1.31131931535\times10^{-2}$\\
4&4&$3.56397326838\times10^{-1}$\\
\bottomrule
\end{tabular}
\end{table}

Among the total-degree-eight entries, only \(\mathbb N_{44}^{(2)}\) is interval-certified; the remaining entries in Table~\ref{tab:first-shell} are diagnostics without interval certification.

\subsection{A homothetic numerical example}
\label{sec:large-amplitude-example}

We finally illustrate the fixed-geometry result on a geometry that is visibly not close to circular.
Let \(K=3\), \(Q=8\), and define four homothetic interfaces by
\begin{equation}
 z_j(\theta)=\rho_j\left(1+\frac{7}{25}\cos(8\theta)\right)
 \mathrm{e}^{i\theta},
 \qquad
 (\rho_0,\rho_1,\rho_2,\rho_3)
 =\left(\frac12,\frac{33}{50},\frac{41}{50},1\right).
 \label{eq:large-amplitude-geometry}
\end{equation}
The outer radial function ranges from \(18/25=0.72\) to \(32/25=1.28\), so its maximum-to-minimum radius ratio is \(16/9\).
The interfaces are smooth, strictly star-shaped, exactly \(C_8\)-symmetric, and strictly nested by homothety.
In particular, Corollary~\ref{cor:homothetic} applies near the homogeneous reference point without any smallness assumption on this geometry.

\begin{figure}[H]
\centering
\begin{tikzpicture}[scale=2.35]
  \newcommand{\starlevel}[3]{%
    \draw[#1,thick,fill=#2]
      plot[smooth cycle,samples=401,domain=0:360]
      ({#3*(1+0.28*cos(8*\x))*cos(\x)},
       {#3*(1+0.28*cos(8*\x))*sin(\x)});
  }
  \starlevel{outerblue}{outerblue!18}{1.00}
  \starlevel{middleblue}{middleblue!24}{0.82}
  \starlevel{innergold}{innergold!28}{0.66}
  \starlevel{coregray}{coregray!24}{0.50}
  \node[font=\footnotesize,fill=white,inner sep=1.5pt]
    at (0,0) {core \(6/5\)};
\end{tikzpicture}
\caption{The large-amplitude homothetic \(C_8\) geometry \eqref{eq:large-amplitude-geometry}.
The outer boundary has \(28\%\) radial oscillation and radius ratio \(16/9\), placing the example well outside a near-circular perturbative regime.}
\label{fig:large-amplitude-geometry}
\end{figure}
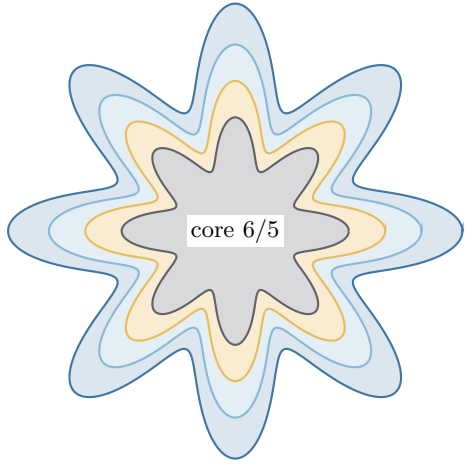

Take core conductivity \(6/5\) and background conductivity \(1\).
A Newton solve of the periodic Nystr\"om system in the physical domain, using the three coating log-conductivities from inner to outer as variables, gives the 576-node discrete root
\begin{equation}
 \boxed{(\sigma_1,\sigma_2,\sigma_3)
 \approx(0.4956062210827645,
          1.600311466225608,
          0.9138845615900373).}
 \label{eq:large-amplitude-materials}
\end{equation}
At this point the Jacobian in log-conductivity variables has determinant approximately
\[
 6.5628259\times10^{-3}.
\]
The boundary-integral system has two-norm condition number approximately \(30.8978\).
Here \(Q_{mn}^{\beta}\) is the physical complex moment in \eqref{eq:computed-moment}; both cosine and sine loadings are included.
The reproducible script and JSON output are included in the supplementary \texttt{gate0} directory.

\begin{table}[H]
\centering
\caption{Resolution check for the large-amplitude homothetic example at the fixed 576-node conductivity vector \eqref{eq:large-amplitude-materials}.
The 576-node solve grid is omitted, so every displayed row is an independent resolution check.}
\label{tab:large-amplitude-convergence}
\begin{tabular}{@{}rr@{}}
\toprule
nodes/interface
& \(\displaystyle\max_{m+n\leq7,\,\beta\in\{c,s\}}
   |Q_{mn}^{\beta}|\)\\
\midrule
448 & \(1.485\times10^{-8}\)\\
512 & \(1.620\times10^{-9}\)\\
640 & \(3.665\times10^{-11}\)\\
\bottomrule
\end{tabular}
\end{table}

Exact \(C_8\) symmetry reduces the complete \(3\)-CGPT block to the three diagonal equations used in the Newton solve.
The decreasing residuals in Table~\ref{tab:large-amplitude-convergence} provide a direct numerical check in the physical domain of the complete block cancellation and, by the selection rules, of every CGPT of total degree at most seven.
This example has no interval certification: the reported decimals define a discrete root, and the calculation does not rigorously enclose the exact infinite-dimensional root at core conductivity \(6/5\).

\section*{Acknowledgments}
H.~Liu has been partially supported by the Hong Kong RGC General Research Funds (projects 11303125, 11304224 and 11311122).
The authors would like to acknowledge the use of generative AI tools for language polishing, reference formatting, and logical structure optimization during the preparation of this manuscript.

\begingroup \hfuzz=2pt
\AtBeginEnvironment{thebibliography}{\small}
\bibliographystyle{siam-nodash}
\bibliography{Conformal_Multilayer_High_Order_GPT}
\endgroup

\end{document}